\documentclass[aihp]{imsart}

\RequirePackage{amsthm,amsmath,amsfonts,amssymb}
\RequirePackage[numbers]{natbib}
\RequirePackage[colorlinks,citecolor=blue,urlcolor=blue]{hyperref}
\RequirePackage{graphicx}

\RequirePackage{microtype} 

\RequirePackage{mathrsfs,enumerate}
\startlocaldefs
\numberwithin{equation}{section}
\theoremstyle{plain}
\newtheorem{theorem}{Theorem}[section]
\newtheorem{proposition}[theorem]{Proposition}
\newtheorem*{proposition*}{Proposition}
\newtheorem*{theorem*}{Theorem}
\newtheorem{lemma}[theorem]{Lemma}
\newtheorem*{lemma*}{Lemma}

\theoremstyle{remark}
\newtheorem{definition}[theorem]{Definition}

\newtheorem{remark}[theorem]{Remark}
\newtheorem{question}{Question}
\newcommand*{\dif}{\ensuremath{\mathop{}\!\mathrm{d}}}

\def\E{\mathbb{E}}
\def\P{\mathbb{P}}

\endlocaldefs
 
\begin{document}

\begin{frontmatter}


\title{Double jump  in the maximum of  two-type reducible branching Brownian motion}
\runtitle{Double jump  in the maximum of  two-type reducible BBM}

\begin{aug}
\author[A]{\inits{M.}\fnms{Heng}~\snm{Ma}\ead[label=e1]{hengmamath@gmail.com} \orcid{0000-0002-1401-7677}}
\author[B]{\inits{R.}\fnms{Yan-Xia}~\snm{Ren}\ead[label=e2]{yxren@math.pku.edu.cn}}
\address[A]{School of Mathematical Sciences, Peking
University, Beijing, P.R. China.\printead[presep={,\ }]{e1}}
\address[B]{LMAM School of Mathematical Sciences \& Center for
Statistical Science, Peking
University, Beijing, P.R. China.\printead[presep={,\ }]{e2}}
\end{aug}

\begin{abstract}
    Consider a two-type reducible branching Brownian motion in which particles' diffusion coefficients and  branching rates  are influenced by their types.
    Here reducible means that type $1$  particles can produce particles of  type $1$ and type $2$, but
    type $2$ particles can only produce particles of type $2$.
    The maximum of this process is determined by two parameters: the ratio of the diffusion coefficients and the ratio of the branching rates for particles of different types.
    Belloum and Mallein [Electron. J. Probab. \textbf{26} (2021), no. 61]
    identified three phases of the maximum and the extremal process, corresponding to three regions in the parameter space.

    We investigate how the extremal process behaves
    asymptotically when the parameters lie on the boundaries between these regions.
    An interesting phenomenon is that  a double jump occurs in the maximum when the parameters cross the boundary of  the so called  anomalous spreading region,
    while only single jump occurs when the parameters cross the boundary between the remaining two regions.
\end{abstract}

\begin{abstract}[language=french]
 Considérons un mouvement brownien branchant à deux types, réductible, dans lequel les coefficients de diffusion et les taux de branchement des particules dépendent de leur type. Ici, “réductible” signifie que les particules de type $1$ peuvent produire des particules de type $1$ et de type $2$, alors que les particules de type $2$ ne peuvent produire que des particules de type $2$. Le maximum de ce processus est déterminé par deux paramètres : le rapport des coefficients de diffusion et le rapport des taux de branchement pour les particules de différents types. Belloum et Mallein [Electron. J. Probab. \textbf{26} (2021), no. 61] ont identifié trois phases du maximum et du processus extrémal, correspondant à trois régions dans l’espace des paramètres.

 Nous étudions le comportement asymptotique du processus extrémal lorsque les paramètres se situent sur les frontières entre ces régions.  Un phénomène intéressant est qu’un double saut se produit dans le maximum lorsque les paramètres franchissent la frontière de la région dite de diffusion anormale, tandis qu’un simple saut se produit lorsque les paramètres franchissent la frontière entre les deux autres régions.
\end{abstract}



\end{frontmatter}


\section{Introduction}
Over the last few years, many people have studied the extreme values of the so-called \textit{log-correlated}  fields, which form a large universality class for the distributions of extreme values of correlated stochastic processes.
One of the simple models in this  class is  the
branching Brownian motion (BBM), which can be described as follows.
Initially we have a particle moving as a standard Brownian motion.
At rate $1$ it splits into two particles.
These particles behave independently of each other, continue to move and split,  subject to the same rule.

Due to the presence of a tree structure and Brownian trajectories, many precise results for the extreme values of BBM were obtained.
Bramson \cite{Bramson78,Bramson83} gave
the correct order of the maximum for BBM and the convergence in law of the centered maximum.  Lalley and Sellke \cite{LS87}
obtained a probabilistic representation of the limit distribution.
A remarkable contribution for the extreme value statistics is the construction of the limiting \textit{extremal process} for BBM, obtained in Arguin, Bovier and Kistler  \cite{ABK13}, as well as in   A\"{i}d\'{e}kon, Berestycki, Brunet and Shi \cite{ABBS13}.
With motivation from disorder system \cite{BI04a,BI04b},  serval works studied the extreme value for variable speed BBMs
see, for examples,  \cite{BH14,BH15,FZ12,MZ16}.
Many results on BBM were extended to branching random walks  \cite{Aidekon13,Madaule17},  and other log-correlated fields, such as $2$-dimensional discrete Gaussian free fields
 \cite{BL16,BL18,BDZ16,BZ12},
 log-correlated Gaussian fields on
$d$-dimensional  boxes \cite{Acosta14,DRZ17,Madaule15},
and high-values of the Riemann zeta-function \cite{ABBRS19,ABH17,ADH21}. For recent reviews  see, e.g. \cite{Arguin16,BK22}.

This article  concerns the extreme values of a \textit{multi-type} branching Brownian motion,
in which particles of different types have
different branching mechanisms and diffusion coefficients.
Just like in the case of Markov chains,
we say a multi-type branching Brownian motion is  \textit{reducible}
if particles of some type $i$ can not have a descendant
of some other type $j$; and otherwise call it \textit{irreducible}.
For irreducible multi-type BBMs (with a common diffusion coefficient for all types),
the spreading speed was given by Ren and Yang \cite{RYS14}, and recently Hou, Ren and Song \cite{HRS23}
obtained the precise order of the maximum and  the limiting extremal process.
For reducible multi-type branching random walks (BRWs),
Biggins \cite{Biggins10,Biggins12} studied the leading coefficient of the maximum and  found that in some cases
the processes exhibit the so-called \textit{anomalous spreading} phenomenon.
More precisely, the leading coefficient of the maximum for a multi-type BRW is larger than that of a BRW consisting only of particles of a single type.
Holzer \cite{Holzer14,Holzer16} extended  results of Biggins to the  BBM setting, by  studying the associated system of F-KPP equations.

Belloum and Mallein  \cite{BM21}  studied
the  extremal process of  a two-type reducible BBM and obtained in particular the precise order of the maximum.
One can construct this process
by first running a BBM with branching rate $\beta$ and diffusion coefficient $\sigma^2$ (type $1$ BBM), and then adding standard BBMs (type $2$ BBMs) along each $(\beta,\sigma^2)$ BBM path according to a Poisson process.
There are three distinct regimes: type $1$/type $2$ domination, and anomalous spreading, corresponding to the parameters  $(\beta, \sigma^2)$ belonging to $\mathscr{C}_{I}, \mathscr{C}_{II},$ and $\mathscr{C}_{III}$, respectively (see Figure \ref{fig-phase}).
However when the parameters $(\beta,\sigma^2)$ are on the
boundaries between these three sets, the precise order of the maximum and the behavior of  extremal process are not clear, except  the common intersection of these boundaries which
was studied by Belloum \cite{Belloum22}.

In this article, we study the asymptotic behavior of extremal particles in the  two-type reducible BBM above when the
parameters $(\beta, \sigma^2)$ lie  on the boundaries between   $\mathscr{C}_{I}, \mathscr{C}_{II},\mathscr{C}_{III}$.
We show that the extremal process converges in law towards a decorated Poisson point process and  give the precise order of the maximum.
Combined with the main results in
 \cite{Belloum22, BM21},
the phase diagram of the two-type reducible BBM  is now complete and clear.
As an interesting by-product,  a \textit{double jump} occurs  in the maximum of the two-type reducible BBM when
the parameters $(\beta,\sigma^2)$ cross the boundary of the anomalous spreading region $\mathscr{C}_{III}$,
and only single jump occurs when the parameters cross the boundary between $\mathscr{C}_{I}, \mathscr{C}_{II}$.

\subsection{Standard branching Brownian motion}
Let  $\{ ( \mathsf{X}_{u}(t), u \in \mathsf{N}_{t})_{t \geq 0}, \mathsf{P} \}$  be a standard BBM,
where $\mathsf{N}_{t}$ represents the set of all particles alive at time $t$,  $\mathsf{X}_u(t)$ denotes the position of individual $u \in \mathsf{N}_{t}$, and $\mathsf{P}$ is the corresponding law of the BBM.
Let $\mathsf{M}_t=\max_{u \in \mathsf{N}_{t}}\mathsf{X}_{u}(t)$ be  the maximal displacement among all the particles alive at time $t$.
Bramson \cite{Bramson78,Bramson83}
obtained an explicit asymptotic formula of $\mathsf{M}_{t}$:
If  let $m_{t}:= \sqrt{2} t - \frac{3}{2\sqrt{2}}\log t $, then $(\mathsf{M}_{t}-m_{t}:t>0)$ converges weakly, and the cumulative distribution function of the limit distribution is the unique (up to transition) travelling wave solution of a certain F-KPP equation.
Lalley and Sellke \cite{LS87} improved this result and they proved that $\mathsf{M}_{t}-m_{t}$ converges weakly to  a random shift of the Gumbel distribution. Specifically,
they   showed that for some constant $C_{\star}>0$,
\begin{equation}\label{eq-Gumbel}
  \lim_{t \to \infty}
  \mathsf{P}( \mathsf{M}_{t}-m_{t} \leq x )
  =  \mathsf{E}[ \exp\{-  C_{\star} \mathsf{Z}_{\infty} e^{-\sqrt{2}x}\} ],
\end{equation}
where $\mathsf{Z}_\infty>0 $ is the a.s. limit of the  \textit{derivative martingale} $(\mathsf{Z}_{t})_{t > 0}$ defined by  $\mathsf{Z}_{t}=\sum\limits_{u \in \mathsf{N}_t}\left(\sqrt{2} t-\mathsf{X}_u(t)\right) e^{\sqrt{2} \mathsf{X}_u(t)-2 t} $.
The name derivative martingale comes from the fact
that $\mathsf{Z}_{t}= -\frac{\partial}{\partial \lambda}|_{\lambda = \sqrt{2}}\mathsf{W}_t(\lambda) $,
where $\mathsf{W}_t(\lambda):=\sum\limits_{u \in \mathsf{N}_t} e^{\lambda X_u(t)-\left(1+\frac{\lambda^2}{2}\right)t}$ is called the \textit{additive martingale} for BBM.
A\"{i}d\'{e}kon \cite{Aidekon13} proved the convergence in law of the centered maximum of branching random walks by probabilistic method.

After the maximum of the process was known,
many researches
focused on the full extreme value statistics for BBM, which can be encoded by the following point process, called \textit{extremal process}
\begin{equation*}
  \mathcal{E}_t:=\sum_{u \in \mathsf{N}_t} \delta_{\mathsf{X}_{u}(t)-m_{t}}.
\end{equation*}
It was proved independently by Arguin, Bovier, Kistler  \cite{ABK13}, and A\"{i}d\'{e}kon, Berestycki, Brunet, Shi \cite{ABBS13}   that   $\mathcal{E}_{t}$ converges in law to a random shifted \textit{decorated Poisson point process} (DPPP for short)   defined in the following paragraph. Bovier and Hartung \cite{BH17}
extended  the convergence of the extremal process by adding an extra dimension that encodes the localization of the particle in the underlying Galton-Watson tree. We refer to  \cite{CHL19} for  further results on the structure of the extremal process; and to \cite{BD11} for some conjectures, many of them still have not been rigorously proven  yet.

A DPPP $\mathcal{E}$ is determined by an intensity measure $\mu$ and a decoration process $\mathfrak{D}$, where $\mu$ is a (random) measure on $\mathbb{R}$ and $\mathfrak{D}$ is the law of a random point process on $\mathbb{R}$. Conditioned on $\mu$, sampling a Poisson point process $\sum_{i} \delta_{x_{i}} $ with intensity $\mu$, and  an independent family of i.i.d. point processes $\left( \sum_{j} \delta_{d^{i}_{j}} :i \geq 0\right)$ with law $\mathfrak{D}$, then the point measure $\mathcal{E} \sim \mathrm{DPPP}(\mu,\mathfrak{D})$ can be constructed as $
\mathcal{E}=\sum_{i,j}   \delta_{x_i+d_j^i}$.

Using this notation, the main result in
 \cite{ ABBS13} and  \cite{ABK13} can be stated as:
\begin{equation}
  \label{eq-converge-to-DPPP}
   \lim_{t \to \infty} \sum_{u \in \mathsf{N}_t} \delta_{\mathsf{X}_{u}(t)-m_{t}}    =\operatorname{DPPP}\left(\sqrt{2} C_{\star} \mathsf{Z}_{\infty} \mathrm{e}^{-\sqrt{2} x} \mathrm{~d} x, \mathfrak{D}^{\sqrt{2}}\right) \text{ in law in the vague topology}.
   \footnote{
   That is to say,
   for every continuous test function $\phi$ with compact support, $ \int \phi \dif \mathcal{E}_{t} $ converges weakly to $\int \phi \dif \mathcal{E}$. In this paper  when dealing with the weak convergence of point process, we always assume that we are using this setting.}
  \end{equation}
The decoration law $\mathfrak{D}^{\sqrt{2}}$ belongs to the family $\left(\mathfrak{D}^{\varrho}, \varrho \geq \sqrt{2}\right)$, defined as the limits  of the ``gap processes" conditioned on $\mathsf{M}_t \geq   \varrho t $:
\begin{equation}\label{eq-decoration-process}
\mathfrak{D}^{\varrho}(\cdot):=\lim _{t \rightarrow \infty}
\mathsf{ P}
 \left(\sum_{u \in \mathsf{N}_t} \delta_{\mathsf{X}_{u}(t)-\mathsf{M}_t } \in \cdot \mid \mathsf{M}_t \geq   \varrho t\right).
\end{equation}
The existence of this limit is established in \cite[Theorem 3.4]{ABK13} for $\rho=\sqrt{2}$ and in \cite[Proposition 7.5]{BH14}  for $\rho> \sqrt{2}$. See also  \cite{BBCM22} for an  alternative proof using the spine decomposition techniques. In \cite{BBCM22},   these point processes $(\mathfrak{D}^{\rho}:\rho \geq \sqrt{2})$ were used as decorations in the extremal processes of $2$-speed BBMs.

\subsection{Two-type reducible branching Brownian motion}

Now we give the definition of a two-type reducible branching Brownian motion, which is the model we are going to study in this paper.
The difference between our two-type reducible BBM and the standard BBM is that in our two-type BBM,
each particle now has a type and the branching and movement depend on the type.
Specifically, type $1$ particles move according to a Brownian motion with diffusion coefficient $\sigma^2$. They branch at rate $\beta$ into two children of type $1$ and give birth to particles of type $2$ at rate $\alpha$.
Type $2$ particles move according to a standard Brownian motion and branch at rate $1$ into $2$ children of type $2$, but can not give birth to offspring of type $1$.
We use $N_t$ to represent  all particles alive at time $t$, as well as  $N_t^1$ and $N_t^2$
for particles of type $1$ and type $2$ alive at time $t$ respectively.
For $u \in N_{t}$ and $s \leq t$, let $X_u(s)$ be the position of the ancestor at time $s$ of particle $u$.
So we write $\{(X_u(t), u \in N_{t})_{t\geq 0}, \mathbb{P}\}$  for a two-type reducible BBM and    $M_{t}:= \max\limits_{u \in N_{t}} X_{u}(t)$ for its maximum.

As mentioned at the beginning of the paper,  Biggins  \cite{Biggins10,Biggins12}  found that
anomalous spreading may occur for multi-type reducible BRWs.
Belloum and Mallein \cite{BM21} studied  more details on the precise order of maximum  and extremal process for this two-type BBM. Especially in the case when anomalous spreading occurs,
they showed that the extremal
process, formed by type $2$ particles at time $t$, converges towards a DPPP weakly.

To describe the limiting extremal process  in the form of \eqref{eq-converge-to-DPPP},
we  introduce the  additive and   derivative martingales of type $1$ particles.
Note that particles $\{ X_u(t): u \in N^{1}_{t} \}$ of type $1$  alone
have the same law as  the BBM   with branching rate $\beta$ and diffusion coefficient $\sigma^2$, which is denoted
by $(\mathsf{X}_{u}^{\beta,\sigma^2}(t):u \in \mathsf{N}_{t})_{t\geq 0}$ (here $\mathsf{N}_{t}$ is a slight abuse of notation as it depends also on $\beta$). We write $\mathsf{X}$ for the standard BBM  $\mathsf{X}^{1,1}$.
  Using the scaling property of Brownian motion, we have
\begin{equation}\label{eq-BBMscaling}
  \left( \mathsf{X}_{u}^{\beta,\sigma^2}(t) :u \in \mathsf{N}_{t} \right)\overset{law}{=} \left( \frac{\sigma}{\sqrt{\beta}} \mathsf{X}_{u}(\beta t) : u \in \mathsf{N}_{\beta t} \right).
\end{equation}
So, by the results for standard BBM,
 the corresponding derivative martingale of type $1$ individuals  and its a.s. limit are given by
\begin{equation}\label{eq-def-martingale-type1}
   Z^{(1)}_{t}  :=   \sum_{u \in N_t^1} \left(\sqrt{2 \beta \sigma^{2}} t-X_u(t)\right) e^{  \sqrt{2 \beta/\sigma^{2}} X_u(t)-2 \beta t}     \text{ and }  Z^{(1)}_{\infty} = \lim _{t \rightarrow \infty}
     Z^{(1)}_{t}>0 \text{ a.s.}
 \end{equation}
The corresponding additive martingales of type $1$  individuals and their a.s. limits are given by
\begin{equation}\label{eq-def-martingale-type1'}
   W^{(1)}_{t}(\lambda)  :=    \sum_{u \in N_t^1}  e^{ \lambda X_u(t)- (\frac{\lambda^2 \sigma^2}{2}+ \beta) t}  \   \text{and}  \ W^{(1)}_{\infty}(\lambda) = \lim _{t \rightarrow \infty}   W^{(1)}_{t}(\lambda)  \ \text{ a.s. for } \lambda \in \mathbb{R}.
 \end{equation}

Divide the parameter space
$\left(\beta, \sigma^2\right) \in \mathbb{R}_{+}^2$   into three regions (see Figure \ref{fig-phase}):
\begin{equation*}
\begin{aligned}
\mathscr{C}_I &
=\left\{\left(\beta, \sigma^2\right): \sigma^2>\frac{1}{\beta}1_{\{\beta \leq 1\}}+\frac{\beta}{2 \beta-1}1_{\{\beta>1\}} \right\} \,,\\
\mathscr{C}_{I I}
& =\left\{\left(\beta, \sigma^2\right): \sigma^2<\frac{1}{\beta}1_{\{\beta \leq 1\}}+(2-\beta)1_{\{\beta>1\}}\right\} \,,\\
\mathscr{C}_{I I I}
& =\left\{\left(\beta, \sigma^2\right): \sigma^2+\beta>2 \text { and } \sigma^2<\frac{\beta}{2 \beta-1}\right\} .
\end{aligned}
\end{equation*}
The main results in \cite{BM21} are as follows. Recall the constant $C_{\star}$ in \eqref{eq-Gumbel}, \eqref{eq-converge-to-DPPP},
and the decorations $(\mathfrak{D}^{\varrho})_{\rho \geq \sqrt{2}}$ in \eqref{eq-decoration-process}.  Let $v:=\sqrt{2 \beta \sigma^2}$ and $\theta:=\sqrt{ {2 \beta}/{\sigma^2}}$.

\begin{itemize}
  \item If $\left(\beta, \sigma^2\right) \in \mathscr{C}_I$, then  $M_{t}=v t - \frac{3}{2 \theta} \log t +O_{\mathbb{P}}(1)$.
  For some constant $C$ and
  some decoration law $\mathfrak{D}_{(I)}$ obtained  implicitly in \cite{BM21},
  we have
   \begin{equation*}
    \lim_{t \to \infty}\sum_{u \in N^{2}_{t}}\delta_{X_{u}(t)- v t + \frac{3}{2 \theta} \log t} = \mathrm{DPPP}\left(\theta C Z^{(1)}_{\infty} e^{-\theta x} \dif x, \mathfrak{D}_{(I)}\right) \text{ in law.}
   \end{equation*}
   \item  If $\left(\beta, \sigma^2\right) \in \mathscr{C}_{I I}$, then   $M_{t}=\sqrt{2}t- \frac{3}{2\sqrt{2}}\log t + O_{\mathbb{P}}(1)$. There is some random variable  $\bar{Z}_{\infty} \in (0,\infty)$ (see Lemma \ref{lem-bar-Z-infinity}) such that
   \begin{equation*}
    \lim_{t \to \infty} \sum_{u \in N^{2}_{t}}\delta_{X_{u}(t)- \sqrt{2 } t + \frac{3}{2\sqrt{2}} \log t}  =    \mathrm{DPPP}\left( \sqrt{2} C_{\star}\bar{Z}_{\infty} e^{-\sqrt{2} x} \dif x, \mathfrak{D}^{\sqrt{2}}\right)  \text{ in law,}
   \end{equation*}
   \item If $\left(\beta, \sigma^2\right) \in \mathscr{C}_{I I I}$, then $M_{t}=v^* t + O_{\mathbb{P}}(1)$, where $v^*=\frac{\beta-\sigma^2}{\sqrt{2(1-\sigma^2)(\beta-1)}} > \max (v,\sqrt{2})$. For $\theta^{*}= \sqrt{2 \frac{\beta-1}{1-\sigma^2}}  $ and  $C=\frac{\alpha C(\theta^{*}) }{2(\beta-1)}$ (where $C(\theta^{*})$ is defined in  Lemma \ref{thm-Laplace-BBM-order}, (ii)), we have
   \begin{equation*}
   \lim_{t \to \infty}  \sum_{u \in N^{2}_{t}}\delta_{X_{u}(t)- v^{*}t} =   \mathrm{DPPP}\left( \theta^{*} C W^{(1)}_{\infty} (\theta^{*} ) e^{-\theta^{*} x} \dif x, \mathfrak{D}^{\theta^{*}}\right)   \text{ in law. }
   \end{equation*}
  \end{itemize}
  Moreover, Belloum  \cite{Belloum22} showed that
  \begin{itemize}
    \item If $\left(\beta, \sigma^2\right) =(1,1)$, then  $M_{t}=\sqrt{2}t - \frac{1}{2\sqrt{2}}\log t + O_{\mathbb{P}}(1)$. The extremal process
    \begin{equation*}
      \lim_{t \to \infty}  \sum_{u \in N_{t}^{2}}\delta_{X_{u}(t)- \sqrt{2}t + \frac{1}{2\sqrt{2}}\log t } =  \mathrm{DPPP}\left( \sqrt{2} C_{\star} Z^{(1)}_{\infty} e^{-\sqrt{2} x} \dif x, \mathfrak{D}^{\sqrt{2}}\right)    \text{ in law. }
     \end{equation*}
  \end{itemize}

The above results were explained by Belloum and  Mallein \cite{BM21} as follows:
 If $\left(\beta, \sigma^2\right) \in \mathcal{C}_{I I I}$, the leading coefficient $v^{*}$ is larger than $\max \left(\sqrt{2},v \right)$, and the extremal process is given by a mixture of the long-time behavior of the processes of particles of type $1$ and $2$. In other words, an extremal particle has an ancestral line that has been of type $1$ for a positive proportion of the time, and also of type $2$ for a positive proportion of the time.
If $\left(\beta, \sigma^2\right) \in \mathcal{C}_{i}$ for $i=I,II$, the order of   $M_{t}$ is the same as a single BBM of particles of type $i=1,2$, and  the extremal process is dominated by the long-time behavior of particles of type $i=1,2$, respectively.

The aim of this article is to  obtain the asymptotic behavior of the maximum and the extremal process of the two-type branching process when parameters $(\beta,\sigma^2)$ are on the boundaries between $\mathscr{C}_{I}, \mathscr{C}_{II},\mathscr{C}_{III}$, except the point $(1,1)$.  In this cases   there were some conjectures in  \cite[Section 2,4]{BM21}.
Our main results confirm that the conjectures are true with some coefficients being corrected and
we also give the result for the case for which there was no conjecture.

\subsection{Main results}
In the statements of our theorems,
we continue to use the notation introduced earlier:
 the constant $C_{\star}$ in \eqref{eq-converge-to-DPPP}   and decorations $(\mathfrak{D}^{\varrho}: \rho \geq \sqrt{2})$ in \eqref{eq-decoration-process},   the derivative martingale $Z^{(1)}_{\infty}$ and additive martingales $W^{(1)}_{\infty} (\lambda)$ in \eqref{eq-def-martingale-type1} and \eqref{eq-def-martingale-type1'}.
We note  that, since $\sqrt{2}< \theta$, we  have $W^{(1)}_{\infty}(\sqrt{2})>0$ almost surely. Denote by
\begin{align*}
 & \mathscr{B}_{I,II}=\partial \mathscr{C}_{I} \cap \partial \mathscr{C}_{II} \backslash \{(1,1)\}  = \left\{ (\beta,\sigma^2) :   \beta \sigma^2=1,  \beta<1  \right\}   ;\\
  & \mathscr{B}_{II,III} =\partial \mathscr{C}_{II} \cap \partial \mathscr{C}_{III} \backslash \{(1,1)\} = \left\{ (\beta,\sigma^2) :   \beta + \sigma^2=2,  \beta>1  \right\} ;\\
 & \mathscr{B}_{I,III} =\partial \mathscr{C}_{I} \cap \partial \mathscr{C}_{III} \backslash \{(1,1)\} = \left\{ (\beta,\sigma^2) :   \frac{1}{\beta} + \frac{1}{\sigma^2}= 2,  \beta<1  \right\}.
\end{align*}

\begin{figure}[htbp]
  \centering
  \includegraphics[scale=0.9]{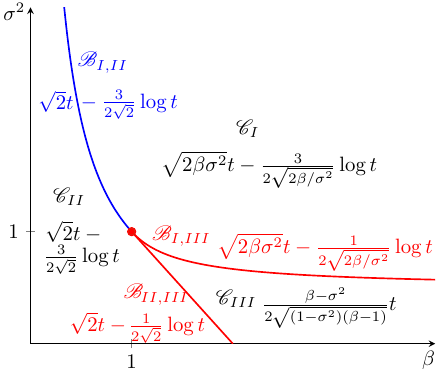}
  \caption{Phase diagram for the maximum of two type reducible BBM} \label{fig-phase}
\end{figure}

\begin{theorem}[Boundary between $\mathscr{C}_{II}$, $\mathscr{C}_{III}$]\label{thm-extremal-process-boundary-C2C3}
Assume that $(\beta,\sigma^2)\in \mathscr{B}_{II,III}$. Let
\begin{equation*}
m^{2,3}_{t}:= \sqrt{2} t - \frac{1}{2 \sqrt{2}} \log t .
\end{equation*}
 Then for the constant $C=\frac{\alpha C_{\star}} {\sqrt{2}(1-\sigma^2)}>0$, we have
 \begin{equation*}
  \lim_{t \to \infty}
 \sum_{u \in N_t} \delta_{X_u(t)-m^{2,3}_{t}} =  \lim_{t \to \infty}\sum_{u \in N^{2}_t} \delta_{X_u(t)-m^{2,3}_{t}}= \operatorname{DPPP}\left(\sqrt{2} C W^{(1)}_{\infty}(\sqrt{2}) e^{-\sqrt{2} x}\dif  x, \mathfrak{D}^{\sqrt{2}}\right) \text{ in law}.
 \end{equation*}

 \end{theorem}

 Theorem \ref{thm-extremal-process-boundary-C2C3} confirms Conjecture 2.2 of Belloum and Mallein  \cite{BM21}
 with the constant before $\log t$ being corrected as $ \frac{1}{2 \sqrt{2}}$, and the random variable $\widetilde{Z}$ there taken to be $W^{(1)}_{\infty} (\sqrt{2})$.

Recall that throughout this paper, we set
 \begin{equation*}
   v:=\sqrt{2 \beta \sigma^2}\ \text{ and } \ \theta:=\sqrt{ {2 \beta}/{\sigma^2}}.
 \end{equation*}

\begin{theorem}[Boundary between $\mathscr{C}_{I}$, $\mathscr{C}_{III}$]\label{thm-extremal-process-boundary-C1C3}
  Assume that $(\beta,\sigma^2)\in \mathscr{B}_{I,III}$. Let
 \begin{equation*}
  m^{1,3}_{t}:= v t - \frac{1}{2 \theta} \log t.
 \end{equation*}
Then we have
\begin{equation*}
  \lim_{t \to \infty}
  \sum_{u \in N_t} \delta_{X_u(t)-m^{1,3}_{t}}=\lim_{t \to \infty}
\sum_{u \in N_t^2} \delta_{X_u(t)-m^{1,3}_{t}}=\mathrm{DPPP}\left(\theta C Z^{(1)}_{\infty} e^{- \theta x}\dif  x, \mathfrak{D}^{\theta}\right) \text{ in law}
\end{equation*}
with $C=\frac{ 2\alpha  \sigma C(\theta)}{\sqrt{2 \pi}\beta (1-\sigma^2)} $  (where $C(\theta)$ is defined in  Lemma \ref{thm-Laplace-BBM-order}, (ii)).
\end{theorem}

Theorem \ref{thm-extremal-process-boundary-C1C3} confirms
Conjecture 2.1 of Belloum and Mallein  \cite{BM21} with the  coefficient
before $\log t$ being corrected as $ \frac{1}{2 \theta}$ and the decoration law $\widetilde{\mathfrak{D}}$ there taken to be $\mathfrak{D}^{\theta}$.

\begin{theorem}[Boundary between $\mathscr{C}_{I}$, $\mathscr{C}_{II}$]\label{thm-boundary-i-ii} Assume that $(\beta,\sigma^2)\in \mathscr{B}_{I,II}$. Let
  \begin{equation*}
    m^{1,2}_{t}:=m_{t} =\sqrt{2} t-\frac{3}{2 \sqrt{2}} \log t.
  \end{equation*}
For some random variable $\bar{Z}_{\infty}$ (defined in  Lemma  \ref{lem-bar-Z-infinity}), we have
  \begin{equation*}
    \lim _{t \rightarrow \infty} \sum_{u \in N_t} \delta_{X_u(t)-m^{1,2}_{t}}=\lim _{t \rightarrow \infty} \sum_{u \in N_t^2} \delta_{X_u(t)-m^{1,2}_{t}}= \operatorname{DPPP}\left(\sqrt{2} C_{\star} \bar{Z}_{\infty} e^{-\sqrt{2} x} \dif x, \mathfrak{D}^{\sqrt{2}}\right) \text{ in law.}
  \end{equation*}
  \end{theorem}

  \begin{remark}
Combining Theorems \ref{thm-extremal-process-boundary-C2C3}, \ref{thm-extremal-process-boundary-C1C3}, \ref{thm-boundary-i-ii} and results in \cite{BM21,Belloum22}, we can observe a double jump in the maximum $\max_{u \in N_{t} }X_{u}(t)$ when
the parameters $(\beta,\sigma^2)$ cross
the boundary of anomalous spreading region $\mathscr{C}_{III}$.
The leading order varies continuously but the logarithmic correction changes from $-\frac{3}{2\sqrt{2}}\log t$ to $-\frac{1}{2\sqrt{2}} \log t$  then to $0$; or  from $-\frac{3}{2\theta}\log t$ to $-\frac{1}{2\theta}\log t $ then to $0$.
See Figure \ref{fig2}.
Such a phase transition reminds us of the study of time inhomogeneous BRW \cite{FZ12b}, in which a constant  multiplying the logarithmic correction changes from $-\frac{1}{2}$ to $-\frac{3}{2 }  $ then to $-\frac{6}{2} $.
Also in a more general setting of \cite{FZ12b}, phase transitions  becomes a little bit more complex and  a double jump can occur as well (see \cite{Mallein15b}).
 A more interesting problem is to make the logarithmic correction smoothly interpolates from $1$ to $6$,  which is done for variable speed BBM in \cite{BH20}
 (see also \cite{KS15}). For  two-type reducible BBMs, we  ask a similar  question as follows.
\begin{question} Can we  let
the parameters $(\beta_{t},\sigma^2_{t})$
depend  on the time horizon $t$, in order that the logarithmic correction for the order of the maximum at time $t$  smoothly interpolates
$-\frac{3}{2\sqrt{2}} \log t$ to $-\frac{1}{2\sqrt{2}} \log t$ then to $0$
or  from $-\frac{3}{2\theta} \log t$ to $-\frac{1}{2\theta} \log t$ then to $0$?
\end{question}
 \end{remark}

\begin{remark}
  Question 1 above has now been addressed in our recent preprint \cite{MR24}.
\end{remark}

\begin{figure}
  \begin{minipage}[t]{0.48\textwidth}
  \centering
  \includegraphics[scale=0.3]{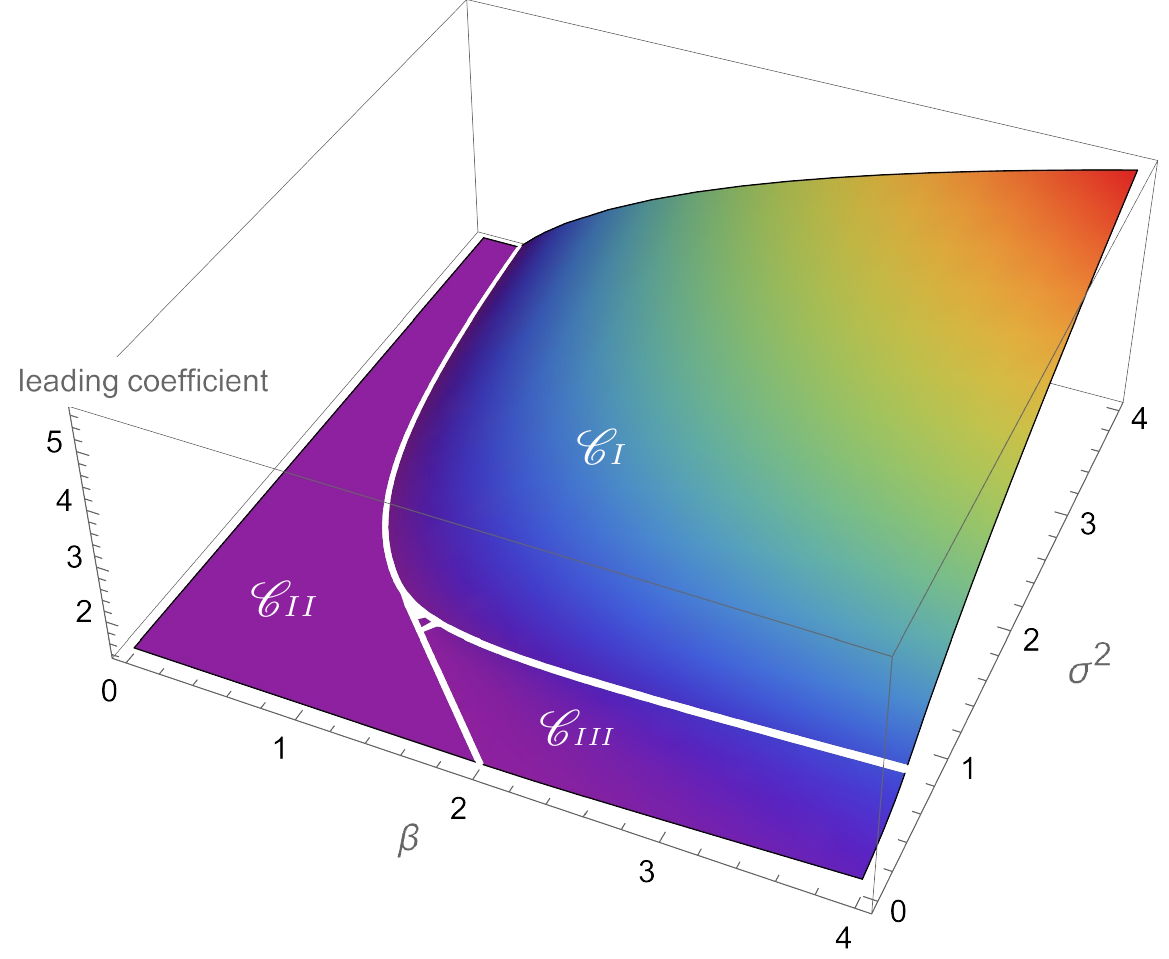}
  \end{minipage}
  \begin{minipage}[t]{0.5\textwidth}
  \centering
  \includegraphics[scale=0.3]{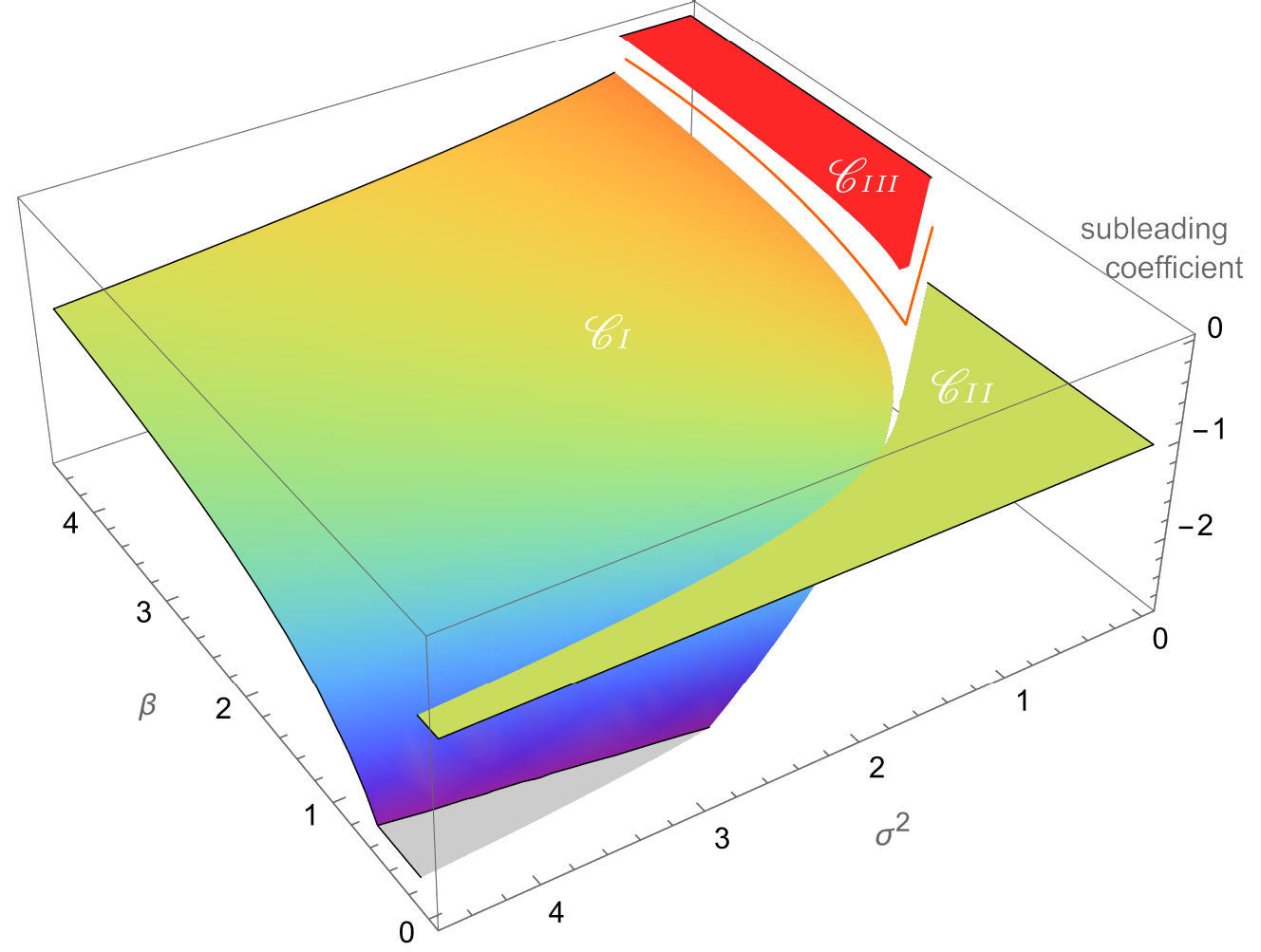}
  \end{minipage}
  \caption{Coefficients for term $t$ and $\log t$ as function of $(\beta,\sigma^2)$}
  \label{fig2}
  \end{figure}

  \begin{remark}
When $(\beta,\sigma^2) \in \mathscr{B}_{I,II}$,
 the localization of paths of extremal particles
is the same as in the case  $(\beta,\sigma^2) \in \mathscr{C}_{II}$ (see Lemma \ref{lem-Tu-less-than-R}).
So when
 $(\beta,\sigma^2)$ crosses the
 boundary between $\mathscr{C}_{I}$, $\mathscr{C}_{II}$, the maximum of the process only experience one jump:  the subleading coefficient
 changes from $-\frac{3}{2\sqrt{2}}$ to $-\frac{3}{2\theta}$.
\end{remark}

\begin{remark}
In the proof of Theorems \ref{thm-extremal-process-boundary-C2C3}, \ref{thm-extremal-process-boundary-C1C3}, \ref{thm-boundary-i-ii}, we only have to prove the convergence of the extremal process  consisting of particles of type $2$. Because by \eqref{eq-Gumbel} and \eqref{eq-BBMscaling}, the height of the highest type $1$ particles is around $v t- \frac{3}{2 \theta} \log t+O_{\P}(1)$, which is
 much below  type $2$ particles.
Moreover, our results Theorems \ref{thm-extremal-process-boundary-C2C3}, \ref{thm-extremal-process-boundary-C1C3}, \ref{thm-boundary-i-ii}  can be strengthened as the
joint convergence of the extremal processes and its maximum  $(\mathcal{E}_{t}, \max \mathcal{E}_{t})$ to
$ (\mathcal{E}_{\infty}, \max \mathcal{E}_{\infty})$, where $\mathcal{E}_{\infty}$ is a DPPP and $\max \mathcal{E}_{\infty}$ is its maximum. The reason is that,  by  \cite[Lemma 4.4]{BBCM22}, to prove
the  convergence of $ (\mathcal{E}_t, \max \mathcal{E}_t)$,
 it is sufficient to show the convergence of Laplace functional $\mathbb{E}[e^{-\langle \phi, \mathcal{E}_{t} \rangle}]$ with certain test functions $\varphi \in \mathcal{T}$ introduced below.
\end{remark}

\paragraph{Notation convention.}
Let  $\mathcal{T} $  be the set of  functions $\varphi \in C_{b}^{+}(\mathbb{R})$ such that $\inf \mathrm{supp}( \varphi)>-\infty$  and for some $a \in \mathbb{R}$,  $ \varphi(x) \equiv$ some positive constant for all $x > a$.
$\mathcal{T}$ will  serve as  test functions in the Laplace functional
(see \cite[Lemma 4.4]{BBCM22}).
For two quantities $f$ and $g$, we write $f \sim g$ if $\lim f/g= 1$.  We write $f  \lesssim g$ if there exists a constant $C>0$ such that $f \leq C g $. We write $f  \lesssim_\lambda g$ to stress that the constant $C$ depends on parameter $\lambda$. We use the standard notation $\Theta(f)$  to denote a non-negative quantity such that there exists constant  $c_{1}, c_{2}>0$ such  that $c_{1} f \leq \Theta(f)  \leq c_{2}f $.

\subsection{Heuristics}
We restate the optimization problem introduced in \cite[Section 2.1]{BM21} (see also Biggins \cite{Biggins10}).
 First, we introduce the following definition:

\begin{definition}
  If $u\in N^{2}_{t}$, we denote by  $T_{u}$   the time at which the oldest ancestor of type $2$ of $u$ was born. We say $T_{u}$ is the \textit{type transformation time} of $u$.
\end{definition}

 For $p\in[0,1]$, let $\mathcal{N}_{p,a,b}(t)$ be the expected number of type $2$ particles at time $t$ who has speed $a$ before time $T_{u}\approx pt$ and speed $b$ after time $pt$.
Note that these particles are at level $[pa+(1-p)b]t$. First-moment computations yield that  there are around $e^{\left(\beta- \frac{a^2}{2 \sigma^2} \right)p t+o(t)}$ type $1$ particles at time $p t$ at level $a pt$, and a type $2$ particle  has probability $e^{\left(1-\frac{b^2}{2}\right)(1-p) t +o(t)}$ of having a descendant  at level $b(1-p)  t$ at time $(1-p)t$.
 Hence we get  $ \mathcal{N}_{p,a,b}(t)=\exp \left\{[(\beta -\frac{a^2}{2 \sigma^2})p+(1-\frac{b^2}{2})(1-p) ] t + o(t) \right\}$. In order to get the maximum,  we just  maximize $pa+(1-p)b$ among all the parameter $p,a,b$ such that $\left(\beta- \frac{a^2}{2 \sigma^2} \right)p  \geq 0$ and $\mathcal{N}_{p,a,b}(t) \geq 1$. This turns to the following optimization problem:
\begin{equation*}
v^{*}=\max  \left\{p a+(1-p) b: p \in [0,1], \left(\beta-\frac{a^2}{2 \sigma^2}\right)p \geq 0,    \left(\beta-\frac{a^2}{2 \sigma^2} \right)p+ (1-\frac{b^2}{2})(1-p) \geq 0\right\} .
  \end{equation*}
Write $\left(p^*, a^*, b^*\right)$ for the maximizer of the problem above. If $\left(\beta, \sigma^2\right) \in \mathscr{C}_I$, then $p^*=1,a^{*}=v$, and $v^*=v$; if $\left(\beta, \sigma^2\right) \in \mathscr{C}_{II}$,  then  $p^*=0, b^{*}=\sqrt{2}$ and   $v^*=\sqrt{2}$; if   $\left(\beta, \sigma^2\right) \in \mathscr{C}_{III}$, then
 $p^{*}=\frac{\sigma^2+\beta-2}{2\left(1-\sigma^2\right)(\beta-1)}$, $b^{*}=\sqrt{2 \frac{\beta-1}{1-\sigma^2}}$, $a^{*}=\sigma^2 b^{*}$, and  $v^{*}= \frac{\beta-\sigma^2}{\sqrt{2\left(1-\sigma^2\right)(\beta-1)}} $. Also if $(\beta,\sigma^2)=(1,1)$, then $p^{*}$ can be arbitrary in $[0,1]$,  $a^{*}=b^{*}=\sqrt{2}$ and  $v^{*}=\sqrt{2}$.

When $(\beta,\sigma^2) \in \mathscr{B}_{II,III}$, the maximizer is $p^{*}=0$, $v^{*}=b^{*}=\sqrt{2}$, and $a^{*}$ can be arbitrary.
Hence  each individual $u \in N^{2}_{t}$ near the maximal position should satisfy $p^{*}=T_{u}/t \approx 0$.
But now the order of $T_{u}$ really matters, and  if $T_{u} \gg 1$, $a^{*}$ should be $\sqrt{2} \sigma^2$ predicted by the formula for  the case $\mathscr{C}_{III}$.
When $(\beta,\sigma^2) \in \mathscr{B}_{I,III}$,
the maximizer   $p^{*}=1$, $v^{*}=a^{*}=v$, and $b^{*}$ can be arbitrary.
We can deduce that  each individual $u \in N^{2}_{t}$ near the maximal position satisfies $p^{*}=T_{u}/t \approx 1$.
The order of $t-T_{u}$ is  also important, and
  if $t-T_{u} \gg 1$, $b^{*}$ should be $\sqrt{2 \beta/\sigma^2}$ predicted by the formula for  the case $\mathscr{C}_{III}$.
 Similar problems  occur when we consider the boundary $\mathscr{B}_{I,II}$.
 The following computations, based on a  finer analysis, provide more insights for  localization of paths of extremal particles.

 \paragraph*{The case $(\beta,\sigma^2) \in \mathscr{B}_{II,III}$.} As in the computation above,
 the expected  number  of particles of type $1$ at time $s=o(t)$  at level $\lambda s$ is roughly
 $e^{(\beta-\frac{\lambda^2}{2 \sigma^2})s+O(\log t)} $.
A typical particle of type $2$ has probability $ e^{-(2-\sqrt{2}\lambda)s- \frac{(\sqrt{2}-\lambda)^2}{2} \frac{s^2}{t-s}+O(\log t)} $ of
having a descendant at level
  $\sqrt{2}t-\lambda s$ at time $t-s$.  Hence there are around
  \begin{equation}\label{eq-number-1}
    \exp \left\{  \left[\beta-\frac{\lambda^2}{2 \sigma^2}-(2-\sqrt{2}\lambda)\right]s - \frac{(\sqrt{2}-\lambda)^2}{2} \frac{s^2}{t-s}+O(\log t) \right\}
  \end{equation}
  particles of type $2$ at level $\sqrt{2}t $ at time $t$ such that $T_{u} \approx s$ and $X_{u}(T_{u}) \approx \lambda s$.
In order that the quantity in  \eqref{eq-number-1} is not zero as $  t\to\infty$,
  first we have to ensure that  $ \beta-\frac{\lambda^2}{2 \sigma^2} -(2 - \sqrt{2}\lambda) = -\frac{1}{2 \sigma^2}(\lambda-\sqrt{2}\sigma^2)^2  \geq 0 $
(here we used $\beta+\sigma^2=2$),
which forces $\lambda=\sqrt{2} \sigma^2$.
Secondly we have to ensure that
$\frac{s^2}{t-s}$ is bounded, i.e., $s=O(\sqrt{t})$.
In other words, the extremal particle $u \in N^2_{t}$ should satisfy $T_{u}=O(\sqrt{t})$
and $X_{u}(T_{u}) \approx  \sqrt{2} \sigma^2 T_{u}$.

\paragraph{The case $(\beta,\sigma^2) \in    \mathscr{B}_{I,III}$.}
The expected  number
 of type $1$ particles at time $s=t-o(t)$ at level $vs-a(t)$ (where  $a(t)$ will be determined later) is roughly $ e^{ \theta a(t) - \frac{a(t)^2}{2\sigma^{2} s} +O(\log t)}$.
A typical particle of type $2$ has  probability
$ e^{-\left[ (\frac{v^2}{2}-1)(t-s)+  v a(t)+  \frac{a(t)^2}{2(t-s)}\right]+O(\log t)} $ of having a descendant at level $v(t-s)+a(t)$ at time $t-s$.  Hence there are around
  \begin{equation}\label{eq-number-2}
    \exp\left\{  -\left[ \left(\frac{v^2}{2}-1\right)(t-s) + \frac{a(t)^2}{2 \sigma^{2} s}+\frac{a(t)^2}{2(t-s)}   \right] + (\theta-v)a(t) +O(\log t) \right\}
  \end{equation}
  particles of type $2$ at level $vt $ at time $t$ such that $T_{u} \approx s$ and $X_{u}(T_{u}) \approx v s-a(t)$.
  In order that the quantity in  \eqref{eq-number-2} is not zero as $  t\to\infty$,  using the prior knowledge $s \sim t$,  first we have to ensure that
  $a(t)$ has the same order as  $t-s$   or $\frac{a(t)^2}{t-s}$, and we get $a(t) =\Theta(t-s)$.
  We also need to ensure that $\frac{a(t)^2}{2s} = O(1)$, thus implies $a(t) = \Theta( t-s) = O( \sqrt{t})$.
   So,  letting $a(t)= a \sqrt{t}$ and  $t-s=\lambda \sqrt{t}$,
   we can rewrite \eqref{eq-number-2} as
  \begin{equation*}
    \exp\left\{ \left(\theta a  -\frac{(a+\lambda v)^2}{2\lambda}+\lambda\right)\sqrt{t}+   O(\log t)  \right\}.
  \end{equation*}
  We  now have to ensure
  $ \theta a  -\frac{(a+\lambda v)^2}{2\lambda}+\lambda  = -\frac{1}{2 \lambda}[a-(\theta-v)\lambda]^2  \geq 0 $
(here we used $\frac{1}{\beta} + \frac{1}{\sigma^2} =2$).
This  forces $a= (\theta-v)\lambda $ and hence $a(t)=a \sqrt{t}=(\theta- v)(t-s)$.
In other words, the extremal particle $u \in N^2_{t}$ should satisfy
$T_{u}=t-\Theta(\sqrt{t})$
and $X_{u}(T_{u}) \approx v T_{u}- (\theta-v)(t-T_{u})$.

\paragraph{The case $(\beta,\sigma^2) \in\mathscr{B}_{I,II}$}
We do the same computation as in the case $(\beta,\sigma^2)\in \mathscr{B}_{I,III}$, and still get that
the expected number of type $2$ particles  at level $vt $ at time $t$ satisfying $T_{u} \approx s$ and $X_{u}(T_{u}) \approx v s-a(t)$ is given by  \eqref{eq-number-2}.
However, now we are in the case $\mathscr{B}_{I,II}$ with $\beta \sigma^2=1$ and $\beta < 1$, and hence  $v=\sqrt{2}$ and $\theta=\sqrt{2} \beta $. Therefore
  \eqref{eq-number-2} becomes
\begin{equation}\label{eq-number-3}
  \exp\left\{  -\left[ \sqrt{2}(1-\beta)a(t)+ \frac{a(t)^2}{2\sigma^{2} s}+\frac{a(t)^2}{2(t-s)}   \right]    +O(\log t) \right\}.
\end{equation}
In order that \eqref{eq-number-3} tends to a nonzero limit, we need $a(t)=O(\log t)$. This is very different from the case $(\beta,\sigma^2)=(1,1)$ in \cite{Belloum22},
in which case we can get $a(t)=O(\sqrt{t})$
 (and $s$ is of order $t$).
Now the simple first moment computations can not tell us more. However, we can still make a \textit{guess}.
Notice that the extremal particle $u \in N^{2}_{t}$ are also extremal (up to $O(\log t)$) at time $T_{u}$.
This fact reminds us of the decreasing variances case in  \cite{FZ12b}: The maximum at time $t$ is the highest value among the descendants of the maximal particle at time $t/2$.
So   we guess that $\max_{u \in N_{t}}X_{u}(t) \approx \max_{T_u \in [0,t]}\{\sqrt{2}T_{u} - \frac{3}{2 \theta} \log T_{u} + \sqrt{2}(t-T_u)- \frac{3}{2 \sqrt{2}} \log (t-T_{u}) \}$.
As $\theta<\sqrt{2}$,
we should choose $T_{u}$ small.   So we expect $T_{u}=O(1)$ and $ \max_{u \in N_{t}}X_{u}(t)$ should be $\sqrt{2}t-\frac{3}{2\sqrt{2}}\log t$.

\section{Preliminary results}

\subsection{Brownian motions estimates}

We always use $\left\{(B_t)_{t\geq 0}; \mathbf{P}\right\}$ to denote  a standard Brownian motion (BM) starting from the origin.  Here is a useful upper bound for the probability that a Brownian  bridge is below a line,
see \cite[Lemma 2]{Bramson78} for a proof.

\begin{lemma}\label{lem-bridge-estimate-0}
Consider a line segment with endpoints $(0,x_{1}), (t,x_{2})$
with  $x_{1},x_{2} \geq 0$. We have
  \begin{equation*}
 \mathbf{P}
    \left( B_{s} \leq \frac{s}{t} x_{2}+ \frac{t-s}{t}x_{1}, \,\forall s \in [0,t] \, \bigg| B_{t}=0 \right) = 1 - e^{-\frac{2 x_{1}x_{2}}{t}} \leq \frac{2x_{1}x_{2}}{t}.
  \end{equation*}
\end{lemma}
As a corollary, we also have an estimate for the probability that
a BM stays below a line and ends up in a finite interval.
For all $K \geq 1$ and $y \geq 0$ we have
\begin{equation}\label{eq-brownian-estimate}
\mathbf{P}\left(B_s \leq K, s \leq t, B_t- K\in [-y-1,-y]\right) \leq C \frac{K (y+1)}{ (1+t)^{3 / 2}} .
\end{equation}
In fact, the desired probability is  less than
the product of $\mathbf{P}(B_{t}-K +y\in [-1,0]) \leq \frac{1}{\sqrt{2\pi t}}$ and $\max_{z \in [0,1]}\mathbf{P} ( B_{s} \leq K, \,\forall s \in [0,t] \, | B_{t}=K-y-z ) \leq 2K(y+1)/t$ by Lemma \ref{lem-bridge-estimate-0} and scaling invariance properties of Brownian bridge.

Later in the proof of Lemma \ref{lem-born-before-t-t1/2}, we will use a slight modification of Lemma \ref{lem-bridge-estimate-0} as follows. For completeness,
 we give its proof in the appendix A.

\begin{lemma}\label{lem-bridge-estimate}
Let $\widetilde{m}_{t}=vt- w_{t}$, where $w_t = \Theta( \log t)$.
Assume that $\sigma^2 \leq 1$. Fix $K \geq 0$.
Then  for  sufficiently large $t$,
  $s \in [t- \sqrt{t}\log t, t]$ and $x \in \mathbb{R}$,
 \begin{equation}\label{eq-bridge-estimate}
    \mathbf{P}\left(  \sigma B_{r} \leq v r + K, \forall r \leq s |  \sigma B_s + B_{t}-B_{s}= \widetilde{m}_{t} + x  \right)
      \lesssim_{K,\beta,\sigma}  \frac{t-s + w_{t}  +|x|}{t} .
  \end{equation}
\end{lemma}

\subsection{Branching Brownian motions estimates}
We always use $\{(\mathsf{X}_{u}^{\beta,\sigma^2}(t):u \in \mathsf{N}_{t})_{t\geq 0},  \mathsf{P}\}$ to denote a BBM  starting from one  particle at the origin with branching rate $\beta$ and diffusion coefficient $\sigma^2$.
For the BBM, there is an upper envelope through which particles find it difficult to pass.
In fact, letting $a_t=\frac{3}{2 \theta} \log (t+1)$,  for  some constant  $C>0$ and  for all $t >0, K>0$,
\begin{equation}\label{eq-upper-envelope}
\mathsf{P}\left(\exists s \leq t, u \in \mathsf{N}_s: \mathsf{X}^{\beta,\sigma^2}_u(s) \geq v s-a_t+a_{t-s}+K\right) \leq C(K+1) e^{-\theta K},
\end{equation}
see  \cite[(6.1)]{BM21} (or  \cite[Lemma 3.1]{Mallein15}).
In particular, we have
\begin{equation}\label{eq-upper-envelope-0}
  \mathsf{P}\left(\exists s \leq t, u \in \mathsf{N}_s: \mathsf{X}^{\beta,\sigma^2}_u(s) \geq v s +K\right) \leq C(K+1) e^{-\theta K} .
  \end{equation}

We collect several results for the standard BBM
$(\mathsf{X}_{u}(t):u \in \mathsf{N}_{t})_{t\geq 0},$ (i.e., $\beta=\sigma^2=1$) that will be used frequently.
Recall that $\mathsf{M}_{t}=  \max_{u \in \mathsf{N}_{t}} \mathsf{X}_{u}(t) $.
The following estimate of the upper tail of $\mathsf{M}_{t}$
was proved in \cite[Corollary 10]{ABK12}

\begin{lemma}\label{lem-Max-BBM-tail}
 For $x>1$ and $t \geq t_o$ (for $t_o$ a numerical constant),
\begin{equation*}
\mathsf{P}\left(\mathsf{M}_{t}  \geq \sqrt{2}t- \frac{3}{2\sqrt{2}} \log t + x \right) \leq \rho \cdot x \cdot \exp \left(-\sqrt{2} x-\frac{x^2}{2 t}+\frac{3}{2 \sqrt{2}} x \frac{\log t}{t}\right)
\end{equation*}
for some constant $\rho>0$.
\end{lemma}

  As a  consequence of Lemma \ref{lem-Max-BBM-tail}, we have
\begin{equation}\label{eq-Max-BBM-tail-2}
  \mathsf{P}\left(\mathsf{M}_{t}  \geq \sqrt{2}t- \frac{3}{2\sqrt{2}} \log t + x \right) \leq \rho \cdot x   \exp \left(-\sqrt{2} x + 1 \right) ,
\end{equation}
 for all $x > 1$ and $t>t_0$, since   $-\frac{x^2}{2 t}+\frac{3}{2 \sqrt{2}} x \frac{\log t}{t} \leq 1$, for all $  t > 100$ and  $x > 1$.

The following estimates for the Laplace functionals of standard BBMs
can be found in \cite[Corollary 2.9]{Belloum22} and \cite[Lemma 3.7]{BM21}.
One can also obtain the same results
from the large deviations probability  $\mathsf{P}(\mathsf{M}_{t} > \rho t + x )$ with $\rho \geq \sqrt{2}$ and the conditioned convergence of the gap processes  \eqref{eq-decoration-process}.

\begin{lemma}\label{thm-Laplace-BBM-order}
Let  $\varphi \in \mathcal{T}$, $R>0 $ and $\rho > \sqrt{2}$. Then the following assertions hold.
\begin{enumerate}[(i)]
  \item   For $x \in [-R\sqrt{t}, - \frac{1}{R}\sqrt{t} ]$ uniformly
  \begin{equation*}
1-\mathsf{E}\left(e^{-\sum_{u \in \mathsf{N}_t} \varphi\left(x+\mathsf{X}_u(t)-\sqrt{2} t\right)}\right)= \gamma_{\sqrt{2}}(\varphi)\frac{ (-x )}{t^{3 / 2}}e^{\sqrt{2} x-\frac{x^2}{2 t}} (1+o(1)),
\end{equation*}
as $t \rightarrow \infty$, where $\gamma_{\sqrt{2}}(\varphi)= \displaystyle \sqrt{2}C_{\star} \int e^{-\sqrt{2} z}\left(1-\mathsf{E}\left(e^{-\langle\mathfrak{D}^{\sqrt{2}}, \varphi(\cdot+z)\rangle}\right)\right) \dif z$.
\item For $|x|  \leq R\sqrt{t} $ uniformly
\begin{equation*}
  1-\mathsf{E}\left(e^{-\sum_{u \in \mathsf{N}_t} \varphi\left(x+\mathsf{X}_u(t)-\rho t\right)}\right)=
   \gamma_\rho(\varphi) \frac{e^{\left(1-\rho^2 / 2\right) t}}{\sqrt{t}} e^{\rho x-\frac{x^2}{2 t}} (1+o(1)),
\end{equation*}
 as $t \rightarrow \infty$, where $
  \gamma_\rho(\varphi)=  \displaystyle\frac{C(\rho)   }{\sqrt{2 \pi }}\int e^{-\rho z}\left(1- \mathsf{E} (e^{-\left\langle\mathfrak{D}^\rho, \varphi(\cdot+z)\right\rangle})  \right) \dif z$.
\end{enumerate}
\end{lemma}

\subsection{Choosing an individual   according to Gibbs measure}
For a standard BBM $\{(\mathsf{X}_{u}(t):u \in \mathsf{N}_{t}), \mathsf{P}\}$, it is well known that the   additive martingale
\begin{equation*}
  \mathsf{W}_{t}(\eta)=\sum_{u \in \mathsf{N}_{t}} e^{\eta\mathsf{X}_{u}(t) - \left(\frac{\eta^2}{2}+1\right) t}
\end{equation*}
converges almost surely and in $L^{1}$
to a non-degenerate random   variable
$\mathsf{W}_{\infty}(\eta)$ when $\eta \in (0,\sqrt{2})$.
And, when $\eta=\sqrt{2}$,
the non-negative martingale $\mathsf{W}_{t}(\eta) $ converges to zero with
probability one (see e.g. \cite{Kyprianou03}).
Conditioned  on BBM at time $t$, we randomly choose a particle $u \in \mathsf{N}_{t}$ with probability   $ \frac{ e^{\eta\mathsf{X}_{u}(t) } } { \sum_{u \in \mathsf{N}_{t}} e^{\eta\mathsf{X}_{u}(t) } } $,
which is the so called  \textit{Gibbs measure} at inverse  temperature $\eta$.
Hence the  additive martingale is the \textit{normalized partition function} of the Gibbs measure.

Firstly we state a law of large number theorem for the particle chosen according to the Gibbs measure. This result is not new.
Since we didn't find a reference,  we offer a simple proof in Section \ref{sec-Proof-spine-decomposition} for completeness.

\begin{proposition}\label{lem-functional-convergence-additive-martingale}
  Let $f$ be a bounded continuous function on $\mathbb{R}$. Suppose $\eta \in (0,\sqrt{2})$. Define
  \begin{equation*}
    \mathsf{W}_{t}^{f}(\eta):= \sum_{u \in \mathsf{N}_{t}} f\left(\frac{\mathsf{X}_{u}(t)}{t}\right)  e^{\eta\mathsf{X}_{u}(t) - \left(\frac{\eta^2}{2}+1\right) t}.
   \end{equation*}
   Then
   \begin{equation*}
      \lim_{t \to \infty}  \mathsf{W}_{t}^{f}(\eta) = f(\eta) \mathsf{W}_{\infty} (\eta) \
      \text{ in } L^{1}(\mathsf{P}).
   \end{equation*}
\end{proposition}

This law of large number holds since the Gibbs measure is supported on the particles at position around $\eta t$ for $\eta \in (0,\sqrt{2})$. One may further ask the fluctuations  of $X_{u}(t)-\eta t$.
Indeed a central limit theorem (CLT) holds (see \cite[(1.14)]{Pain18}): for $\eta \in (0,\sqrt{2})$ and
for  each bounded continuous function $f$,
\begin{equation*}
 \lim _{t \rightarrow \infty}
 \sum_{u \in \mathsf{N}_t}
  f\left(\frac{\mathsf{X}_u(t)-\eta t}{\sqrt{t}}\right) e^{\eta \mathsf{X}_u(t)-t\left(\frac{\eta^2}{2}+1\right)}=W_{\infty}(\eta) \int_{\mathbb{R}}  f(z) e^{-\frac{z^2}{2}}\frac{\dif z }{\sqrt{2 \pi}}
 \quad \mathsf{P}\text{-a.s.}
  \end{equation*}
  We refer to \cite[Corollary 4]{Biggins92} for the corresponding local limit theorem.
 In the critical case $\eta=\sqrt{2}$, the limiting distribution in the CLT  is no longer Gaussian,
but the Rayleigh distribution $\mu(\dif z)= z e^{-\frac{z^2}{2}} 1_{\{z>0\}}\dif z $.
 Madaule  \cite[Theorem 1.2]{Madaule16} showed that for every bounded continuous function $F$,
    \begin{equation}\label{eq-CLT-Gibbs-sqrt2}
   \lim_{t \to \infty}\sqrt{t} \sum_{u \in \mathsf{N}_{t}} F\left( \frac{\sqrt{2} t- \mathsf{X}_{u}(t)}{\sqrt{t}} \right) e^{-\sqrt{2}(\sqrt{2} t-\mathsf{X}_{u}(t))}
   =    \sqrt{\frac{2}{\pi}} \, \mathsf{Z}_{\infty} \langle F, \mu \rangle \ \text{ in probability,}
      \end{equation}
 where $\langle F, \mu \rangle := \int_{0}^{\infty} F(z) \mu(\dif z)$.
In fact, Madaule's result is a Donsker-type theorem for BRWs. A simple proof  of \eqref{eq-CLT-Gibbs-sqrt2}  can be found in \cite[Theorem B.1]{MZ16}.

The following proposition  gives  a natural generalization of   \eqref{eq-CLT-Gibbs-sqrt2}.
We didn't find such a result in the literature and to our knowledge, it is new.

\begin{proposition}\label{lem-functional-convergence-derivative-martingale}
 Let $G$ be a non-negative bounded measurable function with compact support.
Suppose $F_{t}(z)= G(\frac{z- r_{t}}{h_{t}})$,
where $r_{t}$ and $h_{t}$ satisfy that for large $t$, $ \log^{3}(t)/\sqrt{t} \ll r_{t}  \leq \bar{r}<\infty$, $r_{t}+yh_{t}=\Theta(r_{t})$  uniformly for $y \in \mathrm{supp}(G)$ and $h_{t}$ decreases at most polynomially fast.
Define
 \begin{equation*}
  \mathsf{W}^{F_t}_{t}(\sqrt{2})  := \sum_{u \in \mathsf{N}_{t}}  F_{t} \left( \frac{\sqrt{2}t- \mathsf{X}_{u}(t) }{\sqrt{t}} \right)e^{-\sqrt{2}(\sqrt{2} t-\mathsf{X}_{u}(t))}   .
 \end{equation*}
Then we  have
 \begin{equation*}
 \lim_{t \to \infty}\frac{ \sqrt{ t } }{ \langle F_{t},\mu \rangle } \, \mathsf{W}^{F_t}_{t}(\sqrt{2}) = \sqrt{\frac{2}{\pi}} \mathsf{Z}_{\infty}
 \text{ in probability.}
  \end{equation*}
 Taking $r_t=\lambda$ and $h_t=t^{-1/4}$, and denoting $F_{t,\lambda} (z):=G((z-\lambda) t^{1/4})$,
for any finite
interval $I \subset (0,\infty)$ with strictly positive endpoints
we have
 \begin{equation}\label{conv-prop2}
  \lim\limits_{t \to \infty}   t^{3/4}
  \int_{I} \mathsf{W}^{F_{t,\lambda}}_{t-\lambda \sqrt{t}}(\sqrt{2})  \dif \lambda
  = \sqrt{\frac{2}{\pi}} \, \mu(I) \int_{\mathbb{R}}  G(y) \dif y \,  \mathsf{Z}_{\infty}
 \text{ in probability.}
 \end{equation}
  \end{proposition}

\begin{remark}
  To make  Proposition \ref{lem-functional-convergence-derivative-martingale} easier to understand, we choose the function  $G$   to be the indicator function $1_{[a,b]}$.
Letting $ h_{t}=1$ and $r_{t}=\lambda>0$,
Proposition \ref{lem-functional-convergence-derivative-martingale} gives the  CLT  \eqref{eq-CLT-Gibbs-sqrt2}:
\begin{equation*}
  \sqrt{t} \sum_{u \in \mathsf{N}_{t}} e^{-\sqrt{2}(\sqrt{2} t-\mathsf{X}_{u}(t))} 1_{\{ \sqrt{2}t-X_{u}(t)
   \in [(\lambda+a)\sqrt{t}, (\lambda+b)\sqrt{t}]\} }= [1+o_{\mathsf{P}}(1)] \mu([\lambda+a, \lambda+b]) \sqrt{\frac{2}{\pi}} \mathsf{Z}_{\infty}.
  \end{equation*}
 Letting $h_{t}=t^{-1/2}$ and $r_{t}=\lambda>0$,
Proposition \ref{lem-functional-convergence-derivative-martingale} yields that
  \begin{equation*}
    \sqrt{t} \sum_{u \in \mathsf{N}_{t}} e^{-\sqrt{2}(\sqrt{2} t-\mathsf{X}_{u}(t))}1_{\{ \sqrt{2}t-X_{u}(t) \in   [ \lambda\sqrt{t} +a, \lambda\sqrt{t}+b]\}}
 = [1+o_{\mathsf{P}}(1)] \frac{(b-a)}{\sqrt{t}} \lambda e^{-\frac{\lambda^2}{2}}  \sqrt{\frac{2}{\pi}}\, \mathsf{Z}_{\infty}.
  \end{equation*}
This  can be thought of a \textit{local limit theorem} (LLT)  result  for the position of a particle
sampled according to the Gibbs measure with parameter $\eta=\sqrt{2}$.
(See also  \cite[Theorem 4]{Biggins92}
for a LLT result in the case $\eta<\sqrt{2}$.)
Letting $h_{t}=t^{-1/4}$ and $r_{t} =\lambda>0$, we get
\begin{equation*}
  \sqrt{t} \sum_{u \in \mathsf{N}_{t}} e^{-\sqrt{2}(\sqrt{2} t-\mathsf{X}_{u}(t))} 1_{\{ \sqrt{2}t-X_{u}(t) \in   [ \lambda\sqrt{t} +at^{1/4}, \lambda\sqrt{t}+bt^{1/4}]\}}
 =[1+o_{\mathsf{P}}(1)]  \frac{(b-a)}{t^{1/4}} \lambda e^{-\frac{\lambda^2}{2}}  \sqrt{\frac{2}{\pi}}\, \mathsf{Z}_{\infty},
      \end{equation*}
 which  can be regarded as   a result  between CLT and LLT.
\end{remark}

  \begin{remark}
    Formally,  applying  Proposition \ref{lem-functional-convergence-derivative-martingale}
to the function  $zF_t(z)$ and letting $\hat{\mu}(\dif z )= \sqrt{\frac{2}{\pi}}z^2 e^{-z^2/2} 1_{\{z>0\}}\dif z $, we have
    \begin{equation*}
      \frac{1}{\langle
     F_{t},
      \hat{\mu} \rangle} \sum_{u \in \mathsf{N}_{t}}  F_{t} \left( \frac{\sqrt{2}t- \mathsf{X}_{u}(t) }{\sqrt{t}} \right)  (\sqrt{2}t- X_{u}(t)) e^{-\sqrt{2}(\sqrt{2} t-\mathsf{X}_{u}(t))} \to  \mathsf{Z}_{\infty} \quad \text{in probability},
     \end{equation*}
  A rigorous proof can be achieved by slightly modifying the proof of Proposition \ref{lem-functional-convergence-derivative-martingale}.
\end{remark}

  \begin{remark}\label{Rem2.9}
In the proof of Proposition \ref{lem-functional-convergence-derivative-martingale}, we use the powerful method developed in \cite{AS14}. The main step is to prove that a nonnegative random variable
$U_{t}:= U_{t}^{(K)}$ (see    \eqref{def-U^K_t} for the definition of $U_{t}^{(K)}$)
converges in $L^{1}$ under some probability
$\mathsf{Q}: = \mathsf{Q}^{(K)}$ (see \eqref{eq-DK-change-of-measure} for the definition of $\mathsf{Q}^{(K)}$).
In the CLT case (i.e., $r_{t}=\Theta(1)=h_{t}$),  the $2$nd moment of  $U_{t}$ is finite and  one can find good events $G_{t}$ so that
$\limsup\limits_{t\to\infty} \frac{\mathsf E_{\mathsf{Q}}[U_{t}^{2}1_{G_t}]}{\left(\mathsf{E}_{\mathsf{Q}}[(U_{t})]\right)^2}\le 1$  and $\mathsf E_{\mathsf{Q}}[U_{t}^{2}1_{G^c_t}]$ is negligible.
However,  when  $h_{t}=o(1)$ or $r_{t}=o(1)$, the $2$nd moment of $U_{t}$ appears to explode, and
$\mathsf E_{\mathsf{Q}}[U_{t}^{2}I_{G^c_t}]$ is not negligible.

In the first version of this paper \cite{MR23},  we overcome this  difficulty by using a bootstrap argument to show that the integral of $ U_{t}^{p}$ on $G_t^{c}$ is negligible and
$\limsup\limits_{t\to\infty}\frac{\mathsf E_{\mathsf{Q}}[U_{t}^{p}I_{G_t}]}{\left(\mathsf{E}_{\mathsf{Q}}[U_{t}]\right)^p}\leq 1$
 for $p=1+\epsilon$ with small $\epsilon$. But we need to further assume that   $\sqrt{t} r_{t}\gg t^{\epsilon}$.
We thank one of our referees for suggesting the current proof which is simpler than our first version. This type of argument is used in \cite[Section 5]{Pain18} also based on the ideas in \cite{AS14}.
Indeed we do not have to prove the stronger result on the  convergence of $p$-th moment.
In order to show the $L^1$ convergence of $U_t$, it suffices to show that
 the expectation of $U_{t}$ on $G_{t}^{c}$ is negligible and
  the expectation of $(U_{t})^{2}$ on $G_{t}$ is less than the square of the expectation of $U_{t}$.
 Also, the conditions on $r_{t}$ in \cite{MR23} can be
  weakened to
   $\sqrt{t} r_t \gg \log ^3 t $ and even weakened to $\sqrt{t} r_{t} \geq \log^{1+\epsilon}t$,
  by properly choosing  the barrier in the good event $G_{t}$.
However, as pointed by the same referee,
dealing with $r_t=O((\log t) / \sqrt{t})$ requires to work with a more precise barrier. In this regime what matters is not the distance between $X_u(t)$ and $\sqrt{2} t$, but rather the distance between $X_u(t)$ and $m_{t}$. This type of regime is covered in \cite{Pain18}. Moreover, results in \cite[Corollary 1.3]{Pain18} is somewhat similar to our proposition:  the scaling in the test function $F$ is always $1 / \sqrt{t}$, but the fact of having a weight $e^{\beta_t X_u(t)}$ instead of $e^{\sqrt{2} X_u(t)}$ implies that the considered sum is supported in some smaller windows.
  \end{remark}

\subsection{Many-to-one lemmas for two-type BBMs}

Let $\{(X_u(t), u \in N_{t}),\mathbb{P}\}$ be a  two-type reducible BBM. Recall that $\alpha $ is the rate at which a type $1$ individual
gives
birth to one type $2$ children.
Recall that for a type $2$ particle $u$,   $T_u$  is  the time at which the oldest ancestor of type $2$ of $u$ was born.
Let $b_{u}$ be the time when  $u$ was born.
We write
\begin{equation*}
\mathcal{B}=\left\{ u \in \cup_{t \geq 0} \, N_t^2,\quad T_{u}=b_u \right\}
\end{equation*}
for the set of particles of type $2$ that are born from a particle of type $1$.
The following  useful  many-to-one lemmas
were proved in  \cite[Proposition 4.1 and Corollary 4.3]{BM21}.

\begin{lemma}\label{many-to-one-1}
 Let $f$ be a non-negative  measurable function. Then
 \begin{enumerate}[(i)]
  \item  $\displaystyle
 \mathbb{E} \bigg(\sum_{u \in N_t^2} f\left(\left(X_u(r), r \leq t\right), T_{u}\right)\bigg)  =\alpha \int_0^t e^{\beta s+(t-s)}
\mathbf{E}
 \bigg(f\left(\left(\sigma B_{r \wedge s}+\left(B_r-B_{r \wedge s}\right), r \leq t\right), s\right)\bigg) \dif s$.
\item  $\displaystyle
  \mathbb{E}\left(\sum_{u \in \mathcal{B}} f\left(X_u(s), s \leq T_{u}\right)\right) =\alpha \int_0^{\infty} e^{\beta t}
  \mathbf{E}\left(f\left( \sigma B_s, s \leq t\right)\right) \dif t$.
\item  $ \displaystyle \mathbb{E}\bigg[\exp \bigg(-\sum_{u \in \mathcal{B}} f\left(X_u(s), s \leq T_{u}\right)\bigg)\bigg] =\mathbb{E}\bigg[\exp \bigg(-\alpha \int_0^{\infty} \sum_{u \in N_t^1} \left( 1-e^{-f\left(X_u(s), s \leq t\right)} \right) \dif t\bigg)\bigg] $.
 \end{enumerate}
\end{lemma}

\section{Boundary between $\mathscr{C}_{II}$ and $\mathscr{C}_{III}$}

In this section we always assume that
$(\beta,\sigma^2) \in \mathscr{B}_{II,III}$, i.e.,
$\beta> 1$, $\sigma^2 = 2- \beta < 1$  and recall that $m^{2,3}_{t}:= \sqrt{2} t - \frac{1}{2 \sqrt{2}} \log t$.

\begin{lemma}\label{lem-born-after-t1/2}
 For any $A>0$, we have
  \begin{equation*}
  \lim _{R \rightarrow \infty} \limsup _{t \rightarrow \infty} \mathbb{P}\left(\exists u \in N_{t}^{2}: T_{u} \notin [ \frac{1}{R} \sqrt{t}  ,R \sqrt{t}] , X_{u}(t) \geq m^{2,3}_{t}  -A \right)=0 .
  \end{equation*}
  \end{lemma}

We are going to show  Theorem \ref{thm-extremal-process-boundary-C2C3},
  postponing the proof of Lemma \ref{lem-born-after-t1/2} in the end of this section.

\begin{proof}[Proof of Theorem \ref{thm-extremal-process-boundary-C2C3}]
In this proof,
we set $\widehat{\mathcal{E}}_{t}:=\sum_{u \in N_t^2}   \delta_{  X_u(t)-m^{2,3}_{t} }  $, and  for fixed $R>0$, let
  \begin{equation*}
 \widehat{\mathcal{E}}_t^R
 :=\sum_{u \in N_t^2}
 1_{\left\{  T_{u}\in [\frac{1}{R} \sqrt{t}, R \sqrt{t}]  \right\}}
  \delta_{  X_u(t)-m^{2,3}_{t} } .
  \end{equation*}
  Thanks to Lemma \ref{lem-born-after-t1/2},
   $\widehat{\mathcal{E}}_t^R$ is very close to $\widehat{\mathcal{E}}_{t}$. More precisely, for all $\varphi \in \mathcal{T}$,  we have,
  \begin{equation*}
   \left|\mathbb{E}\left(e^{-\left\langle\widehat{\mathcal{E}}_t^R, \varphi\right\rangle}\right)-\mathbb{E}\left(e^{-\left\langle  \widehat{\mathcal{E}}_{t}, \varphi\right\rangle}\right)\right|  \leq \mathbb{P}\left(\exists u \in N_t^2: X_u(t) \geq m^{2,3}_{t}-A,  T_{u}\notin [\frac{1}{R} \sqrt{t}, R \sqrt{t}]  \right),
  \end{equation*}
  where $A$ is chosen such that
  $\mathrm{supp}(\varphi) \subset [-A, \infty)$.
  Thus applying Lemma \ref{lem-born-after-t1/2} we have
  \begin{equation}\label{eq-E-hat-t-2-3}
  \lim _{R \rightarrow \infty} \limsup _{t \rightarrow \infty}
  \left| \mathbb{E} \left(e^{-\left\langle\widehat{\mathcal{E}}_t^R, \varphi\right\rangle}\right)-\mathbb{E}\left(e^{-\left\langle \widehat{\mathcal{E}_{t}}, \varphi\right\rangle}\right)\right|=0.
  \end{equation}
 Hence in order to study the asymptotic behavior of
  $\mathbb{E}\left(e^{-\left\langle \widehat{\mathcal{E}_{t}}, \varphi\right\rangle}\right)$, we only need to study the convergence of $\widehat{\mathcal{E}}_t^R$  as $t$ and  then $R$ goes to infinity.

Fix $R>0$ and $\varphi \in \mathcal{T}$. Observe that we can rewrite $ \langle\widehat{\mathcal{E}}_t^R, \varphi \rangle $ as
  \begin{equation*}
    \sum\nolimits_{\substack{u \in \mathcal{B} \\   T_{u}\in [\frac{1}{R} \sqrt{t}, R \sqrt{t}] }}   \sum\nolimits_{\substack{u^{\prime} \in N_t^2 \\ u^{\prime} \succcurlyeq u}}
    \varphi\left( X_{u'}(t)-X_{u}(T_{u})- \sqrt{2}(t-T_{u}) + X_{u}(T_{u})- \sqrt{2}T_{u}+  \frac{1}{2\sqrt{2}} \log t  \right),
\end{equation*}
where $u^{\prime} \succcurlyeq u$ means that $u'$ is a descendant of $u$.
   Using the branching property first and then
   Lemma \ref{many-to-one-1}(iii),
    we have
\begin{equation*}
\begin{aligned}
&\mathbb{E}\left(e^{-\left\langle\widehat{\mathcal{E}}_t^R, \varphi\right\rangle}\right)
 =  \mathbb{E}\bigg(    \prod_{\substack{u \in \mathcal{B} \\   T_{u}\in [\frac{1}{R} \sqrt{t}, R \sqrt{t}] }}  f\big(t-T_{u} ,   X_{u}(T_{u}) -\sqrt{2} T_{u} + \frac{1}{2 \sqrt{2}} \log t  \big)  \bigg)
\\
&=\mathbb{E}\bigg(\exp \bigg\{ -\alpha \int_{\frac{1}{R}\sqrt{t}}^{R\sqrt{t}} \sum_{u \in N_s^1}   F\big(t-s, X_u(s)-\sqrt{2}s + \frac{1}{2\sqrt{2}}\log t \big) \dif s \bigg\} \bigg),
\end{aligned}
\end{equation*}
where $F(r, x)=1-f(r,x)=1-\mathsf{E}\left(  \exp\{-\sum_{u \in \mathsf{N}_r} \varphi ( \mathsf{X}_u(r)- \sqrt{2}r + x   )\}\right)$.
Additionally, as the speed $v=\sqrt{2 \beta \sigma^2}$ of the BBM of type $1$ is less than $\sqrt{2}$,  we have,  for all $s\in [\frac{1}{R}\sqrt{t}, R \sqrt{t}]$ and  $u \in N^{1}_{s}$,  $ \sqrt{2}s- X_{u}(s) = \Theta(s)$ a.s. since  $\limsup\limits_{t \to \infty}  \frac{1}{t}\max_{u \in N^{1}_{t} } |X_{u}(t) |= v$ almost surely.
Then applying part (i) of Lemma \ref{thm-Laplace-BBM-order}  we have,
uniformly  for $s\in [\frac{1}{R} \sqrt{t}, R \sqrt{t} ]$,
\begin{equation*}
 \begin{aligned}
  & F\left(t-s, X_{u}(s)-\sqrt{2}s + \frac{1}{2\sqrt{2}}\log t \right)\\
   & =(1+o(1))\gamma_{\sqrt{2}}(\varphi) \frac{(\sqrt{2}s-X_{u}(s)) }{t}e^{\sqrt{2}(X_{u}(s)-\sqrt{2}s)} e^{- \frac{(X_{u}(s)-\sqrt{2}s)^2}{2t}  } \text{ a.s.},
 \end{aligned}
\end{equation*}
as $t \rightarrow \infty$,  where the $o(1)$ term is deterministic.
Thus $  \mathbb{E}\left(e^{-\left\langle\widehat{\mathcal{E}}_t^R, \varphi\right\rangle}\right)$
\begin{align}
& =\mathbb{E}\bigg(\exp \bigg\{ - (1+o(1))\alpha\gamma_{\sqrt{2}}(\varphi)  \times \int_{\frac{1}{R} \sqrt{t}}^{R\sqrt{t}} \sum_{u \in N_s^1}   \frac{(\sqrt{2}s-X_{u}(s))}{t}e^{\sqrt{2}(X_{u}(s)-\sqrt{2}s)} e^{- \frac{(X_{u}(s)-\sqrt{2}s)^2}{2t}  }   \dif s \bigg\} \bigg) \notag \\
&=  \mathbb{E}\bigg(\exp \bigg\{ - (1+o(1))\alpha\gamma_{\sqrt{2}}(\varphi) \int_{\frac{1}{R}}^{R}   \mathcal{W}(\lambda, \lambda \sqrt{t})  \dif \lambda \bigg\} \bigg), \label{eq-exp-int-1}
\end{align}
where
\begin{equation*}
  \mathcal{W}(\lambda, t):= \sum_{u \in N^{1}_{t}} \lambda \left(\sqrt{2}- \tfrac{X_{u}(t)}{t}\right)e^{-\frac{\lambda^2}{2}(\sqrt{2}- \frac{X_{u}(t)}{t} )^2 } e^{\sqrt{2}(X_{u}(t) -\sqrt{2}t)}.
\end{equation*}

Recall that  $(\{X_{u}(t)\}_{u \in N^{1}_{t}},\mathbb{P} )$ has the same distribution as $( \{\frac{\sigma}{\sqrt{\beta}} \mathsf{X}_{u}(\beta t)\}_{u \in \mathsf{N}_{\beta t}},\mathsf{P})$. Applying  Proposition \ref{lem-functional-convergence-additive-martingale},
we have, for each $\lambda > 0$,
\begin{equation*}
  \lim_{t \to \infty} \mathcal{W}(\lambda, t) =\sqrt{2}\lambda (1-\sigma^2)e^{- (1-\sigma^2)^2\lambda^2} W^{(1)}_{\infty}(\sqrt{2}) \ \text{ in } L^{1}(\mathbb{P}),
\end{equation*}
 Then by the dominated convergence theorem, we have,  as $t\to \infty$,
 \begin{align*}
&  \E\left|  \int_{1/R}^{R} \mathcal{W}(\lambda, \lambda \sqrt{t}) \dif \lambda - \sqrt{2} \int_{1/R}^{R} \lambda (1-\sigma^2)e^{- (1-\sigma^2)^2\lambda^2}  W^{(1)}_{\infty}(\sqrt{2}) \dif \lambda  \right| \\
 & \leq    \int_{1/R}^{R}  \E\left|\mathcal{W}(\lambda, \lambda \sqrt{t})   - \sqrt{2} \lambda  (1-\sigma^2)e^{- (1-\sigma^2)^2\lambda^2}  W^{(1)}_{\infty}(\sqrt{2})  \right|\dif \lambda  \to 0 .
\end{align*}
 Here dominating functions exist  since
 $\E [\mathcal{W}(\lambda, \lambda \sqrt{t}) ]$, $ \E[\sqrt{2}\lambda   (1-\sigma^2)e^{- (1-\sigma^2)^2\lambda^2}  W^{(1)}_{\infty}(\sqrt{2})] $ are both bounded by $\max_{x \geq 0} xe^{-x^2/2}$.
Letting $t\to\infty$ in \eqref{eq-exp-int-1},
 we have
\begin{equation*}
  \lim_{t \to \infty}  \mathbb{E}\left(e^{-\left\langle\widehat{\mathcal{E}}_t^R, \varphi\right\rangle}\right)  =   \mathbb{E}\bigg(\exp \bigg\{ - \alpha\gamma_{\sqrt{2}}(\varphi) W^{(1)}_{\infty}(\sqrt{2}) \int_{\frac{1}{R}}^{R} \sqrt{2} \lambda  (1-\sigma^2)e^{- (1-\sigma^2)^2\lambda^2}   \dif \lambda \bigg\} \bigg) .
\end{equation*}
Then letting $R\to\infty$, combining \eqref{eq-E-hat-t-2-3}, we obtain
\begin{equation*}\lim_{t \to \infty}  \mathbb{E}\left(e^{-\left\langle\widehat{\mathcal{E}}_t, \varphi\right\rangle}\right)=
  \lim_{R \to \infty}  \lim_{t \to \infty}  \mathbb{E}\left(e^{-\left\langle\widehat{\mathcal{E}}_t^R, \varphi\right\rangle}\right) =
   \mathbb{E}\bigg(\exp \bigg\{  - \frac{\alpha  W^{(1)}_{\infty} (\sqrt{2})}{\sqrt{2}(1-\sigma^2)}    \gamma_{\sqrt{2}}(\varphi)  \bigg\} \bigg) ,
  \end{equation*}
  which is the Laplace functional of $ \operatorname{DPPP}\left( \frac{\alpha C_{\star} }{\sqrt{2}(1-\sigma^2)} \sqrt{2} W^{(1)}_{\infty}(\sqrt{2}) e^{-\sqrt{2} x}\dif  x  , \mathfrak{D}^{\sqrt{2}}\right)$. Using \cite[Lemma 4.4]{BBCM22}, we complete the proof of Theorem \ref{thm-extremal-process-boundary-C2C3}.
\end{proof}

Now it suffices to show Lemma \ref{lem-born-after-t1/2}. First we give  prior estimates for the type transformation times $T_{u}$
for  particles $u \in N^{2}_{t}$
with  positions higher than $m^{2.3}_{t}$.

\begin{lemma}\label{lem-born-after-t1/2logt}
For any $A > 0$ we have
 	\begin{equation}\label{eq-born-after-t1/2logt}
  \limsup _{t \rightarrow \infty} \mathbb{P}\left(\exists u \in N_{t}^{2}: T_{u} >  \sqrt{t} \log t , X_{u}(t) \geq m^{2,3}_{t}  -A \right)=0.
 \end{equation}
 \end{lemma}

 \begin{proof}
Let $Y^A_{t}:= \sum_{u \in N^{2}_{t} } 1_{ \{T_{u} >  \sqrt{t} \log t , X_{u}(t) \geq m^{2,3}_{t} -A\} } $.
By Markov's inequality, the probability in \eqref{eq-born-after-t1/2logt} is bounded by $\mathbb{E}\left(Y^A_t\right)$. It suffices to show that  $\mathbb{E}\left(Y^A_t\right) \to 0$ as $t \to \infty$.

 Using
 Lemma \ref{many-to-one-1}(i)
  and  then the tail probability of Gaussian random variable
   ($\int^\infty_xe^{-y^2/2} \dif y\leq \frac{1}{x}e^{-x^2/2} $ for $x > 0 $),
 we have
\begin{equation*}
\begin{aligned}
\mathbb{E}\left(Y^A_t\right)
&=\alpha \int_{\sqrt{t}\log t}^{t} e^{\beta s+t-s} \mathbf{P}\left(\sigma B_s+\left(B_t-B_s\right) \geq m^{2,3}_{t} -A\right) \dif s \\
& \lesssim_{\alpha} \int_{\sqrt{t}\log t}^{t} e^{\beta s+(t-s)} \frac{  \sqrt{\sigma^2 s+t-s}}{m^{2,3}_{t}-A} e^{-\frac{(m^{2,3}_{t}-A)^2}{2\left(\sigma^2 s+t-s\right)}} \dif s \\
& \lesssim_{A,\alpha,\beta,\sigma^2}  t^{-1/2}  \int_{\sqrt{t}\log t}^{t} e^{\beta s+(t-s) -  \frac{t^2}{\left(\sigma^2 s+t-s\right) }   }  e^{\frac{ t \log t  }{2\left(\sigma^2 s+t-s\right)}} \dif s .
\end{aligned}
\end{equation*}
Set  $\varphi: u \mapsto \beta u+(1-u)-\frac{1}{ \left(\sigma^2 u+1-u\right)}$. Making a change of variable $s=ut$, we have
\begin{equation*}
\mathbb{E}\left(Y^A_t\right)
\lesssim t^{1 / 2} \int_{\frac{\log t}{\sqrt{t}}}^1   \exp \left\{t \varphi(u)+\frac{   \log t  }{2\left(\sigma^2 u+1-u\right)}\right\}  \dif u  \leq  t^{ \frac{1}{2}+ \frac{1}{2 \sigma^2} }  \int_{\frac{\log t}{\sqrt{t}}}^1   e^{t \varphi(u)}   \dif u.
\end{equation*}
Since $\varphi^{\prime}(u)=\beta-1- \frac{1- \sigma^2}{ \left(\sigma^2 u+1-u\right)^2} <0$ and $ \varphi^{\prime \prime}(u)=- \frac{2\left(1-\sigma^2\right)^2}{\left(\sigma^2 u+1-u\right)^3} <0$,
 $\varphi$ is concave,
and takes maximum $\varphi(0)=0$ at point $u=0$.
By Taylor's expansion,
 there exists a constant $\delta>0$ (depending only on $\sigma^2$) such that $\varphi(u) \leq-\delta u^2$ for all $u \in[0,1]$. Thus
\begin{equation*}
\mathbb{E}\left(Y^A_t\right)
\lesssim   t^{ \frac{1}{2}+ \frac{1}{2 \sigma^2} }  \int_{\frac{\log t}{\sqrt{t}}}^1   e^{-\delta u^2 t }   \dif u
\leq   t^{ \frac{1}{2 \sigma^2} }
\int_{ \log t }^{\infty}   e^{-\delta z^2  }   \dif z  \overset{t \to \infty }{\longrightarrow} 0,
\end{equation*}
 which gives the desired result.
 \end{proof}

\begin{proof}[Proof of Lemma \ref{lem-born-after-t1/2}]
Fix  $ A,  K, t\geq 0$ and $R >1$. Set
$I^{R}_{t}:=\left[ 0, \frac{1}{R}\sqrt{t}\right) \cup \left( R \sqrt{t},  \sqrt{t} \log t \right]$, and
\begin{equation*}
Y_{t}(A, R,K )=\sum_{u \in \mathcal{B}} 1_{\left\{T_{u} \in  I^{R}_{t}  \right\}}  1_{\{ |X_{u}(T_{u})| \leq  v T_{u}+K \}}  1_{\left\{M_{t}^{u} \geq m^{2,3}_{t}-A\right\}},
\end{equation*}
where $M_{t}^{u}$ is the maximal position of the  descendants at time $t$ of the individual $u$.
In other words, $Y_{t}(A, R,K )$ is the number of type $2$ particles that are born from a type $1$ particle during the  time interval $I^{R}_{t}$ and have a descendant at time $t$ above $m^{2,3}_{t}-A$.  By Markov's inequality,
\begin{equation*}
\begin{aligned}
&  \mathbb{P}\left(\exists u \in N_{t}^{2}:  T_{u} \notin \left[ \frac{1}{R} \sqrt{t}  ,R \sqrt{t}\right], X_{u}(t) \geq m^{2,3}_{t} -A\right) \\
\leq &
\mathbb{P}\left(\exists s \leq t, u \in N^{1}_s: | X_{u}(s) | \geq v s +K\right)  \\
&+ \mathbb{P}\left(\exists u \in N_{t}^{2}: T_{u} \geq   \sqrt{t} \log t , X_{u}(t) \geq m^{2,3}_{t}  -A \right)  +\mathbb{E}\left(Y_{t}(A, R,K)\right) .
\end{aligned}
\end{equation*}
By \eqref{eq-upper-envelope-0} and Lemma \ref{lem-born-after-t1/2logt}, it suffices to show that
 $  \lim\limits_{R \to \infty} \limsup\limits _{t \rightarrow \infty} \mathbb{E}\left(Y_{t}(A, R,K)\right) =0 $.

 Let  $F_{t}(r, x):=\mathsf{P} (x+ \mathsf{M}_{r} \geq m^{2,3}_{t}-A )$, where $\mathsf{M}_{r}=\max\limits_{u \in \mathsf{N}_{r}} \mathsf{X}_{u}(r)$.
 Applying
the branching property and
Lemma \ref{many-to-one-1}(i), we have
\begin{equation}\label{eq-E-Y-t-A-R}
\begin{aligned}
&\mathbb{E}\left(Y_{t}(A, R,K)\right) =\mathbb{E}\left(\sum_{u \in \mathcal{B}}  1_{\left\{T_{u} \in I^{R}_{t}   \right\}}    1_{\{ |X_{u}(T_{u})| \leq  v T_{u}+K \}}   F_{t}\left(t-T_{u}, X_{u}(T_{u})\right)\right) \\
&=\alpha \int_{ I^{R}_{t} } e^{\beta s} \mathbf{E}\left(F_{t}\left(t-s, \sigma B_{s}\right) 1_{ \{ |\sigma B_{s}| \leq v s + K\} }  \right)   \dif s  \\
&= \alpha \int_{  I^{R}_{t} }  \mathbf{E}\left(e^{-\theta \sigma B_{s}}  F_{t}\left(t-s, \sigma B_{s} + vs \right) \right) 1_{ \{ -2 vs-K \leq \sigma B_{s} \leq  + K\} }  \dif s,
\end{aligned}
\end{equation}
where in the last equality we replaced $\sigma B_{s}$ by $\sigma B_{s}+ vs$  by  Girsanov's  theorem.
Moreover, since
 $\log t-\log (t-s) =o(1)$ for  $s \leq \sqrt{t}\log t$,
 we have  for large $t$,
\begin{align*}
  & F_{t}(t-s, vs+z ) = \mathsf{P}\left(\mathsf{M}_{t-s} \geq  \sqrt{2} t- \frac{1}{2\sqrt{2}}\log t   - A- vs-z  \right)\\
   & \leq  \mathsf{P}\left(\mathsf{M}_{t-s} \geq  \sqrt{2} (t-s)-  \frac{3}{2\sqrt{2}}\log (t-s)  + (\sqrt{2} -v)s +\frac{1}{\sqrt{2}}\log t - z-2A   \right)  .
\end{align*}
For $z \in [-2 vs-K,K ]$ and $s \leq \sqrt{t}\log t$, we have
$1 <(\sqrt{2} -v)s +\frac{1}{\sqrt{2}}\log t - z-2A \lesssim \sqrt{t}\log t $ for large $t$.
Applying Lemma \ref{lem-Max-BBM-tail}, we have
\begin{equation}\label{eq-bound-F}
  F_{t}(t-s, vs+z )
 \lesssim_{A,\beta,\sigma^2 }   \frac{  s + \log t +|z| }{ t }  e^{-\sqrt{2}(\sqrt{2}-v) s+\sqrt{2} z  }
 e^{-\frac{ ( \sqrt{2} - v )^2s^2}{2t}  },
\end{equation}
 where we used the fact that $- \frac{1}{2(t-s)} \big[( \sqrt{2} - v )s +\frac{1}{\sqrt{2}} \log t-z -2A \big]^2  \leq   -\frac{ ( \sqrt{2} - v )^2s^2}{2 t } $ for large $t$.
 Replacing $z$ by $\sigma B_{s}$ in \eqref{eq-bound-F}  and then substituting \eqref{eq-bound-F}  into \eqref{eq-E-Y-t-A-R},
we obtain
\begin{align*}
  \mathbb{E}\left(Y_{t}(A, R,K)\right) &  \lesssim_{\alpha,A,\beta,\sigma^2} \int_{  I^{R}_{t} }  \mathbf{E}\left( \frac{s+\log t+|\sigma B_{s}|}{t} e^{-(\theta-\sqrt{2}) \sigma B_{s}  -\sqrt{2}(\sqrt{2}-v) s  }  \right)  e^{-\frac{ ( \sqrt{2} - v )^2s^2}{2t}  } \dif s  \\
  & = \int_{  I^{R}_{t} } \mathbf{E}\left( \frac{s+\log t+|\sigma B_{s} -(\theta-\sqrt{2})\sigma^2 s |}{t}   \right)  e^{-\frac{ ( \sqrt{2} - v )^2s^2}{2t}  } \dif s   \\
  & \lesssim   \int_{  [0, \frac{1}{R}) \cup ( R ,   \log t] }  \left(\lambda + \frac{\log t + \mathbf{E} |B_{\lambda \sqrt{t}} |}{\sqrt{t}} \right)    e^{-\frac{ ( \sqrt{2} - v )^2 \lambda ^2}{2}  } \dif  \lambda    ,
\end{align*}
where in the equality we used
the Girsanov's theorem  and the fact that $\frac{1}{2}(\theta-\sqrt{2})^2 \sigma^2 = \sqrt{2}(\sqrt{2}-v)$ when  $\beta + \sigma^2 = 2$, and in the last $\lesssim$
we made a change of variable $s=\sqrt{t} \lambda$.
Then the desired results follows from a simple computation.
\end{proof}

\section{Boundary between $\mathscr{C}_{I}$ and $\mathscr{C}_{III}$}

In this section we always assume that
$(\beta,\sigma^2) \in \mathscr{B}_{I,III}$, i.e.,
 $\beta> 1$, $\frac{1}{\beta}+ \frac{1}{\sigma^2}=2$.  Recall that
 $m^{1,3}_{t}:= v t - \frac{1}{2 \theta} \log t$, where $v=\sqrt{2 \beta \sigma^2}$ and $\theta =\sqrt{2 \beta / \sigma^2}$.

Firstly, we describe the paths of extremal particles. It turns out that for a type $2$ particle $u \in N^{2}_{t}$ above level $m^{1,3}_{t}$ at time $t$,  its type transformation time $T_{u}$ should be $s= t- \Theta(\sqrt{t})$,  and its position $x$ at time $s$ should  belong to the set
  \begin{equation*}
   \Gamma_{s,t}^{R}:= \left\{ x :    |\delta(x;s,t)| \leq  R \sqrt{t-s}\right\} \,,  \text{ where } \   \delta(x;s,t) := x- vs + (\theta- v)(t-s).
  \end{equation*}

\begin{lemma}\label{lem-born-before-t-t1/2}
  For any $A>0$, we have
  \begin{equation}\label{eq-born-before-t-t1/2}
  \lim _{R \rightarrow \infty} \limsup _{t \rightarrow \infty} \mathbb{P}\left(\exists u \in N_{t}^{2}: t-T_{u} \notin [ \frac{1}{R} \sqrt{t}  ,R \sqrt{t}] , X_{u}(t) \geq m^{1,3}_{t}  -A \right)=0,
  \end{equation}
 and
  \begin{equation}\label{eq-born-position}
    \lim _{R \rightarrow \infty}   \limsup _{t \rightarrow \infty} \mathbb{P}\left(\exists u \in N_{t}^{2}:  X_{u}(t) \geq m^{1,3}_{t}  -A,  X_{u}(T_{u}) \notin \Gamma^{R}_{T_{u},t} \right) =0 .
    \end{equation}
  \end{lemma}

We will postpone the proof of Lemma \ref{lem-born-before-t-t1/2}
to the end of this section and show
Theorem \ref{thm-extremal-process-boundary-C1C3} first.

\begin{proof}[Proof of Theorem \ref{thm-extremal-process-boundary-C1C3}]
 In this proof,
for simplicity, for $t>0$, we set $\widehat{\mathcal{E}}_{t}:=\sum_{u \in N_t^2}   \delta_{  X_u(t)-m^{1,3}_{t} }  $,  and for  $R>0$, we set $\Omega^{R}_{t}:=\{ (s,x): t-s \in [ \frac{1}{R} \sqrt{t}  , R \sqrt{t}] ,  x \in \Gamma_{s,t}^{R}  \}$  and
 \begin{equation*}
  \widehat{\mathcal{E}}_t^R
:=\sum_{u \in N_t^2} 1_{\left\{ (T_{u},X_{u}(T_{u})) \in \Omega^{R}_{t}  \right\}}    \delta_{  X_u(t)-m^{1,3}_{t}}.
     \end{equation*}
 Take $A>0$ satisfying  $\mathrm{supp}(\varphi) \subset [-A,\infty)$.
  For all $\varphi \in \mathcal{T}$,  we have
  \begin{equation*}
   \left|\mathbb{E}\left(e^{-\left\langle\widehat{\mathcal{E}}_t^R, \varphi\right\rangle}\right)-\mathbb{E}\left(e^{-\left\langle  \widehat{\mathcal{E}}_{t}, \varphi\right\rangle}\right)\right|
     \leq
    \mathbb{P}\left(\exists u \in N_{t}^{2}, X_{u}(t) \geq m^{1,3}_{t}  -A, (T_{u}, X_{u}(T_{u})) \notin \Omega^{R}_{t}  \right).
  \end{equation*}
 Thus by  Lemma \ref{lem-born-before-t-t1/2}
we have
  \begin{equation}\label{eq-E-hat-t-1-3}
  \lim _{R \rightarrow \infty} \limsup _{t \rightarrow \infty}
  \left| \mathbb{E} \left(e^{-\left\langle\widehat{\mathcal{E}}_t^R, \varphi\right\rangle}\right)-\mathbb{E}\left(e^{-\left\langle \widehat{\mathcal{E}_{t}}, \varphi\right\rangle}\right)\right|=0.
  \end{equation}

  We now show the convergence of $\widehat{\mathcal{E}}_t^R$  first as $t$ and then $R$ goes to $\infty$.
Recall that $u^{\prime} \succcurlyeq u$ means that $u'$ is a descendant of $u$. We rewrite $ \langle\widehat{\mathcal{E}}_t^R, \varphi\rangle $ as
\begin{equation*}
 \sum\nolimits_{ \substack{u \in \mathcal{B}\\ (T_{u},X_{u}(T_{u})) \in \Omega^{R}_{t} }}    \sum\nolimits_{\substack{u^{\prime} \in N_t^2 \\ u^{\prime} \succcurlyeq u}}
   \ \varphi(X_{u'}(t)-X_{u}(T_{u})- \theta (t-T_{u})
    +  \delta(X_{u}(T_{u});T_{u},t) +\frac{1}{2 \theta} \log t ).
  \end{equation*}
    Let $f(r,x):=\mathsf{E}\left(  \exp\{-\sum_{u \in \mathsf{N}_r} \varphi (\mathsf{X}_u(r)- \theta r + x   )\}\right)$ and $F:=1-f$.
    Using the branching property first  and then applying
Lemma \ref{many-to-one-1}(iii),
we have
\begin{equation*}
  \begin{aligned}
  &\mathbb{E}\left(e^{-\left\langle\widehat{\mathcal{E}}_t^R, \varphi\right\rangle}\right)
   =  \mathbb{E}\bigg[  \prod\nolimits_{\substack{u \in \mathcal{B} \\   (T_{u},X_{u}(T_{u}))\in \Omega^{R}_{t}  }}  f\bigg(
   t-T_{u} ,   \delta(X_{u}(T_{u});T_{u},t) +\frac{1}{2 \theta} \log t  \bigg)
    \bigg]
  \\
  &=\mathbb{E}\bigg(\exp \bigg\{ -\alpha \int^{t-\frac{1}{R}\sqrt{t}}_{t- R\sqrt{t}} \sum_{u \in N_s^1}   F\bigg(t-s,  \delta(X_{u}(s);s,t) +  \frac{1}{2\theta}\log t \bigg) 1_{\{ X_{u}(s) \in \Gamma_{s,t}^{R}\}}  \dif s \bigg\} \bigg).
  \end{aligned}
  \end{equation*}
Additionally, by Lemma \ref{thm-Laplace-BBM-order} (ii),
uniformly in    $s  \in t-[\frac{1}{R}\sqrt{t}, R\sqrt{t}]$ and  in $X_{u}(s) \in \Gamma^{R}_{s,t}$,
 the function
\begin{equation*}
 \begin{aligned}
 &  F\left(t-s, \delta(X_{u}(s);s,t)+ \frac{1}{2 \theta}\log t \right) \\
  & \sim \gamma_{\theta}(\varphi)\frac{e^{-(\frac{\theta^2}{2}-1)(t-s)}}{\sqrt{t-s}}
     e^{\theta[X_{u}(s)- v s+(\theta-v)(t-s) +\frac{1}{2 \theta} \log t ]} e^{- \frac{[ X_{u}(s)- v s+(\theta-v)(t-s) +\frac{1}{2 \theta} \log t]^2}{2(t-s)}  }\\
   &\sim \gamma_{\theta}(\varphi)
 \sqrt{\frac{t}{t-s}}
   e^{\theta[X_{u}(s)- v s]} e^{- \frac{[  X_{u}(s)- v s+(\theta-v)(t-s)]^2}{2(t-s)}  } , \ \text { as } t \to \infty
 \end{aligned}
\end{equation*}
where in the last ``$\sim$''
 we used the fact that $\theta^2-\theta v-(\frac{\theta^2}{2}-1)=\frac{\beta}{\sigma^2}-2 \beta+1=0$.
Therefore, by
making a change of variable $r=t-s$, we have $\mathbb{E}\left(e^{-\left\langle\widehat{\mathcal{E}}_t^R, \varphi\right\rangle}\right)$
\begin{align}
&
=\mathbb{E}\bigg(\exp \bigg\{ - (1+o(1)) \alpha\gamma_{\theta}(\varphi)   \int_{\frac{1}{R} \sqrt{t}}^{R\sqrt{t}}  \sqrt{\frac{t}{r}}\sum_{\substack{u \in N_s^1 \\  X_{u}(t-r) \in \Gamma_{t-r,t}^{R}}  }
  e^{\theta[X_{u}(t-r)-v(t-r)]} e^{- \frac{[  X_{u}(t-r)-v(t-r)+(\theta-v)r]^2}{2r}  }    \dif r  \bigg\}    \bigg) \notag\\
&=  \mathbb{E}\bigg(\exp \bigg\{ - (1+o(1))\alpha\gamma_{\theta}(\varphi)
\int_{\frac{1}{R}}^{R}
  \frac{t}{\sqrt{r_{\lambda,t}}}
  \mathcal{W}^{G} \left( t-r_{\lambda,t}, r_{\lambda,t} \right) \dif \lambda \bigg\} \bigg),  \label{eq-exp-int-2}
\end{align}
where in
\eqref{eq-exp-int-2} we substitute $r$ by $r_{\lambda,t}=\lambda \sqrt{t}$, and
\begin{equation*}
\mathcal{W}^{G} (t, r):= \sum_{u \in N^{1}_{t}}  e^{-\theta[vt-X_{u}(t)]} G\left( \frac{vt- X_{u}(t)- (\theta-v) r}{\sqrt{r}} \right)  \,, \ G(x)= e^{- \frac{x^2}{2}} 1_{\{|x| \leq R \}} .
\end{equation*}

Recall that  $(\{X_{u}(t)\}_{u \in N^{1}_{t}},\mathbb{P} )$ has the same distribution as $( \{\frac{\sigma}{\sqrt{\beta}} \mathsf{X}_{u}(\beta t)\}_{u \in \mathsf{N}_{\beta t}},\mathsf{P})$. Applying Proposition \ref{lem-functional-convergence-derivative-martingale},
we have, for fixed $\lambda > 0$ and $r_{\lambda,t}=\lambda\sqrt{t}$,
\begin{equation*}
  \lim_{t \to \infty}   \frac{t}{\sqrt{r_{\lambda,t}}}  \mathcal{W}^{G}(t-r_{\lambda,t},r_{\lambda,t})
    =  Z^{(1)}_{\infty}     \sqrt{\frac{2}{\pi}}
   \frac{(\theta - v) }{\sigma^3}\lambda \, e^{-\frac{(\theta - v)^2}{2 \sigma^2} \lambda^2} \int_{\mathbb{R}}  e^{-\frac{x^2}{2}} 1_{\{|x| \leq R\}}  \dif x
 \end{equation*}
in probability, and by using \eqref{conv-prop2} the integral of $\frac{t}{\sqrt{r_{\lambda,t}}}  \mathcal{W}^{G}(t-r_{\lambda,t},r_{\lambda,t})$  with respect to $\lambda$ over the interval $[1/R,R]$ converges in probability to the integral of right hand side.
By  \eqref{eq-exp-int-2} and  the dominated convergence theorem, we have
\begin{equation*}
  \lim\limits_{t \to \infty}  \mathbb{E}\left(e^{-\left\langle\widehat{\mathcal{E}}_t^R, \varphi\right\rangle}\right)
=  \mathbb{E}\bigg(\exp \bigg\{ - \alpha\gamma_{\theta}(\varphi) Z^{(1)}_{\infty}     \sqrt{\frac{2}{\pi}} \int_{\frac{1}{R}}^{R}  \frac{(\theta - v) \lambda }{\sigma^3} e^{-\frac{(\theta - v)^2}{2 \sigma^2} \lambda^2} \dif \lambda  \int_{\mathbb{R}}  e^{-\frac{x^2}{2}} 1_{\{|x| \leq R\}}  \dif x
  \bigg\} \bigg).
\end{equation*}
Letting $R\to\infty$, combining
 \eqref{eq-E-hat-t-1-3},  we have
\begin{equation*}
  \lim_{t \to \infty} \mathbb{E}\left(e^{-\left\langle\widehat{\mathcal{E}}_t, \varphi\right\rangle}\right)=
\lim_{R \to \infty}  \lim_{t \to \infty}  \mathbb{E}\left(e^{-\left\langle\widehat{\mathcal{E}}_t^R, \varphi\right\rangle}\right) =   \mathbb{E}\bigg(\exp \big\{ -  \gamma_{\theta}(\varphi) Z^{(I)}_{\infty}
 \frac{ 2 \alpha }{\sigma(\theta - v)}
\big\} \bigg) ,
\end{equation*}
which is the Laplace functional of $ \mathrm{DPPP}\left( \frac{ 2 \alpha  \sigma C(\theta)}{\sqrt{2 \pi}\beta (1-\sigma^2) } \theta Z^{(1)}_{\infty}  e^{- \theta x}\dif  x, \mathfrak{D}^{\theta}\right)$. Using \cite[Lemma 4.4]{BBCM22}, we complete the proof of Theorem \ref{thm-extremal-process-boundary-C1C3}.
\end{proof}

Before proving   Lemma \ref{lem-born-before-t-t1/2}, we show a weaker result first.

\begin{lemma}
For any $A > 0$ we have
\begin{equation*}
    \limsup _{t \rightarrow \infty} \mathbb{P}\left(\exists u \in N_{t}^{2}: T_{u} \leq t-  \sqrt{t} \log t , X_{u}(t) \geq m^{1,3}_{t}  -A \right)=0 .
  \end{equation*}
  \end{lemma}

   \begin{proof}
For $A>0$, we compute the mean of
$Y^A_{t}:= \sum_{u \in N^{2}_{t} } 1_{ \{T_{u} \leq t-  \sqrt{t} \log t , X_{u}(t) \geq m^{1,3}_{t} -A\} } $.
Applying Lemma \ref{many-to-one-1}(i) and Gaussian  tail bounds, we have
  \begin{equation*}
  \begin{aligned}
\mathbb{E}\left(Y^A_t\right)
  &=\alpha\int_{0}^{t-\sqrt{t}\log t} e^{\beta s+t-s} \mathbf{P}
  \left(\sigma B_s+\left(B_t-B_s\right) \geq m^{1,3}_{t} -A\right) \dif s \\
  & \lesssim \int_{0}^{t-\sqrt{t}\log t} e^{\beta s+(t-s)} \frac{  \sqrt{\sigma^2 s+t-s}}{m^{1,3}_{t}-A} e^{-\frac{(m^{1,3}_{t}-A)^2}{2\left(\sigma^2 s+t-s\right)}} \dif s \\
  & \lesssim_{A}  t^{-1/2}  \int_{\sqrt{t}\log t}^{t} e^{\beta (t-u)+u -  \frac{ v^2 t^2}{2\left(\sigma^2 (t-u)+u\right) }   }  e^{\frac{ v t \log t  }{\theta \left(\sigma^2 (t-u)+ u\right)}} \dif u.
  \end{aligned}
  \end{equation*}
 Set $\varphi: \lambda \mapsto \beta (1-\lambda) +\lambda-\frac{\beta \sigma^2}{ \sigma^2 + (1-\sigma^2) \lambda }$. Making a change of variable $u=t\lambda$, we get
  \begin{equation*}
  \mathbb{E}\left(Y^A_t\right)
  \lesssim_{A}   t^{ \frac{1}{2}+ \frac{v}{\theta \sigma^2} }  \int_{\frac{\log t}{\sqrt{t}}}^1   e^{t \varphi(\lambda)}   \dif \lambda.
  \end{equation*}
As  $\varphi^{\prime}(\lambda)=1-\beta + \beta \sigma^2 \frac{1- \sigma^2}{ \left(\sigma^2 +(1-\sigma^2) \lambda \right)^2} <0$ and $ \varphi^{\prime \prime}(u)=-  \beta \sigma^2  \frac{2\left(1-\sigma^2\right)^2}{\left(  \sigma^2 +(1-\sigma^2) \lambda  \right)^3} <0$,
   $\varphi$ is concave,
  and takes  maximum $\varphi(0)=0$ at point $\lambda=0$.
By Taylor's expansion,
    there is a constant $\delta>0$ (depending only on $\sigma^2$) such that $\varphi(\lambda) \leq-\delta \lambda^2$ for all $\lambda \in[0,1]$. Hence
  \begin{equation*}
 \mathbb{E}\left(Y^A_t\right)
  \lesssim      t^{ \frac{1}{2}+ \frac{v}{\theta \sigma^2} }   \int_{\frac{\log t}{\sqrt{t}}}^1   e^{-\delta \lambda^2 t }   \dif \lambda   \leq
  t^{\frac{v}{\theta \sigma^2} }
    \int_{ \log t }^{\infty}   e^{-\delta z^2  }   \dif z  \overset{t \to \infty }{\longrightarrow} 0,
  \end{equation*}
  which gives the desired result.
   \end{proof}

\begin{proof}
[Proof of Lemma \ref{lem-born-before-t-t1/2}]
\textbf{Step 1.}
We first prove \eqref{eq-born-before-t-t1/2}.
As shown in the proof of Lemma \ref{lem-born-after-t1/2}, it suffices to bound the mean of
\begin{equation*}
 Y_{t}(A,K,R) := \sum_{u \in N^{2}_{t} } 1_{ \{ t-T_{u} \in I^{R}_{t}\}}1_{\{ X_{u}(r) \leq vr + K, \forall r \leq T_{u} \} }  1_{\{X_{u}(t) \geq m^{1,3}_{t} -A \}},
\end{equation*}
where  $A, R, K \geq 0$ and $I^{R}_{t}: =   [0, \frac{1}{R}\sqrt{t}) \cup ( R\sqrt{t}, \sqrt{t} \log t]$.
Applying Lemma \ref{many-to-one-1}(i), we have
  \begin{align}
   &\E[  Y_{t}(A,K,R)] = \alpha \int_{t-I^{R}_{t}} e^{\beta s+ (t-s)} \mathbf{P} \bigg(
\begin{array}{l}
      \sigma B_{s}  + B_{t}-B_{s} \geq m^{1,3}_{t} -A,\\
     \sigma B_{r} \leq vr +K , \forall r \leq s
    \end{array}
  \bigg) \dif s \notag \\
  &\lesssim \int_{t-I^{R}_{t}} \int_{-A}^{\infty} e^{\beta s+ (t-s)} \mathbf{P}\bigg(
    \begin{array}{r}
 \sigma B_{r} \leq vr +K ,\\
  \forall r \leq s
    \end{array}
    \bigg |
\begin{array}{r}
    \sigma B_{s}  + B_{t}-B_{s} \\
    = m^{1,3}_{t} + x
\end{array}
 \bigg) e^{-\frac{(m^{1,3}_{t}+x)^2}{2(\sigma^2 s + t-s)}} \frac{\dif s \dif x}{\sqrt{t}} \notag \\
 &\lesssim \int_{-A}^{\infty}\dif x\int_{I^{R}_{t}}  \frac{ u+ \log t+|x|}{t} e^{\beta (t-u)+ u}    e^{-\frac{(m^{1,3}_{t}+x)^2}{2(\sigma^2 t +(1-\sigma^2)u)}} \frac{\dif u }{\sqrt{t}} , \label{eq-bound-E-Y-AKR}
  \end{align}
where in the third line we used Lemma \ref{lem-bridge-estimate} and we substitute  $t-s$ by $u$ in the integral.
Note that $
\frac{1}{2\sigma^2 t}(m^{1,3}_{t}+x)^2
 \geq   \frac{1 }{2 \sigma^2 t} [v^2 t^2-2 v t (\frac{1}{2 \theta}\log t -x)  ]= (\beta t - \frac{1}{2} \log t + \theta x )$. Then uniformly in $u \in [0,\sqrt{t}\log t]$,
we have    $ 1- \frac{ \sigma^2 t }{\sigma^2 t +(1-\sigma^2)u}=O(\frac{\log t}{\sqrt{t}})  $,  and hence
    \begin{align}
    &\beta (t-u)+ u -\frac{(m^{1,3}_{t}+x)^2}{2(\sigma^2 t +(1-\sigma^2)u)} =  \beta (t-u)+ u -\frac{(m^{1,3}_{t}+x)^2}{2\sigma^2 t} \frac{\sigma^2 t}{\sigma^2 t +(1-\sigma^2)u} \notag \\
    &\leq   - (\beta-1)u   + \frac{1}{2} \log t+ \beta t  \frac{ (1-\sigma^2)u }{\sigma^2 t +(1-\sigma^2)u} - [\theta-o(1)] x + o(1)  \label{eq-estimate-in-exp-1}.
    \end{align}
Substituting \eqref{eq-estimate-in-exp-1} into \eqref{eq-bound-E-Y-AKR}, we have
\begin{equation*}
  \begin{aligned}
 \E[  Y_{t}(A,K,R)]  &\lesssim  \int_{-A}^{\infty}\dif x \, e^{-\theta x / 2 }
  \int_{I^{R}_{t}} \frac{ u+ \log t+|x|}{t}    e^{   \frac{ (1-\sigma^2)t }{\sigma^2 t +(1-\sigma^2)u} \beta u - (\beta - 1)u }   \dif u  \\
  &=  \int_{-A}^{\infty}\dif x \, e^{-\theta x / 2 }
  \int_{I^{R}_{t}} \frac{ u+ \log t+|x|}{t}    e^{  - \frac{ (\beta-1)(1-\sigma^2 )u^2 }{\sigma^2 t +(1-\sigma^2)u} }   \dif u,
  \end{aligned}
\end{equation*}
 where in the equality we used the fact
  $\frac{(1-\sigma^2)t }{\sigma^2 t +(1-\sigma^2)u} \beta u - (\beta - 1)u =  \frac{ - (\beta-1)(1-\sigma^2)u^2  }{\sigma^2 t +(1-\sigma^2)u } $,
which  follows from  the assumption  $\frac{1}{\sigma^2}+ \frac{1}{\beta}=2$.
 Making a change of variable $u=\lambda \sqrt{t}$, we get
  \begin{align*}
 \E[  Y_{t}(A,K,R)]
  &\lesssim  \int_{-A}^{\infty}\dif x \, e^{-\theta x / 2 }
  \int_{[0,\frac{1}{R})\cup(R,\infty)} \frac{ \lambda \sqrt{t}+ \log t+|x|}{\sqrt{t}}    e^{  - \frac{ (\beta-1)(1-\sigma^2 ) \lambda^2}{\sigma^2 + (1-\sigma^2)/\sqrt{t} } } \dif \lambda.
  \end{align*}
  By the dominated convergence theorem,  we get $\lim\limits_{R \to \infty} \limsup\limits_{t \to \infty} \E[  Y_{t}(A,K,R)] = 0$ as desired.

   \textbf{Setp 2.} Now we prove \eqref {eq-born-position} by showing    that for fixed $A, K, \epsilon > 0$, the expectation of
    \begin{equation*}
Y_{t}(A,K,\epsilon, R):= \sum_{u \in N^{2}_{t} } 1_{ \{ t-T_{u} \in [\epsilon \sqrt{t}, \epsilon^{-1}\sqrt{t}] ,  X_{u}(r) \leq vr + K, \forall r \leq T_{u} \} } 1_{\{ X_{u}(T_{u}) \notin  \Gamma_{T_{u},t}^{R} \}}   1_{\{X_{u}(t) \geq m^{1,3}_{t} -A \}}
  \end{equation*}
 vanishes as first $t \to \infty$ then $R \to \infty$.
Thanks to  Lemma \ref{many-to-one-1}(i),
\begin{equation}\label{eq-E-YtAKeR-1}
    \E[  Y_{t}(A,K,\epsilon, R)] = \alpha \int_{t-\frac{1}{\epsilon}\sqrt{t}}^{t-\epsilon\sqrt{t}} e^{\beta s+ (t-s)} \mathbf{P} \bigg(
    \begin{array}{l}
      \sigma B_{s}  + B_{t}-B_{s} \geq m^{1,3}_{t} -A, \\
      |\sigma B_{s}- vs - (\theta-v)(t-s)| > R\sqrt{t-s}, \\
     \sigma B_{r} \leq vr +K , \forall r \leq s
    \end{array}
  \bigg) \dif s.
\end{equation}
 Then conditioned on $\sigma B_{s}= vs -(\theta-v)(t-s)+ x$,
by the Markov property,
 and noting that   $m^{1,3}_{t}- [vs -(\theta-v)(t-s)+ x]= \theta(t-s)-x-\frac{1}{2 \theta}\log t$,
\begin{align}
 & \mathbf{P} \bigg(
    \begin{array}{l}
      \sigma B_{s}  + B_{t}-B_{s} \geq m^{1,3}_{t} -A, \\
      |\sigma B_{s}- vs - (\theta-v)(t-s)| > R\sqrt{t-s}, \\
     \sigma B_{r} \leq vr +K , \forall r \leq s
    \end{array}
  \bigg) \notag \\
  &\lesssim  \int\limits_{\substack{ |x|>R\sqrt{t-s} \\ x \leq (\theta-v) (t-s)+K } } \frac{\dif x}{\sqrt{ s}} e^{-\frac{(vs-(\theta-v)(t-s)+x)^2}{2 \sigma^2 s}} \mathbf{P}(B_{t-s}> \theta(t-s)-x-\frac{1}{2\theta}\log t - A) \notag \\
  &\qquad \qquad  \qquad  \qquad \quad
  \times \mathbf{P}( \sigma B_{u} \leq v u + K, \forall u \leq s  |\sigma B_{s}= vs- (\theta-v)(t-s) + x )   \notag \\
  & \lesssim   \frac{e^{  -\beta s + (\frac{\theta^2}{2}- \theta v) (t-s) } }{\sqrt{t-s}}
   \int_{|x|>R \sqrt{t-s}} \frac{t-s+|x|}{t}   e^{ - \frac{ ((\theta-v) (t-s) - x )^2 }{2 \sigma^2 s } }   e^{ -\frac{(x+ \frac{\log t}{2 \theta}+A)^2}{2 (t-s)}  }   \dif x  \label{eq-E-YtAKeR-2},
\end{align}
where the  last inequality follows from the following three items:
\begin{itemize}
  \item $\exp\left\{-\frac{(vs-(\theta-v)(t-s)+x)^2}{2 \sigma^2 s}\right\}= e^{  -\beta s +  (\theta^2- \theta v) (t-s) } e^{-\theta x}  \exp \left\{ - \frac{ ((\theta-v) (t-s) - x )^2 }{2 \sigma^2 s } \right\}.$
  \item Gaussian tail bounds yield that
    for $s \in [t-\epsilon^{-1}\sqrt{t}, t - \epsilon \sqrt{t}]$ and $x \leq  (\theta- v)(t-s)+ K$,
  \begin{equation*}
   \mathbf{P}(B_{t-s}> \theta(t-s)-x-\frac{\log t}{2\theta}  - A)  \lesssim_{A}    \frac{\sqrt{t}}{\sqrt{t-s}} e^{ - \frac{\theta^2}{2}(t-s)} e^{ \theta x } \exp \bigg\{   -\frac{(x+ \frac{\log t}{2 \theta}+A)^2}{2 (t-s)}  \bigg\}.
  \end{equation*}
  \item
  Lemma \ref{lem-bridge-estimate-0} implies that
  for $s \in [t-\epsilon^{-1}\sqrt{t}, t - \epsilon \sqrt{t}]$ and $x \leq  (\theta- v)(t-s)+ K$,
  \begin{equation*}
   \mathbf{P}( \sigma B_{u} \leq v u + K, \forall u \leq s  |\sigma B_{s}= vs- (\theta-v)(t-s) + x )  \lesssim (t-s + |x|)/ t .
  \end{equation*}
\end{itemize}
Substituting \eqref{eq-E-YtAKeR-2} into \eqref{eq-E-YtAKeR-1},
noting that
 $\frac{\theta^2}{2}-\theta v +1=  \frac{\beta}{\sigma^2}  - 2 \beta + 1 = 0$    and making change of variables $ t-s=\lambda \sqrt{t}$, $y=x/\sqrt{t-s}$,  we get
\begin{align*}
  &\E[  Y_{t}(A,K,\epsilon, R)]
  \lesssim  \int_{t-\epsilon^{-1} \sqrt{t}} ^{t-\epsilon \sqrt{t}} \dif s    \int_{|x|>R \sqrt{t-s}}  \frac{t-s+|x|}{t }   e^{ - \frac{ ((\theta-v) (t-s) - x )^2 }{2 \sigma^2 s } }    e^{  -\frac{(x+ \frac{\log t}{2 \theta}+A)^2}{2 (t-s)}  }   \frac{\dif x }{\sqrt{t-s}} \\
  &=  \int_{\epsilon  } ^{\epsilon^{-1} }   \dif \lambda   \int_{|y|>R  } \frac{\lambda \sqrt{t}+t^{1/4} \lambda^{1/2}|y| }{\sqrt{t}  }    e^{ - \frac{ [(\theta-v)  \lambda - y \lambda^{1/2} t^{-1/4} ]^2 }{2 \sigma^2 (1-\lambda t^{-1/2}) } }   e^{  -\frac{1}{2  }\left(y + \frac{\log t /(2 \theta) + A}{  \lambda^{1/2} t^{1/4}}\right)^2  }    \dif y.
\end{align*}
Letting  $t \to \infty$, we have
$ \limsup\limits_{ t \to \infty} \E[  Y_{t}(A,K,\epsilon,R)] \lesssim \int_{\epsilon  } ^{\epsilon^{-1} } \lambda e^{ - \frac{ (\theta-v)^2  \lambda^2   }{2 \sigma^2   } }\dif \lambda   \int_{|y|>R  }       e^{  -\frac{1}{2  } y^2  }    \dif y $
by the dominated convergence theorem.   Then  letting  $R \to \infty$, the desired result follows.
\end{proof}

\section{Boundary between $\mathscr{C}_{I}$ and $\mathscr{C}_{II}$}
In this section we always assume that $(\beta,\sigma^2) \in \mathscr{B}_{I,II}$, i.e., $\beta<1$ and $\beta \sigma^2=1$. Recall that  $m^{1,2}_{t}=m_{t}=\sqrt{2}t-\frac{3}{2\sqrt{2}}\log t$.
As we have remarked, the driven mechanism of the asymptotic behavior of the extremal particles in this setting is the same as in the case $(\beta,\sigma^2) \in \mathscr{C}_{II}$.
The outline of the proof is very similar to that of \cite[Theorem 1.2]{BM21}, but   more careful estimations are needed.

Firstly we show that each particle $u \in N^{2}_{t}  $ above level $m^{1,2}_{t}$ at time $t$ should satisfy   $T_{u} =   O(1)$.

\begin{lemma}\label{lem-Tu-less-than-R}
 For any $A>0$, we have
  \begin{equation*}
 \lim _{R \rightarrow \infty} \limsup _{t \rightarrow \infty} \mathbb{P}\left(\exists u \in N_{t}^{2}: T_{u} \geq R , X_{u}(t) \geq m^{1,2}_{t}  -A \right)=0 .
  \end{equation*}
\end{lemma}

We postpone the proof of Lemma \ref{lem-Tu-less-than-R} in the end of this section.
Lemma \ref{lem-Tu-less-than-R} implies that
we can approximate $\widehat{\mathcal{E}}_{t}:=\sum_{u \in N_t^2}   \delta_{  X_u(t)-m^{1,2}_{t} }  $  by  $
\widehat{\mathcal{E}}_t^R:=\sum_{u \in N_t^2} 1_{\{T_{u} \leq R\}} \delta_{X_u(t)-m^{1,2}_{t}}$. Indeed  we have
  \begin{equation*}
  \lim _{R \rightarrow \infty} \limsup _{t \rightarrow \infty}
  \left| \mathbb{E} \left(e^{-\left\langle\widehat{\mathcal{E}}_t^R, \varphi\right\rangle}\right)-\mathbb{E}\left(e^{-\left\langle \widehat{\mathcal{E}_{t}}, \varphi\right\rangle}\right)\right|=0.
  \end{equation*}
By the results in \cite{ABBS13,ABK13},
a simple computation  gives the convergence of $\widehat{\mathcal{E}}_t^R$
 as $t$  to $\infty$.
 We only restate the results here,  see \cite[Lemma 5.2]{BM21} for a proof.

 \begin{lemma}
 For any $\varphi \in \mathcal{T}$,
   we have $\lim_{t \rightarrow \infty} \langle\widehat{\mathcal{E}}_t^R, \varphi \rangle= \langle\widehat{\mathcal{E}}_{\infty}^R, \varphi \rangle$ in law, where
  \begin{equation*}
    \widehat{\mathcal{E}}_{\infty}^R=\mathrm{DPPP}\left( C_{\star} \bar{Z}_R \sqrt{2} e^{-\sqrt{2} x} \dif x,\mathfrak{D}^{\sqrt{2}} \right).
  \end{equation*}
 \end{lemma}
Here $ \bar{Z}_R$ is defined as follows.
 Firstly,  for each $u \in \mathcal{B}$,  the convergence of derivative martingale for the standard BBM and the branching property implies the a.s. convergence of
  \begin{equation*}
Z_t^{(u)}:=\sum_{\substack{u^{\prime} \in N_t^2 \\ u^{\prime} \succcurlyeq u}}\left(\sqrt{2} t-X_{u^{\prime}}(t)\right) e^{\sqrt{2}\left(X_{u^{\prime}}(t)-\sqrt{2} t\right)}.
\end{equation*}
Let $Z_{\infty}^{(u)}:=\liminf\limits_{t \rightarrow \infty} Z_t^{(u)}$.   Then $Z_{\infty}^{(u)} \overset{\mathrm{law}}{=} \mathsf{Z}_{\infty} e^{\sqrt{2} X_{u}(T_{u}) - 2 T_{u} } >0$ a.s., where $\mathsf{Z}_{\infty}$ is independent of $\{X_{u}(t):u\in N_{t}\}_{t \geq 0}$. We set $\bar{Z}_R:=\sum_{u \in \mathcal{B}} 1_{\{T_{u} \leq R\}} Z_{\infty}^{(u)}$.
The main ingredient is to show that
$\widehat{\mathcal{E}}_{\infty}^R$ converges in law as $R \rightarrow \infty$, which is done as follows.

\begin{lemma}\label{lem-bar-Z-infinity}
   For all $\varphi \in \mathcal{T}$, we have $\lim _{R \rightarrow \infty}\langle\widehat{\mathcal{E}}_{\infty}^R, \varphi\rangle=\langle\widehat{\mathcal{E}}_{\infty}, \varphi\rangle$ in law, where
   \begin{equation*}
    \widehat{\mathcal{E}}_{\infty} =\mathrm{DPPP}\left( C_{\star} \bar{Z}_{\infty} \sqrt{2} e^{-\sqrt{2} x} \dif x,\mathfrak{D}^{\sqrt{2}} \right),
   \end{equation*}
    and  $\bar{Z}_{\infty}:=\lim\limits_{R \rightarrow \infty} \bar{Z}_R=\sum_{u \in \mathcal{B}} Z_{\infty}^{(u)} <\infty$ a.s..
\end{lemma}

\begin{proof}
  As shown in the proof of  \cite[Lemma 5.3]{BM21}, it is sufficient  to prove that
  \begin{equation*}
    Y:= \sum_{u \in \mathcal{B}} \left(1+\left( \sqrt{2}T_{u}-  X_u(T_{u})\right)_{+}\right) e^{\sqrt{2} X_u(T_{u})-2 T_{u}}<\infty \quad
    \P\text{-a.s. }
    \end{equation*}

We claim that
it is enough   to show that for each $K \in \mathbb{N}$,
\begin{equation*}
  Y_{K}= \sum_{ \substack{ u \in \mathcal{B} \\ X_{u}(s) \leq \sqrt{2}s+ \sigma K ,   s \leq T_{u}}}
    (1+ [ \sqrt{2}T_{u}-  X_u(T_{u}) ]_{+} ) e^{\sqrt{2} X_u(T_{u})-2 T_{u}}<\infty \quad
     \P\text{-a.s. }
  \end{equation*}
In fact, as $\beta\sigma^2=1$, by \eqref{eq-upper-envelope-0} we have
$  \P\left( \cup_{K \in \mathbb{N}} \{  \forall \, t \geq 0, u \in N^{1}_{t}, X_{u}(t) \leq \sqrt{2} t + \sigma K \}   \right) = 1$.
Then for  almost every realization $\omega$, we can find $K=K(\omega) \in \mathbb{N}$ large enough so that
\begin{equation*}
 X_{u}(t) \leq \sqrt{2} t + \sigma K  \, , \ \forall t \geq 0, u \in N^{1}_{t} .
\end{equation*}
As a consequence we have $Y(\omega)= Y_{K(\omega)}(\omega)< \infty$ for almost all $\omega$.

Next we compute $\E(Y_K)$ for fixed $K \in \mathbb{N}$. Applying Lemma \ref{many-to-one-1}(ii), we have
\begin{equation*}
\E(Y_{K}) = \alpha \int_{0}^{\infty} e^{\beta s} \mathbf{E}\left(   1_{\{\sigma B_{u} \leq \sqrt{2}u+ \sigma K, u \leq s \}} [1+(\sqrt{2}s-\sigma B_{s})_+] e^{\sqrt{2}(\sigma B_{s}-\sqrt{2}s)} \right) \dif s .
\end{equation*}
By
Girsanov's theorem, we replace  $(\sigma B_{u}-\sqrt{2}u)_{u \leq s}$ by $(\sigma B_{u})_{u \leq s}$
 (noting $\beta \sigma^2=1$)
 and get
\begin{align*}
  \E(Y_{K}) & \leq  \alpha \int_{0}^{\infty}  \mathbf{E}\left(   1_{\{  B_{u}  \leq  K, u \leq s \}} (1+\sigma |B_{s}|) e^{\sqrt{2}(\sigma-\sqrt{\beta}) B_{s} } \right) \dif s \\
  &\leq   \alpha \int_{0}^{\infty}  \mathbf{E}\left(  \sum_{n \geq 0} 1_{\{  B_{u}  \leq  K, u \leq s; B_{s}-K \in [-n-1,-n] \}} (1+ \sigma |B_{s}|) e^{\sqrt{2}(\sigma-\sqrt{\beta}) B_{s} } \right) \dif s  .
\end{align*}
Thanks to  the Brownian estimate \eqref{eq-brownian-estimate} and
noticing that $\sigma=\frac{1}{\sqrt{\beta}}>\sqrt{\beta}$, we have
 \begin{equation*}
   \E(Y_{K})
   \lesssim  \alpha \int_{0}^{\infty} \frac{K}{(1+s)^{3/2}}  \sum_{n \geq 0} (1+K+n)^{2} e^{\sqrt{2}(\sigma-\sqrt{\beta}) (K-n)}    \dif s  < \infty,
 \end{equation*}
which implies $Y_{K}<\infty$, $\P$-a.s.
We complete the proof.
\end{proof}

Now it suffices to show Lemma \ref{lem-Tu-less-than-R}.

\begin{proof}[Proof of Lemma \ref{lem-Tu-less-than-R}]
The proof consists of two step. For convenience we write $\varepsilon(t)=(\log t)^3$.
Firstly we show that
 \begin{equation}\label{eq-Tu-less-t-polylogt}
   \limsup _{t \rightarrow \infty} \mathbb{P}\left(\exists u \in N_{t}^{2}: T_{u} \geq t- \varepsilon(t) , X_{u}(t) \geq m^{1,2}_{t}  -A \right)=0 ;
     \end{equation}
and secondly,  we prove
  \begin{equation}\label{eq-Tu-greater-R}
 \lim _{R \rightarrow \infty} \limsup _{t \rightarrow \infty} \mathbb{P}\left(\exists u \in N_{t}^{2}: T_{u} \in [ R  ,t- \varepsilon(t)] , X_{u}(t) \geq  m^{1,2}_{t}   -A \right)=0 .
  \end{equation}

{\bf Step 1} (Proof of \eqref{eq-Tu-less-t-polylogt}).
Recall that  $a_t=\frac{3}{2 \theta} \log (t+1)$, $t\geq 0$.
As shown in the proof of Lemma \ref{lem-born-after-t1/2}, thanks to the  inequality \eqref{eq-upper-envelope}, it suffices to show that for each $K \geq 1$,  the mean of
\begin{equation*}
  Y_{t}(K) := \sum_{u \in N^2_{t}} 1_{\{ X_{u}(t)\geq
  m^{1,2}_{t} -A\}}
  1_{\{T_{u}>t-\varepsilon(t)\}}1_{\{X_{u}(r) \leq \sqrt{2}r-a_{t}+a_{t-r}+K,\, r \leq T_{u}\}}
\end{equation*}
vanishes as $t \to \infty$.
Applying  Lemma \ref{many-to-one-1}(i) and the Markov property for Brownian motion,
\begin{align*}
  &\E[ Y_{t}(K) ]
  = \alpha \int_{t- \varepsilon(t)}^{t} e^{\beta s + t-s}  \mathbf{P}\left(\begin{array}{l}
   \sigma B_{s}+B_{t}-B_{s} \geq
    m^{1,2}_{t}-A \\
 \sigma B_{r} \leq  \sqrt{2}r-a_{t}+a_{t-r}+K, r \leq s
  \end{array}\right)
  \dif s \\
  &=  \alpha \int_{t- \varepsilon(t)}^{t} e^{\beta s }  \mathbf{E} \left(   F(t-s, \sigma B_{s}-\sqrt{2}s+ \frac{3}{2\sqrt{2}}\log t + A)   1_{\{ \sigma B_{r} \leq  \sqrt{2}r-a_{t}+a_{t-r}+K, r \leq s  \}}  \right)
   \dif s,
\end{align*}
where $F(r,x)=e^{r} \mathbf{P}(B_{r} \geq \sqrt{2}r-x )$. By Markov's inequality,   $F(r,x)\leq e^{\sqrt{2}x}$.
By Girsanov's theorem, we  replace $\sigma B_{s}-\sqrt{2}s$ by $\sigma B_{s}$, and then the integral above equals to
\begin{equation*}
  \int_{t- \varepsilon(t)}^{t} \mathbf{E} \left( e^{-\sqrt{2\beta}B_{s}}  F(t-s, \sigma B_{s} + \frac{3}{2\sqrt{2}}\log t + A)   1_{\{ \sigma B_{r} \leq  a_{t-r}-a_{t}+K, r \leq s  \}}  \right) \dif s .
\end{equation*}
Thus
 \begin{equation*}
\E[ Y_{t}(K) ] \lesssim_{A,\alpha}
t^{3/2} \int_{t- \varepsilon(t)}^{t} \mathbf{E}
\left( e^{\sqrt{2} (\sigma-\sqrt{\beta})B_{s} }
1_{\{ \sigma B_{r} \leq  a_{t-r}-a_{t}+K, r \leq s  \}}  \right)
\dif s .
 \end{equation*}
By  inequality (6.3) in \cite{BM21}
with $\lambda=\sqrt{2}-\theta$
(or making a change of measure to Bessel-3 process),  we have
\begin{equation*}
  \mathbf{E}\left( e^{\sqrt{2} (\sigma-\sqrt{\beta})B_{s} }
1_{\{ \sigma B_{r} \leq  a_{t-r}-a_{t}+K,\, r \leq s  \}}  \right)
 \leq
C e^{(\sqrt{2}-\theta) K}\left(\frac{t-s+1}{t+1}\right)^{\frac{3(\sqrt{2}-\theta)}{2 \theta}}(s+1)^{-\frac{3}{2}} .
\end{equation*}
Since now $\theta < \sqrt{2}$, a simple computation yields
\begin{equation*}
  \E[ Y_{t}(K) ] \lesssim _{A, \alpha, K}  t^{3/2} \int_{t- \varepsilon(t)}^{t}  \left(\frac{t-s+1}{t+1}\right)^{\frac{3(\sqrt{2}-\theta)}{2 \theta}}(s+1)^{-\frac{3}{2}}
  \dif s  \overset{t \to \infty}{\longrightarrow} 0 .
\end{equation*}

{\bf Step 2} (Proof of \eqref{eq-Tu-greater-R}).
As in Step 1, it suffices to show that for each  $A, R, K>0$, the expectation of
\begin{equation*}
 Y_{t}(A,R,K):= \sum_{u \in \mathcal{B}} 1_{\{T_{u}\in [R, t- \varepsilon(t)]\}}
  1_{\{ M^{u}_{t} \geq m^{1,2}_{t}-A\}} 1_{\{X_{u}(r) \leq \sqrt{2}r +\sigma K, r \leq T_{u}\}}
\end{equation*}
converges to $0$ as first $t \to \infty$ then $R \to \infty$.
Let $F_{t}(r,x):=\mathsf{P}(x+\mathsf{M}_{r} > m^{1,2}_{t}-A)$, where  $\mathsf{M}_{r}$ is the maximum at time $r$ of a standard BBM.
By the branching  Markov property, Lemma \ref{many-to-one-1}(ii),
 and Girsanov's theorem
\begin{align*}
	\E[Y_{t}(A,R,K)]&= \E\left[ \sum_{u \in \mathcal{B}}  1_{\{T_{u}\in [R, t- \varepsilon(t)]\}}
	1_{\{X_{u}(r) \leq \sqrt{2}r +\sigma K, r \leq T_{u}\}} F_t(t-T_{u}, X_{u}(T_{u})) \right]\\
	&= \alpha \int_{R}^{t-\varepsilon(t)} e^{\beta s} \mathbf{E}\left[  F_t(t-s, \sigma B_{s}) 1_{\{B_{r} \leq \sqrt{2 \beta} r + K, r \leq s \} } \right] \dif s \\
	&=  \alpha \int_{R}^{t-\varepsilon(t)} \mathbf{E
	}\left[  e^{- \sqrt{2\beta} B_s} F_t(t-s, \sigma B_{s}+\sqrt{2}s) 1_{\{B_{r} \leq   K, r \leq s \} } \right] \dif s .
\end{align*}
Applying \eqref{eq-Max-BBM-tail-2}, we have  for all $s \leq t-\varepsilon(t)$ and $x \in \mathbb{R}$,
\begin{equation*}
\begin{aligned}
F_{t}(t-s, x) &=\mathsf{P}\left(\mathsf{M}_{t-s} \geq
m^{1,2}_{t-s}
+\sqrt{2} s+\frac{3}{2 \sqrt{2}} \log \frac{t-s}{t} -A-x\right) \\
& \lesssim_{A}   \left(\frac{t}{t-s}\right)^{\frac{3}{2}}\left(1+ |\sqrt{2} s -x| + \log\frac{t}{t-s} \right) e^{-\sqrt{2}(\sqrt{2} s-x)} ,
\end{aligned}
\end{equation*}
We now  get upper bound the expectation:
\begin{align*}
  \E[Y_{t}(A,R,K)]
     \lesssim_{A}
  \alpha\int_{R}^{t-\varepsilon(t)} \left(\frac{t}{t-s}\right)^{\frac{3}{2}} \mathbf{E}\left[  \left(1+ |B_{s}| + \log\frac{t}{t-s}\right)  e^{\sqrt{2} (\sigma-\sqrt{\beta})B_{s}} 1_{\{B_{r} \leq K, r \leq s\}} \right] \dif s \,.
\end{align*}
Applying \eqref{eq-brownian-estimate}, as in the proof of Lemma \ref{lem-bar-Z-infinity}, we have
\begin{align*}
  & \mathbf{E}\left[  \left(1+ |B_{s}| + \log\frac{t}{t-s}\right)  e^{\sqrt{2} (\sigma-\sqrt{\beta})B_{s}} 1_{\{B_{r} \leq K, r \leq s\}} \right] \\
  & \leq \sum_{n\geq 0} \left(2+K+n+ \log\frac{t}{t-s} \right)   e^{\sqrt{2} (\sigma-\sqrt{\beta})(K-n)}  \frac{CK n}{s^{3/2}} \lesssim_{\beta,K}  \frac{ (1+ \log\frac{t}{t-s}) }{s^{3/2}} .
\end{align*}
Thus
\begin{align*}
 &  \E[Y_{t}(A,R,K)]
  \lesssim_{\alpha,\beta, A,K} \int_{R}^{t-\varepsilon(t)} \left(\frac{t}{t-s}\right)^{\frac{3}{2}}   \frac{ (1+ \log\frac{t}{t-s}) }{s^{3/2}}  \dif s \\
  &= \left( \int_{R}^{t/2}   +  \int_{t/2}^{t-\varepsilon(t)} \right) \left(\frac{t}{t-s}\right)^{\frac{3}{2}}   \frac{ (1+ \log\frac{t}{t-s}) }{s^{3/2}}  \dif s =: (I)+ (II).
\end{align*}
For (I), since $\frac{t}{t-s} \leq 2$ when $s \in [R,t/2]$, we have
\begin{equation*}
  \int_{R}^{t/2} \left(\frac{t}{t-s}\right)^{\frac{3}{2}}   \frac{ (1+ \log\frac{t}{t-s}) }{s^{3/2}}  \dif s \leq 100  \int_{R}^{\infty}   \frac{  1 }{s^{3/2}}  \dif s \overset{R \to \infty}{\longrightarrow} 0 .
\end{equation*}
For (II), making a change of variable  $t-s = \lambda t$, we have
\begin{align*}
  & \int_{t/2}^{t-\varepsilon(t)} \left(\frac{t}{t-s}\right)^{\frac{3}{2}}   \frac{ (1+ \log\frac{t}{t-s}) }{s^{3/2}}  \dif s  = \frac{1}{\sqrt{t}} \int_{\varepsilon(t)/t}^{1/2}  \frac{1}{\lambda^{3/2}} \left(1+\log \frac{1}{\lambda} \right)  \frac{1}{(1-\lambda)^{3/2}}  \dif \lambda \\
  & \lesssim   \frac{1+\log t }{\sqrt{t}}  \int_{\varepsilon(t)/t}^{1/2}  \frac{1}{\lambda^{3/2}}  \dif \lambda    \lesssim \frac{1+\log t}{\sqrt{t}} \sqrt{ \frac{t}{\varepsilon(t)}  } \overset{t \to \infty}{\longrightarrow} 0  .
\end{align*}
 We now complete the proof.
\end{proof}

\section{Spine decomposition and Proofs of Propositions \ref{lem-functional-convergence-additive-martingale} and \ref{lem-functional-convergence-derivative-martingale}}\label{sec-Proof-spine-decomposition}

\subsection{Proof of Proposition \ref{lem-functional-convergence-additive-martingale}}\label{spine1}
\begin{proof}
  It is well-known that, for $\eta \in [0, \sqrt{2})$, the additive martingale
  \begin{equation*}
    \mathsf{W}_{t}(\eta) = \sum_{u \in \mathsf{N}_{t}}  e^{\eta \mathsf{X}_{u}(t)- (\eta^2/2+1) t }
  \end{equation*}
  is uniformly integrable and converges to a non-trivial limit $\mathsf{W}_{\infty}(\eta)$ as $t\to\infty$.
Moreover, the law of the BBM under the  probability measure tilted by $\mathsf{W}_{t}(\eta)$, can be described naturally by introducing a spine (see for example \cite{Kyprianou03}).
We identify the spine as a distinguished genealogical line of descent from the initial ancestor and denote the individual on the spine at time $t$ by $\xi_{t}$.
More precisely, 
there exists a branching diffusion $(\{\mathsf{X}_{u}(t):u \in \mathsf{N}_{t}, \xi_{t}\}_{t\geq 0}, \widehat{\mathsf{P}})$ with  distinguished and randomized spine,  defined on some enlarged probability space $(\widehat\Omega, \widehat{\mathcal{F}}, \widehat{\mathsf{P}})$, having the following properties:
\begin{enumerate}[(i)]
\item the diffusion along the spine begins from the origin of space at time $0$ and moves according to a Brownian motion with drift $\eta$,
\item  points of fission along the spine are independent of its motion and occur with  rate $2$,
\item  at each fission time of the spine, the spine gives birth to 2 offspring, and the spine is chosen randomly so that at each fission point, the next individual to represent the spine is chosen
with uniform probability from the two offsprings,
\item offspring of individuals on the spine which are not part of the spine initiate $\mathsf{P}$-branching Brownian motions
at their space-time point of creation.
\end{enumerate}
Let $\mathcal{F}_{t}$ be the $\sigma$-algebra generated by the BBM up to time $t$, i.e., $\mathcal{F}_{t}=\sigma( \{\mathsf{X}_{u}(s): s \leq t, u \in \mathsf{N}_{t}\})$. It was established in \cite{Kyprianou03} that
  \begin{equation*}
  \dif \widehat{\mathsf{P}}|_{\mathcal{F}_{t}} = \mathsf{W}_{t}(\eta) \dif \mathsf{P}|_{\mathcal{F}_{t}} \  \text{ and } \
    \widehat{\mathsf{P}}( \xi_{t}= u | \mathcal{F}_{t}  ) =   \frac{ e^{\eta \mathsf{X}_{u}(t)- (\eta^2/2+1) t } }{\mathsf{W}_{t}(\eta)} 1_{\{u \in \mathsf{N}_{t} \}},
  \end{equation*}
Denote $\Xi(t):= \mathsf{X}_{\xi_{t}}(t)$, i.e., $\Xi_{t}$ is the position of the individual on the spine at time $t$.
   We now have, for every $t\geq 0$ and measurable function $f$,
  \begin{equation*}
    \begin{aligned}
  \widehat{\mathsf{E}}\left[  f\left(\frac{\Xi(t)}{t}\right)\big| \mathcal{F}_{t} \right]
  & =   \widehat{\mathsf{E}}\left[ \sum_{u \in \mathsf{N}_{t}}  f\left(\frac{\mathsf{X}_{u}(t)}{t}\right) 1_{\{\xi_{t}=u\}} \big| \mathcal{F}_{t} \right]  \\
  &
  = \sum_{u \in \mathsf{N}_{t}} f\left(\frac{\mathsf{X}_{u}(t)}{t}\right)  \frac{ e^{\eta \mathsf{X}_{u}(t)- (\eta^2/2+1) t } }{\mathsf{W}_{t}(\eta)} = \frac{W^{f}_{t}(\eta)}{W_{t}(\eta)} .
    \end{aligned}
  \end{equation*}
  Since under $\widehat{\mathsf{P}}$,  $\Xi$ is a Brownian motion with drift $\eta$, and
  the function $x \mapsto f(x)$ is bounded continuous,
  $f(\frac{\Xi(t)}{t})$ converges to $f(\eta) $ in $L^{1}(\widehat{\mathsf{P}})$. Thanks to the  Jensen's inequality  we have
  \begin{equation*}
   \lim_{t \to \infty} \frac{W^{f}_{t}(\eta)}{W_{t}(\eta)}= \lim_{t \to \infty} \widehat{\mathsf{E}}\left[  f\left(\frac{\Xi(t)}{t}\right) \Big| \mathcal{F}_{t} \right] \to f(\eta) \text{ in } L^{1}(\widehat{\mathsf{P}}),
  \end{equation*}
  Thus
  \begin{equation*}
  \mathsf{E}\left[   |W^{f}_{t}(\eta)-f(\eta) W_{t}(\eta) | \right]= \widehat{\mathsf{E}} \bigg[\bigg|\frac{W^{f}_{t}(\eta)}{W_{t}(\eta)} -f(\eta)\bigg| \bigg] \to 0 ,
  \end{equation*}
  yielding the desired result.
  \end{proof}

\subsection{Proof of Proposition \ref{lem-functional-convergence-derivative-martingale}}

The proof is inspired by \cite{AS14}.
Fix an arbitrarily $K>0$.
Define the \textit{truncated derivative martingale }
 $\mathsf{D}^{(K)}_{t}$ as the following:
\begin{equation}\label{eq-truncated derivative martingale}
\mathsf{D}^{(K)}_{t}:= \sum_{u \in \mathsf{N}_{t}} (\sqrt{2} t- \mathsf{X}_{u}(t)+ K)  e^{-\sqrt{2}(\sqrt{2} t-\mathsf{X}_{u}(t))}   1_{\{ \sqrt{2} s- \mathsf{X}_{u}(s) \geq - K,\, \forall s \leq t\}} .
\end{equation}

It was proved in \cite{Kyprianou03}
that $(\mathsf{D}^{(K)}_{t})_{t>0}$ is a uniformly integrable martingale  and $ \lim\limits_{t \to \infty} \mathsf{D}^{(K)}_{t} = \mathsf{Z}_{\infty} $ a.s. on the event $\{  \mathsf{X}_{u}(t) \leq \sqrt{2} t+ K, \forall t  \geq 0, u \in \mathsf{N}_{t} \}$, whose probability tends to $1$ when $K\to \infty$.
Moreover, using this martingale to do a change of measure,
we can also  get a  ``spine decomposition" of the BBM:
Let $\mathcal{F}_{t}$ be the $\sigma$-algebra generated by the branching Brownian motion up to time $t$.
Let  $\mathsf{Q}^{(K)}$ be the probability measure such that
\begin{equation}\label{eq-DK-change-of-measure}
\dif \mathsf{Q}^{(K)}|_{\mathcal{F}_{t}} = \frac{\mathsf{D}^{(K)}_{t}}{K} \cdot \dif \mathsf{P}|_{\mathcal{F}_{t}}.
\end{equation}
Similar to the spine decomposition in Subsection \ref{spine1}, we can identify distinguished genealogical lines of descent from the initial ancestor each of which
shall be referred to as a spine.
Denoting by $\xi_{t}$ the individual belonging to the spine, and by $\Xi(t)= \mathsf{X}_{\xi_{t}}(t)$ the position of this individual.
The process $(\mathsf{X}_{u}(t):u \in \mathsf{N}_{t})_{t\geq 0}$ under  $\mathsf{Q}^{(K)}$ has similar properties as in Subsection \ref{spine1} with (i)  changed to (i') below and other items (ii) to (iv) unchanged:
\begin{enumerate}
  \item[(i')] The spatial motion $(\Xi(t))_{t\geq 0}$ of the  spine is such that $R_{K}(t):= \sqrt{2} t- \Xi(t)+ K$ is a $3$-dim Bessel process started at $K$.
\end{enumerate}
For any $t>0$ and any $u \in \mathsf{N}_{t}$,  we have
\begin{equation}\label{eq-conditional-probability-spine-is-u}
\mathsf{Q}^{(K)}(\xi_{t}= u | \mathcal{F}_{t}) = \frac{  (\sqrt{2} t- \mathsf{X}_{u}(t)+ K)  e^{-\sqrt{2}(\sqrt{2} t-\mathsf{X}_{u}(t))}  }{\mathsf{D}^{(K)}_{t}}    1_{\{ \sqrt{2}s- \mathsf{X}_{u}(s) \geq - K, \,\forall s \leq t\}}.
\end{equation}

Throughout  this section, we use $(R_{t}, \mathbf{P}^{\mathrm{Bes}}_{x})$ to denote a Bessel-$3$ process  starting from $x$. The expectation with respect to $\mathbf{P}^{\mathrm{Bes}}_{x}$ is denoted as $\mathbf{E}^{\mathrm{Bes}}_{x}$.
We truncate $\mathsf{W}_{t}^{F_{t}}(\sqrt{2})$ as in \eqref{eq-truncated derivative martingale},
 that is to say, we  define
\begin{equation}\label{def-W-K-F-t-sqrt2}
  \mathsf{W}^{(K),F_{t}}_{t}(\sqrt{2}):= \sum_{u \in \mathsf{N}_{t}} F_{t}\left( \frac{\sqrt{2} t- \mathsf{X}_{u}(t)}{\sqrt{t}}  \right)e^{-\sqrt{2}(\sqrt{2} t-\mathsf{X}_{u}(t))}   1_{\{ \sqrt{2} s- \mathsf{X}_{u}(s) \geq - K,\, \forall s \leq t\}} .
\end{equation}

\begin{proposition}\label{prop-convergence-in-L^2-Q}
For each fixed $K \in \mathbb{N}$ and $\lambda>0$,
\begin{equation}\label{eq-convergence-in-L^2-Q}
\mathsf E_{\mathsf{Q}^{(K)}}
  \bigg| \frac{\sqrt{t}}{\langle F_{t}, \mu \rangle}  \frac{\mathsf{W}^{(K),F_{t}}_{t}(\sqrt{2}) }{\mathsf{D}^{(K)}_{t}} - \sqrt{\frac{2}{\pi}}\bigg|  \to 0, \quad \text{ as } t \to \infty .
\end{equation}
\end{proposition}

Proposition \ref{lem-functional-convergence-derivative-martingale} is  an easy consequence of Proposition \ref{prop-convergence-in-L^2-Q}.

\begin{proof}[Proof of  Proposition \ref{lem-functional-convergence-derivative-martingale} admitting Proposition \ref{prop-convergence-in-L^2-Q}]
	
(i)
Applying
  \eqref{eq-DK-change-of-measure} and \eqref{eq-convergence-in-L^2-Q},
 we have, for any $K>0$,
\begin{equation*}
\mathsf E  \bigg|\frac{\sqrt{t}}{\langle F_{t}, \mu \rangle} \mathsf{W}^{(K),F_{t}}_{t}(\sqrt{2}) -  \sqrt{\frac{2}{\pi}} \mathsf{D}^{(K)}_{t} \bigg|    = K \, \mathsf E_{\mathsf{Q}^{(K)}}   \bigg|\frac{\sqrt{t}}{\langle F_{t}, \mu \rangle}  \frac{\mathsf{W}^{(K),F_{t}}_{t}(\sqrt{2}) }{\mathsf{D}^{(K)}_{t}} - \sqrt{\frac{2}{\pi}} \bigg|  \  \overset{t \to \infty}{\longrightarrow} 0.
\end{equation*}
Since $\mathsf{D}^{(K)}_{t}$ is a uniformly integrable martingale, writing $\mathsf{D}^{(K)}_{\infty}:=\lim_{t}\mathsf{D}^{(K)}_{t} $, we have $\mathsf{D}^{(K)}_{t}
\overset{t \to \infty}{\longrightarrow}
\mathsf{D}^{(K)}_{\infty} $ in $L^{1}(\mathsf{P})$. Combined with the equation above, we have
\begin{equation}\label{eq-exp-to-0}
\mathsf E\left[  \bigg|  \frac{\sqrt{t}}{\langle F_{t}, \mu \rangle} \mathsf{W}^{(K),F_{t}}_{t}(\sqrt{2}) -  \sqrt{\frac{2}{\pi}}  \mathsf{D}^{(K)}_{\infty}  \bigg| \right] \overset{t \to \infty}{\longrightarrow} 0.
\end{equation}
On the event
$\{ \mathsf{X}_{u}(t) \leq \sqrt{2}t + K , \forall \ t \geq 0 , u \in \mathsf{N}_t\}$, which has high probability when $K$ is large (see \eqref{eq-upper-envelope-0}), we have $\mathsf{W}^{(K), F_{t}}_{ t}(\sqrt{2})= \mathsf{W}^{F_{t}}_{ t}(\sqrt{2})$ and $\mathsf{D}^{(K)}_{\infty}=\mathsf{Z}_{\infty}$.
Therefore, $ \frac{\sqrt{t}}{\langle F_{t}, \mu \rangle} \mathsf{W}^{F_{t}}_{t}(\sqrt{2}) $ converges to $ \sqrt{\frac{2}{\pi}} \mathsf{Z}_{\infty}$ in probability  $\mathsf{P}$  as $t\to\infty$.

(ii) We now prove \eqref{conv-prop2}. Note that for $F_{t,\lambda}=G((z-\lambda) t^{1/4})$,
 $\langle F_{t,\lambda},\mu \rangle \sim C_{\lambda} t^{-1/4}$
uniformly in $\lambda \in I$,  where
$C_{\lambda}=\lambda e^{-\frac{\lambda^2}{2}} \int G(y) \dif y$, thus
\begin{equation*}
  \mathsf{W}^{(K),F_{t,\lambda}}_{t-\lambda \sqrt{t}}(\sqrt{2})=\mathsf{W}^{(K),\widetilde F_{\tilde t,\lambda}}_{\tilde t }(\sqrt{2})
\end{equation*}
with $\tilde t=t-\lambda \sqrt{t}$, $ \widetilde F_{\tilde t,\lambda}=G((z-\lambda)\tilde h_{\tilde t})$, and $\tilde h_{\tilde t}=t^{-1/4}$. Using \eqref{eq-exp-to-0} with $t$ replaced by $\tilde t$ and  $h_t$ replaced by
$\tilde h_{\tilde t}$, we get
\begin{align*}
  \mathsf E  \bigg|
  \int_{I}  t^{3/4}  \mathsf{W}^{(K),F_{t,\lambda}}_{t-\lambda \sqrt{t}}(\sqrt{2})  -  \sqrt{\frac{2}{\pi}}  C_{\lambda} \mathsf{D}^{(K)}_{\infty} \dif \lambda    \bigg|   \leq \int_{I} \mathsf E \bigg|  t^{3/4}  \mathsf{W}^{(K),F_{t,\lambda}}_{t-\lambda \sqrt{t}}(\sqrt{2}) - \sqrt{\frac{2}{\pi}} C_{\lambda} \mathsf{D}^{(K)}_{\infty}  \bigg|  \dif \lambda  \to 0
\end{align*}
$t \to \infty$, where we used the dominated convergence theorem.
In fact, \eqref{eq-1st-moment-convergence} below implies that $\mathsf{E}\left[  \frac{\sqrt{t}}{\langle F_{t,\lambda}, \mu \rangle} \mathsf{W}^{(K),F_{t,\lambda}}_{t-\lambda \sqrt{t}}(\sqrt{2}) \right] \leq 1$ for all $\lambda \in I$ and for large $t$,  which guarantees the  existence of a dominating function.
 Then the argument below \eqref{eq-exp-to-0} yields that we can  remove the truncation in $\mathsf{W}^{(K),F_{t,\lambda}}_{t-\lambda\sqrt{t}}(\sqrt{2})$ and $\mathsf{D}^{(K)}_{\infty} $ and get  the desired  convergence in probability result.
\end{proof}

  The remaining part in this section is
 devoted to proving
  Proposition \ref{prop-convergence-in-L^2-Q}. The parameter $K$ is always fixed. For simplicity, $\mathsf{W}^{(K),F_{t}}_{t}(\sqrt{2})$ is abbreviated as $\mathsf{W}^{(K),F}_{t}$.
  Recall that $R_K(t)=\sqrt{2}t-\Xi(t)+K$. We set
  \begin{equation}\label{def-U^K_t}
  U^{(K)}_{t} :=  \frac{\sqrt{t}}{\langle F_{t}, \mu \rangle}
  \frac{\mathsf{W}^{(K),F}_{t} }{\mathsf{D}^{(K)}_{t}},
  \end{equation}
  and
  \begin{equation*} \ V^{(K)}_{t} :=  \frac{\sqrt{t}}{\langle F_{t}, \mu \rangle}   \frac{1}{R_{K}(t)} F_{t}\left( \frac{ R_{K}(t)-K}{\sqrt{t}}\right).
 \end{equation*}
 To prove Proposition \ref{prop-convergence-in-L^2-Q}, we only need to prove that
 \begin{equation}\label{limit-U^k} \lim_{t\to\infty}\mathsf E_{\mathsf{Q}^{(K)}} \left| U^{(K)}_{t}- \sqrt{2/\pi}  \right|= 0.\end{equation}

  \begin{lemma}
  For each $t > 0$, we have
  \begin{equation}\label{eq-conditional-expectation}
      U^{(K)}_{t} = \mathsf E_{\mathsf{Q}^{(K)}}\left[  V^{(K)}_{t}  \big| \mathcal{F}_{t} \right].
    \end{equation}
    \end{lemma}

    \begin{proof}
    By \eqref{eq-conditional-probability-spine-is-u}, we have
    \begin{equation*}
      \begin{aligned}
        & \mathsf E_{\mathsf{Q}^{(K)}}\left[  \frac{1}{R_{K}(t)} F_{t}\left(  \frac{ R_{K}(t)-K}{\sqrt{t}}\right) \big| \mathcal{F}_{t} \right] = \sum_{u \in \mathsf{N}_{t}}  \mathsf{Q}^{(K)}[ \xi_{t}=u | \mathcal{F}_{t} ] \frac{F_{t} ( [\sqrt{2}t-\mathsf{X}_{u}(t)]/\sqrt{t} )}{\sqrt{2}t-\mathsf{X}_{u}(t)+K}  \\
  &= \frac{1}{\mathsf{D}^{(K)}_{t}}\sum_{u \in \mathsf{N}_{t}} F_{t}\left(\frac{ \sqrt{2}t-\mathsf{X}_{u}(t)}{\sqrt{t}}  \right)  e^{-\sqrt{2}(\sqrt{2}t-\mathsf{X}_{u}(t))}     1_{\{ \sqrt{2}s- \mathsf{X}_{u}(s) \geq - K, \forall s \leq t\}} =   \frac{\mathsf{W}^{(K),F}_{t} }{\mathsf{D}^{(K)}_{t}}
      \end{aligned}
    \end{equation*}
    as desired.
  \end{proof}

  \begin{lemma}\label{lem-reduce-to-p-moment}
    For $p \in [1,2]$, define
    \begin{equation}\label{eq-p-moment}
      I_{p}(t):=\mathsf E_{\mathsf{Q}^{(K)}}\left[   \left(  U^{(K)}_{t} \right)^{p}  \right].
    \end{equation}
    Then
    \begin{equation}\label{eq-1st-moment-convergence}
     \lim_{t \to \infty} I_{1}(t) =    \lim_{t \to \infty}  \mathsf{E}_{\mathsf{Q}^{(K)}}\left[ V^{(K)}_{t} \right] = \sqrt{\frac{2}{\pi}},
    \end{equation}
     and  $\limsup_{t \to \infty} \,   (r_{t}^2 h_{t})^{p-1} I_{p}(t) < \infty$.
    \end{lemma}

  \begin{proof}
  Noticing that
  $(R_{K}(t),\mathsf{Q}^{(K)})$ is a Bessel-$3$ process starting from $K$, by \eqref{eq-conditional-expectation}, we have
  \begin{equation}\label{eq-1st-moment-1}
    I_{1}(t) = \mathsf E_{\mathsf{Q}^{(K)}} \left(   V^{(K)}_{t} \right)
  =
  \mathbf{E}^{\mathrm{Bes}}_{K} \left(   \frac{\sqrt{t}}{\langle F_{t}, \mu \rangle}  \frac{1}{R_{t} } F_{t}\left( \frac{R_{t}-K}{\sqrt{t}} \right)  \right)\,.
  \end{equation}
 For $p\in(1,2]$, since $x \mapsto x^{p}$ is convex, by \eqref{eq-conditional-expectation} and Jensen's inequality,
  \begin{equation}\label{eq-2nd-moment-2}
   I_{p}(t) \leq  \mathsf E_{\mathsf{Q}^{(K)}} \left( \left[  V^{(K)}_{t} \right]^{p} \right) =
    \mathbf{E}^{\mathrm{Bes}}_{K} \left( \left[  \frac{\sqrt{t}}{\langle F_{t}, \mu \rangle}  \frac{1}{R_{t} }  F_{t}\left( \frac{R_{t}-K}{\sqrt{t}} \right) \right]^{p} \right).
  \end{equation}
 Recall that  $F_{t}(z)=G([z-r_{t}]h_{t}^{-1})$  and
 our assumption that
 $r_{t}+y h_{t}= \Theta(r_{t})$ uniformly for $y \in \mathrm{supp}(G)$, simple computations yield that
  \begin{equation}\label{eq-order-of-Ft-mu}
  \langle F_{t},\mu \rangle = \int_{0}^{\infty}  G\left(\frac{z-r_{t}}{h_{t}}\right) z e^{-\frac{z^2}{2}} \dif  z  =  h_{t}   \int_{\mathbb{R}} G(y) (r_{t}+yh_{t})e^{-\frac{(r_{t}+yh_{t})^2}{2} }\dif y = \Theta(r_{t}h_{t}).
   \end{equation}
  We claim that, for $p \in [1, 2]$ and any $\epsilon_0>0$,
  \begin{equation}\label{eq-Bessel-moment}
  \mathbf{E}^{\mathrm{Bes}}_{u} \left( \left[  \frac{\sqrt{t}}{R_{t} }  F_{t}\left( \frac{R_{t}-K}{\sqrt{t}} \right) \right]^{p} \right)
  \sim   \sqrt{\frac{2}{\pi}} \, h_{t}
   \int_{\mathbb{R}} G(y)^{p}  (r_{t}+yh_{t})^{2-p}e^{-\frac{(r_{t}+yh_{t})^2}{2} }   \dif y = \Theta(r_{t}^{2-p}h_{t})
  \end{equation}
  uniformly for  $u \in (0,t^{\frac{1}{2}-\epsilon_0}]$.
  Then taking  $p=1$ in \eqref{eq-Bessel-moment} and combining  \eqref{eq-1st-moment-1}, \eqref{eq-order-of-Ft-mu}, we get \eqref{eq-1st-moment-convergence}.  Taking $p \in (1,2]$ in \eqref{eq-Bessel-moment} and using  \eqref{eq-2nd-moment-2}, \eqref{eq-order-of-Ft-mu}, we get $I(t) \lesssim (r_{t}^2h_{t})^{1-p}$.

  Now we show the claim.  Suppose $G$ is supported on $[-A, A]$ for some constant $A>0$.
  First we recall the density formula for $3$-dimensional Bessel process (see \cite[Section VI.3]{RY99})
  \begin{equation}\label{density-Bessel}
    \mathbf{P}^{\mathrm{Bes}}_{x}(R_{s} \in \dif z)=\sqrt{\frac{2}{\pi}} z^2e^{-\frac{(z-x)^2}{2s}}  \frac{  1-e^{-\frac{2xz}{s}} }{2 x z\sqrt{s}} \dif z, \quad\mbox{for }x, z >0.
  \end{equation}
  Using the scaling property for Bessel process, we have
    \begin{align}
  & \mathbf{E}^{\mathrm{Bes}}_{u} \left( \left[  \frac{\sqrt{t}}{R_{t} }  F_{t}\left( \frac{R_{t}-K}{\sqrt{t}} \right) \right]^{p} \right)  =  \mathbf{E}^{\mathrm{Bes}}_{\frac{u}{\sqrt{t}}} \left[    \frac{1}{R_{1}^{p} } F_{t}\left( R_{1}-K/{\sqrt{t}}\right)^{p} \right] \nonumber \\
  &=   \sqrt{\frac{2}{\pi}}
  \int_{| z-r_{t}-K/\sqrt{t}| \leq A h_{t}} G \left( \frac{  z-r_{t}-K/\sqrt{t}}{h_{t}}  \right)^{p}z^{2-p}
  e^{-\frac{(z-u/\sqrt{t})^2}{2s}} \frac{1-e^{-2zu/\sqrt{t} }}{2z u/\sqrt{t}}   \dif z \nonumber \\
  & =  h_{t} \sqrt{\frac{2}{\pi}} \int_{-A}^{A} G(y)^{p}z_{y}^{2-p}  e^{-\frac{1}{2}(z_{y}-\frac{u}{\sqrt{t}})^2}  \frac{ 1-e^{-2z_{y}{u}/{\sqrt{t}} } }{2z_{y} u/\sqrt{t}} \dif y,  \label{eq-expectation-Bes-0}
  \end{align}
  where in the last equality, we substituted  $z$ by $z_{y}= r_t+K/\sqrt{t}+ y h_{t}$.
  By the dominated convergence theorem, we get that,
  for any $ \epsilon_0>0$,
  \begin{equation*}
  \int_{-A}^{A}  G(y)^{p} z_{y}^{2-p}  e^{-\frac{1}{2}(z_{y}-\frac{u}{\sqrt{t}})^2}  \frac{ 1-e^{-2z_{y}{u}/{\sqrt{t}} } }{2z_{y} u/\sqrt{t}} \dif y  \sim
  \int_{-A}^{A} G(y)^{p}  (r_{t}+yh_{t})^{2-p}e^{-\frac{(r_{t}+yh_{t})^2}{2} }   \dif y
  \end{equation*}
  uniformly for  $u\in (0, t^{\frac{1}{2}-\epsilon_0}]$ as $t \to \infty$,
  and  \eqref{eq-Bessel-moment} follows.
  \end{proof}

  The naive bound  $I_{p}(t) \lesssim (r_{t}^2 h_{t})^{1-p}$
  in  Lemma \ref{lem-reduce-to-p-moment}  is not good enough to apply the second moment method (which needs to show that $\limsup\limits_{t \to \infty} I_{2}(t) \leq (\frac{2}{\pi})$). To overcome this difficulties, we need to further truncate $W^{(K),F}_{t}$, by only counting particles with typical behavior in the summation in \eqref{def-W-K-F-t-sqrt2}. The arguments are borrowed from A\"{i}d\'{e}kon, Shi  \cite{AS14} and Pain \cite{Pain18}. Before giving the typical behavior of the particles  we want, we need to introduce some notation as follows.

  Recall that $\xi_{s}$ represents the individual on the spine at time $s$.
  We write $\mathsf{N}^{\xi_{s}}_{t}$ as the descendants of $\xi_{s}$ at time $t$.
  The points of fission along the spine form a Poisson process with rate $2$,
  which is denoted by $\Lambda(\dif s)$.
  We then define,
   for a interval $[c_{1},c_{2}] \subset [0,t]$,
  \begin{equation*}
    \mathsf{D}_t^{(K), [c_{1},c_{2}]}:= \int_{[c_{1},c_{2}]} \sum_{u \in \mathsf{N}^{\xi_{s}}_{t}} (\sqrt{2} t-\mathsf{X}_{u}(t)+K) e^{-\sqrt{2}(\sqrt{2} t-\mathsf{X}_{u}(t))}   1_{\{ \sqrt{2}s- \mathsf{X}_{u}(s) \geq - K, \,\forall s \leq t\}} \Lambda(\dif s) \,,
    \end{equation*}
    which is the contribution  for the martingale $\mathsf{D}^{(K)}_{t}$ of particles who are descendants of individual on spine at time interval $[c_{1},c_{2}]$.
  Then $\mathsf{D}^{(K)}_{t}=\mathsf{D}^{(K),[0,t]}_{t}= \mathsf{D}^{(K),[0,c_{1})}_{t}+ \mathsf{D}^{(K),[c_{1},t]}_{t}$.
  Similarly, we define
  \begin{equation*}
  W_t^{(K),F,[c_{1},c_{2}]}:= \int_{[c_{1},c_{2}]} \sum_{u \in \mathsf{N}^{\xi_{s}}_{t}}  e^{-\sqrt{2}(\sqrt{2} t-\mathsf{X}_{u}(t))}   F_{t}\left( \frac{\sqrt{2} t- \mathsf{X}_{u}(t)}{\sqrt{t}} \right)   1_{\{ \sqrt{2}s- \mathsf{X}_{u}(s) \geq - K, \forall s \leq t\}} \Lambda(\dif s).
  \end{equation*}

  Put $k_{t}=t^{\gamma}$ for some $\gamma \in (0,1)$, and
  for $0<a<\frac{1}{2}<b<1$ ($\gamma,a,b$ will be chosen later), let
  \begin{align*}
    G^{1}_{t} &:= \left\{ k_{t}^{a} \leq R_K(k_{t}) \leq k_{t}^{b}\right\} \cap \left\{    R_K(s) \geq \log^3 t, s \in [k_{t},t] \right\}  \\
    G^{2}_{t} &:= \left\{  \mathsf{D}^{(K),[k_{t},t]}_{t} \leq e^{-\log^{2}t} \right\}.
    \end{align*}
  We define
  \begin{equation}\label{def-G}
  G_{t}: =    G^{1}_{t} \cap G^{2}_{t}.
  \end{equation}

  Now we describe the road to  \eqref{limit-U^k}. As said before,
   by Lemma \ref{lem-reduce-to-p-moment}, the first moment of $U^{(K)}_{t}$ converges, but it's not trivial to show this is the case for the second moment. A good idea then is to prove that
    $  \mathsf E_{\mathsf{Q}^{(K)}} \left[ U^{(K)}_{t} 1_{G_{t}}  \right] \to \sqrt{2/\pi}$ and   $   \mathsf E_{\mathsf{Q}^{(K)}} \left[ (U^{(K)}_{t})^21_{G_{t}}  \right]\to  2/\pi$ as  $t \to \infty$.
     The first assertion is easy, but the second one is not (at least for the authors). However, using the idea in \cite{AS14} and  thanks to \eqref{eq-conditional-expectation}, we   rewrite $I_{2}(t)$ as
    \begin{equation*}
    I_{2}(t)   =   \mathsf E_{\mathsf{Q}^{(K)}} \left[ U^{(K)}_{t} V^{(K)}_{t}  \right] .
    \end{equation*}
   We expect that the integration of $U^{(K)}_{t} V^{(K)}_{t}  $ on the good event $G_{t}$ has the desired limit $2/\pi$.
   Specifically, we are going to show that
    \begin{equation}\label{eq-bnd-J}
    J(t)  :=   \mathsf E_{\mathsf{Q}^{(K)}} \left[U^{(K)}_{t} V^{(K)}_{t}   1_{G_{t}}  \right]    \overset{t \to \infty}{\longrightarrow}  2 /\pi.
    \end{equation}
     Although this equation does not implies  that  $ \mathsf E_{\mathsf{Q}^{(K)}} \left[ \left(  U^{(K)}_{t}  \right)^{2} 1_{G_{t}}  \right] \to  2 /\pi $, but this is the case if we replace $U^{(K)}_{t}=  \mathsf E_{\mathsf{Q}^{(K)}} \left[ V^{(K)}_{t}   \mid \mathcal{F}_{t} \right]$ by
    \begin{equation*}
      \overline{U}^{(K)}_{t} := \mathsf E_{\mathsf{Q}^{(K)}} \left[ V^{(K)}_{t} 1_{G_{t}}  \mid \mathcal{F}_{t} \right].
    \end{equation*}
    Indeed, observing that $  \overline{U}^{(K)}_{t} \leq U^{(K)}_{t}$,  we have
    \begin{equation}\label{eq-moment-bound-pi}
      \limsup_{t \to \infty} \mathsf E_{\mathsf{Q}^{(K)}} \left[ \left(  \overline{U}^{(K)}_{t}  \right)^{2}  \right]  \leq  \lim_{t \to \infty} \mathsf E_{\mathsf{Q}^{(K)}} \left[    \overline{U}^{(K)}_{t} U^{(K)}_{t} \right]=   \lim_{t \to \infty} \mathsf E_{\mathsf{Q}^{(K)}} \left[      U^{(K)}_{t} V^{(K)}_{t} 1_{G_{t}} \right]= 2 /\pi.
    \end{equation}
    Then we shall show that $      \overline{U}^{(K)}_{t} $ is a good approximation for $U^{(K)}_{t}  $, in the sense that
    \begin{equation}\label{eq-bnd-error}
      \mathsf E_{\mathsf{Q}^{(K)}} \left| U^{(K)}_{t}- \overline{U}^{(K)}_{t}     \right|=  \mathsf E_{\mathsf{Q}^{(K)}} \left( V^{(K)}_{t} 1_{G_{t}^{c}} \right) \overset{t \to \infty}{\longrightarrow}  0.
    \end{equation}
    Firstly this  yields that
     $  \mathsf E_{\mathsf{Q}^{(K)}}\left[    \overline{U}^{(K)}_{t}  \right]  \overset{t \to \infty}{\longrightarrow}  \sqrt{2 / \pi} $ since  $  \mathsf E_{\mathsf{Q}^{(K)}}  \left[  U^{(K)}_{t}  \right]  \overset{t \to \infty}{\longrightarrow}  \sqrt{2 / \pi} $ by Lemma  \ref{lem-reduce-to-p-moment}.
    Combining this with
    \eqref{eq-moment-bound-pi} it follows that $ \overline{U}^{(K)}_{t} $ converges in $L^2(\mathsf{Q}^{(K)})$, and hence in $L^1(\mathsf{Q}^{(K)})$ to $\sqrt{2/\pi}$. Again, using \eqref{eq-bnd-error} we get \eqref{limit-U^k}.
    \begin{equation*}
     \textbf{In conclusion, now it suffices to prove \eqref{eq-bnd-J}   and   \eqref{eq-bnd-error} !}
    \end{equation*}
  Before giving our proof, we state a lemma first:

  \begin{lemma}\label{lem-estimate-good-events}
  Let $k_{t}=t^{\gamma}$ for some $\gamma \in (0,1)$, and let $0<a<\frac{1}{2}<b<1$ such that $k_{t}^{b}\leq t^{\frac{1}{2}-\epsilon_0}$ for some  $\epsilon_0>0$.
  For any fixed $A>0$, the following assertions hold.
  \begin{enumerate}[(i)]
    \item $\lim\limits_{t \to \infty}  \mathsf{Q}^{(K)}(G_{t}) =1 $ and $\lim\limits_{t \to \infty} \inf\limits_{u \in [k_{t}^{a},k_{t}^{b}]}  \mathsf{Q}^{(K)}(G_{t}| R_{K}(k_{t})=u) =1  $.
      \item $\mathsf{Q}^{(K)}\left( \left\{ |R_{K}(t)- r_{t} \sqrt{t}| \leq A \sqrt{t} h_{t} \right\} \cap G_{t}^{c} \right)  \lesssim (r_{t}^2 h_{t}) \log^{3}t\left( (r_{t}\sqrt{t})^{-1} + k_{t}^{-1/2}+ k_{t}^{-3(\frac{1}{2}-a)} \right)$.
  \end{enumerate}
  \end{lemma}

  It not hard to see that we can choose $\gamma, a ,b$ satisfying the condition in Lemma \ref{lem-estimate-good-events}. Now fix these three parameters.
  Based on Lemma \ref{lem-estimate-good-events},
  whose proof is postponed to the end of this section, we now prove   \eqref{eq-bnd-error} and  \eqref{eq-bnd-J}.

    \begin{proof}[Proof of \eqref{eq-bnd-error}]
    By the definition of $
        V^{(K)}_{t}$
        \begin{equation*}
          \mathsf E_{\mathsf{Q}^{(K)}} \left( V^{(K)}_{t} 1_{G_{t}^{c}} \right)
           =  \frac{\sqrt{t}}{\langle F_{t}, \mu \rangle}   \mathsf {E}_{\mathsf{Q}^{(K)}} \left[  \frac{1}{R_{K}(t) } F_{t}\left( \frac{ R_{K}(t)-K}{\sqrt{t}}\right)   1_{G_{t}^{c}}  \right]
        \end{equation*}
       When $F_{t}\left( \frac{ R_{K}(t)-K}{\sqrt{t}}\right) >0$, using the assumption that $r_{t}+yh_{t}=\Theta(r_{t})$  uniformly in $y \in \mathrm{supp}(G)$, we have $   R_{K}(t) = \Theta(r_{t} \sqrt{t})$.
         Since  $F_{t}$ is bounded and  $\sqrt{t} r_{t} \gg \log^{3}t $,
         applying \eqref{eq-order-of-Ft-mu} and  Lemma \ref{lem-estimate-good-events} (ii),
        \begin{align*}
          \mathsf E_{\mathsf{Q}^{(K)}} \left( V^{(K)}_{t} 1_{G_{t}^{c}} \right)  &\lesssim_{\beta,\sigma^2,\bar{r},G}
           \frac{\sqrt{t}}{ \langle F_{t},\mu\rangle}   \frac{1}{ r_{t} \sqrt{t} }    \  \mathsf{Q}^{(K)}\left( \left\{ |R_{K}(t)- r_{t} \sqrt{t}| \leq 2A \sqrt{t} h_{t} \right\} \cap G_{t}^{c} \right)   \\
          & \lesssim    \frac{1}{r_{t}^2 h_{t} }  \ (r_{t}^2 h_{t}) \log^{3}t\left(\frac{1}{r_{t}\sqrt{t}}  + \frac{1}{k_{t}^{1/2}} + \frac{1}{k_{t}^{3/4}}\right) .
        \end{align*}
         Since by our assumption $r_{t}\sqrt{t} \gg \log^{3}(t)$
         and $k_{t}=t^{\gamma}$,
         we get that $ \mathsf E_{\mathsf{Q}^{(K)}} \left( V^{(K)}_{t} 1_{G_{t}^{c}} \right) \overset{t \to \infty}{\longrightarrow}  0$ as desired.
      \end{proof}

       \begin{proof}[Proof of \eqref{eq-bnd-J}]
       Recall that
        \begin{equation*}
          J(t)=   \mathsf E_{\mathsf{Q}^{(K)}} \left[U^{(K)}_{t} V^{(K)}_{t}   1_{G_{t}}  \right] = \left(  \frac{\sqrt{t}}{\langle F_{t}, \mu \rangle}   \right)^{2}  \mathsf E_{\mathsf{Q}^{(K)}} \left[   \frac{\mathsf{W}^{(K),F}_{t} }{\mathsf{D}^{(K)}_{t}} \,  \frac{1}{R_{K}(t)} F_{t}\left( \frac{ R_{K}(t)-K}{\sqrt{t}}\right)  1_{G_{t}}  \right] .
        \end{equation*}
     It suffices to show that $\limsup\limits_{t \to \infty} J(t) \leq 2/\pi$, because similar to   \eqref{eq-moment-bound-pi}, $\liminf\limits_{t \to \infty} J(t) \geq       \liminf\limits_{t \to \infty} \mathsf E_{\mathsf{Q}^{(K)}} \left[ \left(  \overline{U}^{(K)}_{t}  \right)^{2}  \right] $ and
      by   \eqref{eq-bnd-error}  and   Lemma  \ref{lem-reduce-to-p-moment}, $     \liminf\limits_{t \to \infty} \mathsf E_{\mathsf{Q}^{(K)}} \left[ \left(  \overline{U}^{(K)}_{t}  \right)^{2}  \right] \geq 2/\pi $. The rest of the proof is divided into four steps:
         \begin{itemize}
          \item We claim that $  J^{(1)}(t):= \left(  \frac{\sqrt{t}}{\langle F_{t}, \mu \rangle}   \right)^{2}     \mathsf E_{\mathsf{Q}^{(K)}} \left[   \frac{\mathsf{W}^{(K),F,[k_{t},t]}_{t} }{\mathsf{D}^{(K)}_{t}} \,  \frac{1}{R_{K}(t)} F_{t}\left( \frac{ R_{K}(t)-K}{\sqrt{t}}\right)  1_{G_{t}}  \right]  = o(1)$. To see this, noting that
          $F_{t}\left( \frac{\sqrt{2}t-X_{u}(t)}{\sqrt{t}}\right) >0$ only if   $\sqrt{2}t-X_{u}(t)=\Theta(r_{t}\sqrt{t})$, we have  $
            W_t^{(K),F,[c_{1},c_{2}]} \leq   \mathsf{D}_t^{(K), [c_{1},c_{2}]}$.
          So on $G_{t}$, we have $ \mathsf{W}^{(K),F,[k_{t},t]}_{t} \leq \mathsf{D}^{(K),[k_{t},t]}_{t} \leq  e^{-\log^2 t}$, and $\frac{1}{R_{K}(t)} F_{t}( \frac{ R_{K}(t)-K}{\sqrt{t}}) \leq 1$.
          Hence by \eqref{eq-order-of-Ft-mu} and \eqref{eq-DK-change-of-measure},
          $ J^{(1)}(t) \leq (\sqrt{t}r_{t}^{-1}h_{t}^{-1})^{2}  e^{- \log^2 t} =o(1) $ by \eqref{eq-order-of-Ft-mu}.

          \item Let $J^{(2)}(t) = \left(  \frac{\sqrt{t}}{\langle F_{t}, \mu \rangle}   \right)^{2}     \mathsf E_{\mathsf{Q}^{(K)}} \left[    \frac{\mathsf{W}^{(K),F,[0,k_{t})}_{t} }{\mathsf{D}^{(K),[0,k_{t})}_{t}}   \frac{1}{R_{K}(t)} F_{t}\left( \frac{ R_{K}(t)-K}{\sqrt{t}}\right)  1_{G_{t}}  \right] $.
         Using   that $\frac{1 }{\mathsf{D}^{(K)}_{t}} \leq \frac{1 }{\mathsf{D}^{(K),[0,k_{t})}_{t} } $, we have $J(t) \leq J^{(1)}(t)+ J^{(2)}(t)$. Then by   the branching Markov property, we get
           \begin{align*}
            J^{(2)}(t) \leq  \left(  \frac{\sqrt{t}}{\langle F_{t}, \mu \rangle}   \right)^{2}   \mathsf E_{\mathsf{Q}^{(K)}} \left[  \frac{\mathsf{W}^{(K),F,[0,k_{t})}_{t} }{\mathsf{D}^{(K),[0,k_{t})}_{t}}   1_{\{R_{K}(k_{t}) \in [k_{t}^{a},k_{t}^{b}] \} }\right]
            \sup_{u \in [k_{t}^{a},k_{t}^{b}]}
            \mathbf{E}^{\mathrm{Bes}}_{u}
            \left[  \frac{F_{t} (\frac{ R_{t-k_{t}}-K}{\sqrt{t}})}{R_{t-k_{t}}} \right].
           \end{align*}
        By doing the same  computation
        as in
        \eqref{eq-expectation-Bes-0}, we have uniformly for $ u \in [k_{t}^{a},k_{t}^{b}]$, $
            \mathbf{E}^{\mathrm{Bes}}_{u} \left[  \frac{F_{t}( \frac{ R_{t-k_{t}}-K}{\sqrt{t}})}{R_{t-k_{t}}}  \right]
             \sim   \sqrt{\frac{2}{\pi}} \, \frac{ \langle F_{t}, \mu \rangle}{\sqrt{t}}$.
      And this yields that
           \begin{equation}\label{eq-limsupJ(t)-1}
           \limsup_{t \to \infty} J(t)
            \leq  \sqrt{\frac{2}{\pi}} \limsup_{t \to \infty }  \mathsf E_{\mathsf{Q}^{(K)}} \left[    \frac{\sqrt{t}}{\langle F_{t}, \mu \rangle}   \frac{\mathsf{W}^{(K),F,[0,k_{t})}_{t} }{\mathsf{D}^{(K),[0,k_{t})}_{t}}  1_{\{R_{K}(k_{t}) \in [k_{t}^{a},k_{t}^{b}] \} }\right].
           \end{equation}
           \item We claim that
           \begin{equation}\label{eq-limsupJ(t)-2}
            \limsup_{t \to \infty} J(t) \leq  \sqrt{\frac{2}{\pi}} \limsup_{t \to \infty }  \mathsf E_{\mathsf{Q}^{(K)}} \left[    \frac{\sqrt{t}}{\langle F_{t}, \mu \rangle}   \frac{\mathsf{W}^{(K),F,[0,k_{t})}_{t} }{\mathsf{D}^{(K),[0,k_{t})}_{t}}  1_{G_{t}} \right].
           \end{equation}
        To see this, using the branching Markov property   and part (i) of  Lemma \ref{lem-estimate-good-events}, we have
           \begin{align*}
           &\mathsf E_{\mathsf{Q}^{(K)}} \left[    \frac{\mathsf{W}^{(K),F,[0,k_{t})}_{t} }{\mathsf{D}^{(K),[0,k_{t})}_{t}} 1_{G_{t}}\right]   \\
            & \geq \mathsf E_{\mathsf{Q}^{(K)}} \left[    \frac{\mathsf{W}^{(K),F,[0,k_{t})}_{t} }{\mathsf{D}^{(K),[0,k_{t})}_{t}} 1_{\{R_{K}(k_{t}) \in [k_{t}^{a},k_{t}^{b}] \} }\right] \, \inf_{u \in [k_{t}^{a},k_{t}^{b}]}  \mathsf{Q}^{(K)}(G_{t}| R_{K}(k_{t})=u) \\
            & =  (1-o(1))\mathsf E_{\mathsf{Q}^{(K)}} \left[   \frac{\mathsf{W}^{(K),F,[0,k_{t})}_{t} }{\mathsf{D}^{(K),[0,k_{t})}_{t}}  1_{\{R_{K}(k_{t}) \in [k_{t}^{a},k_{t}^{b}] \} }\right] .
           \end{align*}
           Substituting into \eqref{eq-limsupJ(t)-1},  the inequality \eqref{eq-limsupJ(t)-2} follows.

            \item  By  \eqref{eq-limsupJ(t)-2}, we have  $\limsup\limits_{t \to \infty} J(t) \leq \sqrt{\frac{2}{\pi}} \limsup\limits_{t \to \infty}  J^{(3)}(t) $, where $
               J^{(3)}(t):=  \frac{\sqrt{t}} { \langle F_{t}, \mu \rangle} \mathsf E_{\mathsf{Q}^{(K)}} \left[ \frac{\mathsf{W}^{(K),F,[0,k_{t})}_{t} }{\mathsf{D}^{(K),[0,k_{t})}_{t} } 1_{G_{t}}  \right] $.
           So it suffices to show that $J^{(3)}(t)= (1+o(1)) \sqrt{\frac{2}{\pi}}$. By assumption, there is a constant $q>0$ such that $t^{-q} \lesssim h_{t}$ for large $t$. Set $\eta := q+2$. Then note that
        \begin{equation}\label{eq-J-3-1}
       J^{(3)}(t)
        \leq   \frac{\sqrt{t}} { \langle F_{t}, \mu \rangle} \mathsf E_{\mathsf{Q}^{(K)}} \left[ \frac{\mathsf{W}^{(K),F,[0,k_{t})}_{t} }{\mathsf{D}^{(K),[0,k_{t})}_{t} } 1_{G_{t} \cap \{ \mathsf{D}^{(K)}_{t} \geq   t^{-\eta}  \}}  \right] +  \frac{\sqrt{t}} { \langle F_{t}, \mu \rangle} \mathsf{Q}^{(K)}( \mathsf{D}^{(K)}_{t} <  t^{-\eta}   ).
        \end{equation}
        On the one hand,  the Markov  inequality and \eqref{eq-order-of-Ft-mu} yields
        \begin{equation}\label{eq-J-3-2}
         \frac{\sqrt{t}} { \langle F_{t}, \mu \rangle}  \mathsf{Q}^{(K)}( \mathsf{D}^{(K)}_{t} < t^{-\eta}  ) =  \frac{\sqrt{t}} { \langle F_{t}, \mu \rangle} \mathsf{Q}^{(K)}\left(  \frac{1}{\mathsf{D}^{(K)}_{t}}  >  t^{\eta}   \right) \lesssim \frac{\sqrt{t}}{t^{\eta}r_{t}h_{t} } \mathsf{E}_{\mathsf{Q}^{(K)}}\left( \frac{1}{\mathsf{D}^{(K)}_{t}}   \right) = o(1) .
        \end{equation}
        On the other hand,  on $G_{t} \cap \{  \mathsf{D}^{(K)}_{t} \geq  t^{-\eta} \} $, we have $\mathsf{D}^{(K),[k_{t},t]}_{t} \leq e^{-\log^2 t}\leq  \frac{ 1}{t}\mathsf{D}^{(K)}_{t}$. Then $ \mathsf{D}^{(K),[0,k_{t})}_{t} \geq (1-t^{-1}) \mathsf{D}^{(K)}_{t}$ and hence
        \begin{equation} \label{eq-J-3-3}
        \begin{aligned}
        & \frac{\sqrt{t}} { \langle F_{t}, \mu \rangle}  \mathsf E_{\mathsf{Q}^{(K)}}\left[ \frac{\mathsf{W}^{(K),F,[0,k_{t})}_{t} }{\mathsf{D}^{(K),[0,k_{t})}_{t} } 1_{G_{t} \cap \{ \mathsf{D}^{(K)}_{t} \geq  t^{-\eta} \}}  \right]
        \leq \frac{1}{1-\frac{1}{t} }   \mathsf E_{\mathsf{Q}^{(K)}} \left[ \frac{\sqrt{t}} { \langle F_{t}, \mu \rangle}  \frac{\mathsf{W}^{(K),F{\color{red},}[0,k_{t})}_{t} }{\mathsf{D}^{(K)}_{t} }   \right]\\
        & \leq  \frac{1}{1-\frac{1}{t}}   \mathsf E_{\mathsf{Q}^{(K)}} \left[ \frac{\sqrt{t}} { \langle F_{t}, \mu \rangle}  \frac{\mathsf{W}^{(K),F}_{t} }{\mathsf{D}^{(K)}_{t} }   \right]= (1+o(1)) \sqrt{\frac{2}{\pi}},
        \end{aligned}
        \end{equation}
        where the last  equality follows from
        \eqref{eq-1st-moment-convergence}.
        Combining  \eqref{eq-J-3-1},  \eqref{eq-J-3-2} and \eqref{eq-J-3-3},  we finally get $J^{(3)}(t)= (1+o(1)) \sqrt{\frac{2}{\pi}}$.
         \end{itemize}
        We now complete the proof.
        \end{proof}

  In the end, we give the proof of Lemma \ref{lem-estimate-good-events}.

\begin{proof}[Proof of Lemma \ref{lem-estimate-good-events}]
 Firstly, we show that $\lim\limits_{t\to\infty} \mathsf{Q}^{(K)}(G^{1}_{t})=1$. In fact, by the Markov property,
  \begin{equation*}
    \mathsf{Q}^{(K)}(G^{1}_{t}) \geq  \mathbf{P}^{\mathrm{Bes}}_{K}\left( R_{K}(k_{t}) \in [k_{t}^{a}, k_{t}^{b}] \right) \inf_{u \in [k_{t}^{a}, k_{t}^{b}] }\mathbf{P}^{\mathrm{Bes}}_{u}\left(  R(s) \geq \log^3 t , \,\forall s \in [0,t-k_{t}] \right).
  \end{equation*}
  The following  estimates \eqref{eq-Bessel-estimate-1} and \eqref{eq-Bessel-estimate-2}  tell us  that $\lim\limits_{t\to\infty}\mathbf{P}^{\mathrm{Bes}}_{K}\left( R_{K}(k_{t}) \in [k_{t}^{a}, k_{t}^{b}] \right)=1.$ Then  combining  with \eqref{eq-Bessel-estimate-3} below, we get  $\lim\limits_{t\to\infty} \mathbb{Q}^{(K)}(G^{1}_{t})=1$.

  \begin{itemize}
    \item Given some small constant $\epsilon_0$, for fixed $a<\frac{1}{2}<b$, using \eqref{density-Bessel} and the fact  $1-e^{-y} \leq y$ for all $y>0$, we have
    uniformly for  $x\leq s^{\frac{1}{2}-\epsilon_0}$,
  \begin{equation}\label{eq-Bessel-estimate-1}
      \mathbf{P}^{\mathrm{Bes}}_{x}(R_{s} \leq s^a )=\mathbf{P}^{\mathrm{Bes}}_{x/\sqrt{s}}(R_{1} \leq s^{a-1/2}  )  \lesssim  \int_{0}^{s^{a-\frac{1}{2}}} e^{-\frac{(z-x/\sqrt{s})^2}{2}}  z^2   \dif z  \lesssim \Theta( s^{-3(\frac{1}{2}-a)} );
    \end{equation}
    and
    \begin{equation}\label{eq-Bessel-estimate-2}
        \mathbf{P}^{\mathrm{Bes}}_{x}(R_{s} \geq s^b )=\mathbf{P}^{\mathrm{Bes}}_{x/\sqrt{s}}(R_{1} \geq s^{b-1/2}  ) \lesssim \int_{s^{b-\frac{1}{2}}}^{\infty} e^{-\frac{(z-x/\sqrt{s})^2}{2}}  z^2    \dif z = s^{-g_1(s)},
      \end{equation}
      where $g_1$ satisfies that $\lim_{s\to\infty}g_1(s)=\infty.$
  \item For $  u \in[k_{t}^{a}, k_{t}^{b}]$, let $\tau=\inf\{s>0: R_{s} < \log^3 t \}$.
  By  the hitting probability of  $3$-dim Brownian motion
  (see e.g. \cite[(9.1.5)]{Durrett19}), we have
  \begin{equation}\label{eq-Bessel-estimate-3}
    \begin{aligned}
  & \mathbf{P}^{\mathrm{Bes}}_{u}\left(  R(s) \geq \log^3 t , \forall s \in [0,t] \right)    \geq  \mathbf{P}^{\mathrm{Bes}}_{u}(\tau = \infty   ) = 1- \frac{\log^3 t}{u}.
    \end{aligned}
  \end{equation}
  \end{itemize}

  Secondly, we show that
   $\mathsf{Q}^{(K)}( G^{1}_{t} \backslash G^{2}_{t}) \to 0$.
  Let $\mathcal{G}_{\infty}$ be  the sigma-algebra generated by
    the spacial motion of the spine $(\Xi(t))_{t\geq 0}$ and  the Poisson point process $\Lambda$  which represents the birth times along the
  spine.
  We know that, under $\mathsf{Q}^{(K)}$, given $\mathcal{G}_{\infty}$,   the processes $\left( \{ \mathsf{X}_{u}(t): u \in \mathsf{N}^{\xi_{s}}_{t}, t \geq s \}, s \in  \mathrm{supp}(\Lambda)\right)$ are independent BBMs starting from $\Xi(s)$. Therefore,
  \begin{equation*}
    \begin{aligned}
    \mathsf{E}_{\mathsf{Q}^{(K)}}[ \mathsf{D}^{(K), [k_t,t]}_{t} | \mathcal{G}_{\infty} ] & \leq   \int_{[k_{t},t]} \mathsf{E}_{\mathsf{Q}^{(K)}} \bigg[ \sum_{u \in \mathsf{N}^{\xi_{s}}_{t}} (\sqrt{2}t-\mathsf{X}_{u}(t)+K) e^{-\sqrt{2}(\sqrt{2}t-\mathsf{X}_{u}(t))}    | \mathcal{G}_{\infty} \bigg] \Lambda(\dif s) \\
    &=\int_{[k_{t},t]}(\sqrt{2} s-\Xi(s)+K) e^{-\sqrt{2}(\sqrt{2}s-\Xi(s))} \Lambda(\dif s).
    \end{aligned}
    \end{equation*}
  By the definition of $G^{1}_{t}$, we have
  $\mathsf{E}_{\mathsf{Q}^{(K)}}[ \mathsf{D}^{(K), [k_t,t]}_{t} | \mathcal{G}_{\infty} ] 1_{G^{1}_{t}} \leq     e^{- \sqrt{2} (\log^3 t-K)} \int_{k_{t}}^{t} R_{K}(s) \Lambda(\dif s) $.
  By the Markov inequality, we get
  \begin{equation}\label{eq-G1-G2}
  \begin{aligned}
  &  \mathsf{Q}^{(K)}\left(G^{1}_{t} \cap \{ \mathsf{D}^{(K), F,[k_{t},t]}_{t} \geq e^{-\log^2 t}\} \right)
  \leq \mathsf{E}_{\mathsf{Q}^{(K)}}\left[ e^{\log^2 t}  \mathsf{D}^{(K), F,[k_{t},t]}_{t} 1_{G^{1}_{t}}    \right]\\
  & = e^{\log^2 t} \mathsf{E}_{\mathsf{Q}^{(K)}}\left(   \mathsf{E}_{\mathsf{Q}^{(K)}}[ \mathsf{D}^{(K), F,[k_{t},t]}_{t} | \mathcal{G}_{\infty} ]  1_{G^{1}_{t}}   \right) \lesssim   e^{-\sqrt{2}\log^3 t+\log^2 t}  \mathsf{E}_{\mathsf{Q}^{(K)}}\left[ \int_{k_{t}}^{t} R_{K}(s)   \dif s   \right]=
   t^{-g_2(t)},
  \end{aligned}
  \end{equation}
  where $g_2$ satisfies that $\lim_{t\to\infty}g_2(t)=\infty$, and
   we used the fact that  $\Lambda$ and $R_{K}$ are independent under $\mathsf{Q}^{(K)}$,
  and
  $\mathsf{E}_{\mathsf{Q}^{(K)}}[ \int_{k_{t}}^{t} R_{K}(s)  \Lambda (\dif s)   ] =  2 \mathsf{E}_{\mathsf{Q}^{(K)}}[ \int_{k_{t}}^{t} R_{K}(s)   \dif s   ]$.
  Therefore, $\mathsf{Q}^{(K)}( G^{1}_{t} \backslash G^{2}_{t}) \to 0$, and then  $\lim_{t\to\infty} \mathsf{Q}^{(K)}(G_{t})=1$.

  Thirdly, the same computations as in \eqref{eq-G1-G2}  yields
  \begin{equation*}
  \begin{aligned}
    &\mathsf{Q}^{(K)}( G^{1}_{t} \backslash G^{2}_{t} | R_{K}(k_{t})=u) =   \mathsf{Q}^{(K)}\left(G^{1}_{t}  \cap \{ \mathsf{D}^{(K), F,[k_{t},t]}_{t} \geq e^{-\log^2 t}\}   | R_{K}(k_{t})=u\right) \\
   & \leq  e^{-\sqrt{2}  \log^3 t+\log^2 t}   \mathsf{E}_{\mathsf{Q}^{(K)}}\left[ \int_{k_{t}}^{t} R_{K}(s)   \dif s | R(k_{t})= u   \right]
    \to 0 \,,  \text{ uniformly in } u \leq k_{t}^{b}.
  \end{aligned}
  \end{equation*}
  Also  \eqref{eq-Bessel-estimate-3} implies that  $  \inf\limits_{u \in [k_{t}^{a},k_{t}^{b}]} \mathsf{Q}^{(K)}( G^{1}_{t}| R_{K}(k_{t}) = u)  \to 1$. Thus Lemma \ref{lem-estimate-good-events}(i) follows.

  Finally, we give the upper bound of $\mathsf{Q}^{(K)}\left( \left\{ |R_{K}(t)- r_{t} \sqrt{t}| \leq A \sqrt{t} h_{t} \right\} \backslash G_{t} \right)$.
  Recall that in \eqref{eq-G1-G2} we have shown that $\mathsf{Q}^{(K)}( G^{1}_{t} \backslash G^{2}_{t} ) = t^{-g_2(t)}$ with  $g_2(t) \to \infty$.
  So it suffices to consider
   $\mathsf{Q}^{(K)}\left( \left\{ |R_{K}(t)- r_{t} \sqrt{t}| \leq A \sqrt{t} h_{t} \right\} \backslash G^{1}_{t}  \right) $.

  \begin{itemize}
  \item  The estimates \eqref{eq-Bessel-estimate-1} and  \eqref{eq-Bessel-estimate-2}
  imply that
  \begin{equation}\label{B5}
    \begin{aligned}
  & \mathsf{Q}^{(K)}\left( R_{K}(k_{t})  \notin [k_{t}^{a}, k_{t}^{b}] ,     |R_{K}(t)- r_{t} \sqrt{t}| \leq A \sqrt{t} h_{t}   \right) \\
  & \leq  \mathbf{P}^{\mathrm{Bes}}_{K}\left( R(k_{t})  \leq k_{t}^{a}\right) \sup_{u \leq k_{t}^{a}}\mathbf{P}^{\mathrm{Bes}}_{u}\left(  |R(t)- r_{t} \sqrt{t}| \leq A \sqrt{t} h_{t}  \right) +   \mathbf{P}^{\mathrm{Bes}}_{K}\left( R(k_{t})  \geq k_{t}^{b} \right) \\
  & \lesssim   \mathbf{P}^{\mathrm{Bes}}_{K}\left( R(k_{t})  \leq k_{t}^{a}\right)
  \sup_{x \leq k_{t}^{a}/\sqrt{t}}\mathbf{P}^{\mathrm{Bes}}_{x} \left(  \left|R_{1}- r_{t}\right| \leq A h_{t} \right) +
  k_{t}^{-g_1(k_t)} \lesssim  k_{t}^{-3 (\frac{1}{2}-a)}  (r_{t}^2 h_{t}) .
   \end{aligned}
  \end{equation}

\item
 Note that by the Markov property,
   \begin{equation*}
  \begin{aligned}
      & \mathsf{Q}^{(K)}\left(  \{ R_{K}(k_{t}) \in  [k_{t}^{a},k_{t}^{b}],  |R_{K}(t)- r_{t} \sqrt{t}| \leq A \sqrt{t} h_{t}   \} \backslash G^{1}_{t} \right) \\
     &= \int_{k_{t}^{a}}^{k_{t}^{b}} \mathbf{P}^{\mathrm{Bes}}_{K}(R(k_{t}) \in \dif u ) \mathbf{P}^{\mathrm{Bes}}_{u}\left( \tau <t-k_{t}; |R_{t-k_{t}} -r_{t} \sqrt{t}|\leq A \sqrt{t}h_{t}\right),
  \end{aligned}
  \end{equation*}
 where $\tau=\inf\{s>0: R_{s} < \log^3 t \}$.
It is known that $(\tau,   \mathbf{P}^{\mathrm{Bes}}_{u})$ has a probability density function  (see e.g., \cite[Page 339, 2.0.2]{BS96})
\begin{equation*}
p_{\tau}(s;u)=\frac{\log^3t}{u} \frac{(u-\log^{3} t)}{\sqrt{2 \pi}s^{3/2}} \exp\left\{ - \frac{(u-\log^{3}t)^2}{2s} \right\} \dif s  \lesssim \frac{\log^3t}{s^{3/2}}
  \exp\left\{ - \frac{u^2}{4s}\right\} \dif s.
\end{equation*}
Let $t'=t-k_{t}$. Using the strong Markov property, $\mathbf{P}^{\mathrm{Bes}}_{u}\left(\tau < t'; |R_{t'}- r_{t} \sqrt{t}| \leq A \sqrt{t} h_{t} \right)$ equals
    \begin{equation*}
    \begin{aligned}
    &    \int_{0}^{t'} \mathbf{P}^{\mathrm{Bes}}_{\log^{3}t} \left(  \left|R_{s}- r_{t}\sqrt{t}\right| \leq A h_{t} \sqrt{t}\right)   p_{\tau}(t'-s;u)\dif s \\
    & \lesssim   \int_{0}^{t'} r_{t}^2 \frac{t}{s}  e^{-c r_{t}^2 \frac{t'}{s} } h_{t} \sqrt{\frac{t}{s}}  \, \frac{\log^3t}{(t'-s)^{3/2}}   e^{ - \frac{u^2}{4(t'-s)} } \dif s  \\
    &\lesssim (r_{t}^2 h_{t}) \frac{\log^{3}t}{\sqrt{t}} \int_{0}^{1}
   \frac{1}{[\lambda(1-\lambda)]^{3/2}}
      e^{-c r_{t}^2 \frac{1}{\lambda} }  e^{- \frac{u^2}{4t'} \frac{1}{1-\lambda} }
    \dif \lambda
     \lesssim  r_{t}^2 h_{t} \log^{3}t   \left( \frac{1}{r_{t}\sqrt{t}} + \frac{1}{u} \right),
      \end{aligned}
    \end{equation*}
 where  in the first inequality we used the domination
   $e^{-(r_{t}\sqrt{t}-Ah_{t}\sqrt{t}-\log^{3}t)^2/2s} \leq e^{-cr_{t}^{2}t'/s}$
 with $c>0$ being some  constant
   and in the third inequality we used the domination
    $\int_{0}^{1} \lambda^{-3/2} e^{-cr_{t}^2/\lambda} \dif \lambda  \leq   \int_{0}^{\infty}  \frac{1}{r_{t}\sqrt{\eta}} e^{-c \eta} \dif \eta \lesssim \frac{1}{r_{t}}$. Thus
  \begin{equation}\label{B6}
  \begin{aligned}
    & \mathsf{Q}^{(K)}\left(  \left\{ R_{K}(k_{t}) \in  [k_{t}^{a},k_{t}^{b}],   |R_{K}(t)- r_{t} \sqrt{t}| \leq A\sqrt{t} h_{t}  \right\} \backslash G^{1}_{t} \right) \\
    & \lesssim    r_{t}^2 h_{t} \log^{3}t  \int_{k_{t}^{a}}^{k_{t}^{b}}  \left( \frac{1}{r_{t}\sqrt{t}} + \frac{1}{u} \right)   \mathbf{P}^{\mathrm{Bes}}_{K}(R_{k_{t}} \in \dif u )  \leq  r_{t}^2 h_{t} \log^{3}t \left(\frac{1}{r_{t}\sqrt{t}} +  \frac{1}{k_{t}^{1/2}}\right).
  \end{aligned}
  \end{equation}
  \end{itemize}
  Combining \eqref{B5} and \eqref{B6}, we get Lemma \ref{lem-estimate-good-events}(ii).
  \end{proof}

 \begin{appendix}
 \section*{Appendix: Proof of Lemma \ref{lem-bridge-estimate}}
\begin{proof}
Denote by $\mathrm{Pr}$ the probability in the left hand side of \eqref{eq-bridge-estimate}.
Noting that the Gaussian process $( \sigma B_{r} - \frac{\sigma^2 r}{\sigma^2 s + t-s}[\sigma B_{s}+B_{t}-B_{s}] )_{r \in [0,s]}$ is independent of $\sigma B_{s}+B_{t}-B_{s} $ (checking their covariance), we have
      \begin{align*}
       \mathrm{Pr}
       &= \mathbf{P}\left( \sigma B_{r} - \frac{\sigma^2 r}{\sigma^2 s + t-s}(\sigma B_{s}+B_{t}-B_{s}) \leq vr - \frac{\sigma^2 r}{\sigma^2 s + (t-s)}(\widetilde{m}_{t} + x ) + K\,,\forall r \leq s   \right)  \\
       &= \mathbf{P}\left( \sigma B_{r} - \frac{r}{s}\sigma B_{s}   \leq  \bar{Z} +   vr - \frac{\sigma^2 r}{\sigma^2 s + (t-s)}(\widetilde{m}_{t} + x ) + K\,,\forall r \leq s   \right) \,,
      \end{align*}
    where $ \bar{Z} := \frac{\sigma^2 r}{\sigma^2 s + t-s}(\sigma B_{s}+B_{t}-B_{s})  - \frac{r}{s}\sigma B_{s}$.
      Simple computation yields
      \begin{equation*}
        vr - \frac{\sigma^2 r}{\sigma^2 s + (t-s)}(\widetilde{m}_{t} + x ) =   \frac{  (1-\sigma^2)(t-s)v + \sigma^2(w_{t}-x)}{\sigma^2 s + (t-s)}  r .
      \end{equation*}
        Let $Z :=   \frac{\sigma^2 s+(t-s)}{r} \bar{Z} = \sigma^2(B_{t}-B_{s}) -\frac{t-s}{s} \sigma B_{s} $. We now have
      \begin{align*}
      \mathrm{Pr}=\mathbf{P}\left( \sigma B_{r} - \frac{r}{s}\sigma B_{s}\leq \frac{Z+ (1-\sigma^2)(t-s)v + \sigma^2(w_{t}-x)}{\sigma^2 s + (t-s)}  r + K\,,\forall r \leq s   \right).
      \end{align*}

 We bound the probability that $Z$ is large.
 Observe that $Z$ is a  Gaussian and  $\mathrm{Var}(Z)=\sigma^2(t-s) ( \sigma^2+\frac{t-s}{s}) \leq 2 (t-s) $ as $\sigma^2\leq 1$.
Applying the Gaussian tail bound,  we have for  large $t$,   $ \mathbf{P}( |Z|> 10  (t-s + \log t) )   \leq \frac{1}{t^2}$  for all $s \in [t-\sqrt{t}\log t, t]$.
      Moreover, $Z$ is  independent of the Brownian bridge $(B_{r}-\frac{r}{s}B_{s})_{r \leq s}$.
   By formula of total probability with partition  $\{ Z\leq  10  (t-s + \log t) \}$ and its complement, using again $\sigma \leq 1$, we have
    \begin{align*}
      \mathrm{Pr}
         \leq &\mathbf{P}\left( \sigma B_{r} - \frac{r}{s}\sigma B_{s}\leq \frac{  (10+v) (t-s    +   w_{t}+|x|)}{\sigma^2 s + (t-s)}  r + K\,,\forall r \leq s   \right) + \frac{1}{t^2}\\
            \leq & 3 \left[\frac{  (10+v) (t-s    +   w_{t}+|x|)s}{\sigma(\sigma^2 s + (t-s))}  +\frac{K}{\sigma}\right]\frac{K}{\sigma}  + \frac{1}{t^2} \lesssim_{K,\beta,\sigma}  \frac{t-s + w_{t}  +|x|}{t},
  \end{align*}
  where in   the second inequality we used Lemma \ref{lem-bridge-estimate-0}.
   \end{proof}
 \end{appendix}
%
%

\begin{acks}[Acknowledgments]
 We thank Bastien Mallein for helpful discussions, Renming Song for valuable suggestions, Haojie Hou and Fan Yang for their careful reading of the manuscript.
 We also thank
 the referees for their valuable comments and corrections, as well as for suggesting a simpler proof of Proposition \ref{lem-functional-convergence-derivative-martingale}. 
\end{acks}
\begin{funding}
%
The research of this project was supported  by the National Key R\&D Program of China (No. 2020YFA0712900).
The second author was supported by NSFC (Grant Nos. 12071011 and 12231002) and  the Fundamental Research Funds for Central Universities, Peking University LMEQF.
\end{funding}



\begin{thebibliography}{10}

    \bibitem{Acosta14}
    J.~Acosta.
    \newblock Tightness of the recentered maximum of log-correlated {Gaussian}
      fields.
    \newblock {\em Electron. J. Probab.}, 19, no. 90, 2014.


    \bibitem{Aidekon13}
    E. A{\"{\i}}d{\'e}kon.
    \newblock Convergence in law of the minimum of a branching random walk.
    \newblock {\em Ann. Probab.}, 41(3A):1362--1426, 2013.

    \bibitem{ABBS13}
    E. A{\"\i}d{\'e}kon, J.~Berestycki, {\'E}.~Brunet, and Z.~Shi.
    \newblock Branching {B}rownian motion seen from its tip.
    \newblock {\em Probab. Theory Relat. Fields}, 157(1):405--451, 2013.

    \bibitem{AS14}
    E. Aidekon and Z.~Shi.
    \newblock The {Seneta}-{Heyde} scaling for the branching random walk.
    \newblock {\em Ann. Probab.}, 42(3):959--993, 2014.

    \bibitem{Arguin16}
    L.-P. Arguin.
    \newblock {\em Extrema of log-correlated random variables principles and examples.}
    \newblock {\em Advances in disordered systems, random processes and some applications},  pages 166--204, Cambridge Univ. Press, Cambridge, 2017.


    \bibitem{ABBRS19}
    L.-P. Arguin, D.~Belius, P.~Bourgade, M.~Radziwi{\l}{\l}, and K.~Soundararajan.
    \newblock Maximum of the Riemann zeta function on a short interval of the
      critical line.
    \newblock {\em Comm.  Pure  Appl. Math.}, 72(3):500--535,
      2019.

    \bibitem{ABH17}
    L.-P. Arguin, D.~Belius, and A.~J. Harper.
    \newblock Maxima of a randomized {R}iemann zeta function, and branching random walks.
    \newblock {\em Ann. Appl. Probab.}, 27(1):178--215, 2017.

    \bibitem{ABK12}
    L.-P. Arguin, A.~Bovier, and N.~Kistler.
    \newblock Poissonian statistics in the extremal process of branching {Brownian}
      motion.
    \newblock {\em Ann. Appl. Probab.}, 22(4):1693--1711, 2012.

    \bibitem{ABK13}
    L.-P. Arguin, A.~Bovier, and N.~Kistler.
    \newblock The extremal process of branching {B}rownian motion.
    \newblock {\em Probab. Theory Relat. Fields}, 157(3):535--574, 2013.

    \bibitem{ADH21}
    L.-P. Arguin, G.~Dubach, and L.~Hartung.
    \newblock Maxima of a random model of the {R}iemann zeta function over intervals
      of varying length.
    \newblock {\em  arXiv:2103.04817}, 2021.

    \bibitem{BK22}
    E. C. Bailey and J.~P. Keating.
    \newblock Maxima of log-correlated fields: some recent developments.
    \newblock {\em J. Phys. A.}, 55(5), no. 053001, 2022.



    \bibitem{Belloum22}
    M. A. Belloum.
    \newblock The extremal process of a cascading family of branching brownian
      motion.
    \newblock {\em arXiv:2202.01584}, 2022.

    \bibitem{BM21}
    M. A. Belloum and B.~Mallein.
    \newblock {Anomalous spreading in reducible multitype branching Brownian
      motion}.
    \newblock {\em Electron. J. Probab.}, 26, no. 39, 2021.



    \bibitem{BBCM22}
    J. Berestycki, {\'E}.~Brunet, A.~Cortines, and B.~Mallein.
    \newblock {A simple backward construction of branching Brownian motion with
      large displacement and applications}.
    \newblock {\em  Ann. Inst. Henri Poincar{\'e}, Probab. Stat.}, 58(4):2094--2113, 2022.

    \bibitem{Biggins92}
    J. D. Biggins.
    \newblock Uniform convergence of martingales in the branching random walk.
    \newblock {\em Ann. Probab.}, 20(1):137--151, 1992.

    \bibitem{Biggins10}
    J. D. Biggins.
    \newblock Branching out.
    \newblock In {\em Probability and mathematical genetics. Papers in honour of
      Sir John Kingman}, pages 113--134. Cambridge: Cambridge University Press,
      2010.

    \bibitem{Biggins12}
    J. D. Biggins.
    \newblock {Spreading speeds in reducible multitype branching random walk}.
    \newblock {\em Ann. Appl. Probab.}, 22(5):1778--1821, 2012

    \bibitem{BL16}
    M. Biskup and O.~Louidor.
    \newblock Extreme local extrema of two-dimensional discrete {Gaussian} free
      field.
    \newblock {\em Commun. Math. Phys.}, 345(1):271--304, 2016.

    \bibitem{BL18}
    M. Biskup and O. Louidor.
    \newblock Full extremal process, cluster law and freezing for the
      two-dimensional discrete {G}aussian free field.
    \newblock {\em Adv. Math.}, 330:589--687, 2018.

    \bibitem{BS96}
    A. Borodin and P. Salminen.
    \newblock {\em Handbook of Brownian motion--facts and formulae}.
    \newblock Probab. Appl.  Birkh{\"a}user, Basel, 1996.

    \bibitem{BH14}
    A. Bovier and L. Hartung.
    \newblock The extremal process of two-speed branching {Brownian} motion.
    \newblock {\em Electron. J. Probab.}, 19, no. 18, 2014.


    \bibitem{BH15}
    A. Bovier and L. Hartung.
    \newblock Variable speed branching {Brownian} motion. {I}: {Extremal} processes
      in the weak correlation regime.
    \newblock {\em ALEA, Lat. Am. J. Probab. Math. Stat.}, 12(1):261--291, 2015.


\bibitem{BH17}
    A. Bovier and L. Hartung.
    \newblock Extended convergence of the extremal process of branching Brownian motion.
    \newblock {\em Ann. Appl. Probab.}, 27(3):1756--1777, 2017.




    \bibitem{BH20}
    A. Bovier and L. Hartung.
    \newblock From 1 to 6 : a finer analysis of perturbed branching {Brownian}
      motion.
    \newblock {\em Commun. Pure Appl. Math.}, 73(7):1490--1525, 2020.

    \bibitem{BI04a}
    A. Bovier and I. Kurkova.
    \newblock Derrida's generalised random energy models 1 : models with finitely
      many hierarchies.
    \newblock {\em Ann. Inst. Henri Poincar{\'e}, Probab. Stat.},
    40(4):439--480, 2004.


    \bibitem{BI04b}
    A. Bovier and I. Kurkova.
    \newblock Derrida's generalized random energy models 2 : models with continuous
      hierarchies.
    \newblock {\em Ann. Inst. Henri Poincar{\'e}, Probab. Stat.},
      40(4):481--495, 2004.

    \bibitem{Bramson78}
    M. Bramson.
    \newblock Maximal displacement of branching {B}rownian motion.
    \newblock {\em Comm.  Pure  Appl. Math.}, 31(5):531--581, 1978.

    \bibitem{Bramson83}
    M. Bramson.
    \newblock {\em Convergence of solutions of the {Kolmogorov} equation to
      travelling waves}, volume 285 of {\em Mem. Am. Math. Soc.}
    \newblock Providence, RI: American Mathematical Society (AMS), 1983.

    \bibitem{BDZ16}
    M. Bramson, J. Ding, and O. Zeitouni.
    \newblock Convergence in law of the maximum of the two-dimensional discrete
      {G}aussian free field.
    \newblock {\em Comm. Pure Appl. Math.}, 69(1):62--123, 2016.


    \bibitem{BZ12}
    M. Bramson and O. Zeitouni.
    \newblock Tightness of the recentered maximum of the two-dimensional discrete
      {Gaussian} free field.
    \newblock {\em Comm. Pure Appl. Math.}, 65(1):1--20, 2012.

     \bibitem{BD11}
    \'{E}. Brunet and B. Derrida.
    \newblock  A branching random walk seen from the tip.
    \newblock{\em J Stat Phys} 143, 420–446, 2011 

    
    \bibitem{CHL19}
    A. Cortines, L. Hartung and O.Louidor.
    \newblock The structure of extreme level sets in branching {Brownian} motion
    \newblock {\em Ann. Probab.}, 47(4):2257 -- 2302, 2019. 

    \bibitem{DRZ17}
    J. Ding, R. Roy, and O. Zeitouni.
    \newblock Convergence of the centered maximum of log-correlated {Gaussian}
      fields.
    \newblock {\em Ann. Probab.}, 45(6A):3886--3928, 2017.


    \bibitem{Durrett19}
    R. Durrett.
    \newblock {\em Probability--theory and examples.}, 5th edition edition, {\em Camb.
      Ser. Stat. Probab. Math.}, 49,
    \newblock Cambridge University Press, Cambridge, 2019.

    \bibitem{FZ12b}
    M. Fang and O. Zeitouni.
    \newblock Branching random walks in time inhomogeneous environments.
    \newblock {\em Electron. J. Probab.}, 17, no. 67, 2012.


    \bibitem{FZ12}
    M. Fang and O. Zeitouni.
    \newblock Slowdown for time inhomogeneous branching {Brownian} motion.
    \newblock {\em J. Stat. Phys.}, 149(1):1--9, 2012.

    \bibitem{Holzer14}
    M. Holzer.
    \newblock Anomalous spreading in a system of coupled {Fisher}-{KPP} equations.
    \newblock {\em Phys. D}, 270:1--10, 2014.

    \bibitem{Holzer16}
    M. Holzer.
    \newblock A proof of anomalous invasion speeds in a system of coupled
      {Fisher-KPP} equations.
    \newblock {\em Discrete Conin.  Dyn. Syst.},
     36(4):2069--2084,
      2016.

    \bibitem{HRS23}
    H. Hou, Y.-X. Ren, and R. Song.
    \newblock Extremal process for irreducible multitype branching {B}rownian motion, \newblock {\em  arXiv:2303.12256},  2023.

    \bibitem{KS15}
    N. Kistler and M. Schmidt.
    \newblock {From Derrida's random energy model to branching random walks: from 1
      to 3}.
    \newblock {\em Electron. Comm. Probab.}, 20:1--12, 2015.



    \bibitem{Kyprianou03}
    A. E. Kyprianou.
    \newblock Travelling wave solutions to the {K}-{P}-{P} equation: alternatives
      to {Simon} {Harris}' probabilistic analysis.
    \newblock {\em Ann. Inst. Henri Poincar{\'e}, Probab. Stat.}, 40(1):53--72,
      2004.


      \bibitem{MR23}
      H. Ma. and Y.-X. Ren.
      \newblock Double jump in the maximum of two-type reducible branching Brownian motion.
      \newblock {\em arXiv:2305.09988v2}, 2023.

  
      \bibitem{MR24}
      H. Ma. and Y.-X. Ren.
      \newblock From 0 to 3: Intermediate phases between normal and anomalous spreading of two-type branching Brownian motion.
      \newblock {\em arXiv:2312.13595v2}, 2024.
 


    \bibitem{LS87}
    S. P. Lalley and T. Sellke.
    \newblock {A conditional limit theorem for the frontier of a branching Brownian
      motion}.
    \newblock {\em  Ann.  Probab.}, 15(3):1052 -- 1061, 1987.

    \bibitem{Madaule15}
    T. Madaule.
    \newblock Maximum of a log-correlated {Gaussian} field.
    \newblock {\em Ann. Inst. Henri Poincar{\'e}, Probab. Stat.}, 51(4):1369--1431,
      2015.

    \bibitem{Madaule16}
    T. Madaule.
    \newblock First order transition for the branching random walk at the critical
      parameter.
    \newblock {\em Stochastic Process. Appl.}, 126(2):470--502, 2016.

    \bibitem{Madaule17}
    T. Madaule.
    \newblock Convergence in law for the branching random walk seen from its tip.
    \newblock {\em J. Theor. Probab.}, 30(1):27--63, 2017.

    \bibitem{MZ16}
    P. Maillard and O.~Zeitouni.
    \newblock Slowdown in branching {Brownian} motion with inhomogeneous variance.
    \newblock {\em Ann. Inst. Henri Poincar{\'e}, Probab. Stat.}, 52(3):1144--1160,
      2016.

    \bibitem{Mallein15b}
      B. Mallein.
      \newblock Maximal displacement in a branching random walk through interfaces.
      \newblock {\em Electron. J. Probab.}, 20, no. 68, 2015.


    \bibitem{Mallein15}
    B. Mallein.
    \newblock {Maximal displacement in the $d$-dimensional branching Brownian
      motion}.
    \newblock {\em Electron. Commun. Probab.}, 20, no. 76, 2015.


    \bibitem{Pain18}
    M. Pain.
    \newblock {The near-critical Gibbs measure of the branching random walk}.
    \newblock {\em Ann. Inst. Henri Poincar{\'e}, Probab. Stat.}, 54(3):1622--1666, 2018.

    \bibitem{RYS14}
    Y.-X. Ren and T. Yang.
    \newblock Multitype branching {Brownian} motion and traveling waves.
    \newblock {\em Adv. Appl. Probab.}, 46(1):217--240, 2014.


    \bibitem{RY99}
    D. Revuz and M. Yor.
    \newblock Continuous martingales and Brownian motion.
    \newblock {\em
   Grundlehren Math. Wiss., 293[Fundamental Principles of Mathematical Sciences]},
    Springer-Verlag, Berlin, 1999.
    \end{thebibliography}


\end{document}